\documentclass{amsart}

\usepackage{amssymb,latexsym}

\usepackage{ytableau} 
\usepackage{xcolor}
\usepackage[bookmarks=false]{hyperref}
\usepackage{tikz,tikz-cd}
\usetikzlibrary{arrows.meta,arrows}

\newcommand{\af}{\mathrm{af}}
\newcommand{\al}{\alpha}
\newcommand{\alb}{{\overline{\alpha}}}
\newcommand{\Aut}{\mathrm{Aut}}

\newcommand{\bv}{\mathbf{v}}
\newcommand{\C}{\mathbb{C}}
\newcommand{\cB}{\mathcal{B}}
\newcommand{\cC}{\mathcal{C}}
\newcommand{\cF}{\mathcal{F}}
\newcommand{\cK}{\mathcal{K}}
\newcommand{\cM}{\mathcal{M}}

\newcommand{\charge}{\mathrm{charge}}
\newcommand{\cl}{\mathrm{cl}}
\newcommand{\cO}{\mathcal{O}}
\newcommand{\core}{\mathrm{core}}
\newcommand{\coremap}{\mathrm{Core}}
\newcommand{\cP}{\mathcal{P}}
\newcommand{\cV}{\mathcal{V}}
\newcommand{\diag}{\mathrm{diag}}
\newcommand{\Dh}{\hat{D}}
\newcommand{\dom}{\trianglerighteq}
\newcommand{\domop}{\dom^{\mathrm{opp}}}

\newcommand{\End}{\mathrm{End}}

\newcommand{\F}{\mathbb{F}}
\newcommand{\fM}{\mathfrak{M}}
\newcommand{\hb}{\overline{h}}
\newcommand{\hD}{{\hat{D}}}
\newcommand{\hH}{{\hat{H}}}
\newcommand{\Hilb}{\mathrm{Hilb}}
\newcommand{\id}{\mathrm{id}}
\newcommand{\inv}{\mathrm{inv}}
\newcommand{\K}{\mathbb{K}}
\newcommand{\kb}{\overline{\kappa}}
\newcommand{\tK}{\tilde{K}}
\newcommand{\la}{\lambda}
\newcommand{\lad}{{\la^\bullet}}
\newcommand{\La}{\Lambda}
\newcommand{\Mat}{\mathrm{Mat}}
\newcommand{\mud}{{\mu^\bullet}}
\newcommand{\negate}{\mathrm{neg}}

\newcommand{\nud}{{\nu^\bullet}}
\newcommand{\ol}[1]{\overline{#1}}
\newcommand{\Q}{\mathbb{Q}}
\newcommand{\quot}{\mathrm{quot}}
\newcommand{\rdom}{\trianglelefteq}
\newcommand{\rdomstrict}{\vartriangleleft}
\newcommand{\reg}{\mathrm{reg}}
\newcommand{\res}[1]{\Big|_{\cV_{{#1}}}}
\newcommand{\resGamma}{\mathrm{res}_\Gamma}
\newcommand{\rev}{\mathrm{rev}}
\newcommand{\rootmap}{\mathrm{Root}}

\newcommand{\sg}{\mathrm{s}}
\newcommand{\shape}{\mathrm{shape}}
\newcommand{\sq}{\square}
\newcommand{\swap}{\mathrm{swap}}
\newcommand{\T}{\mathbb{T}}
\newcommand{\Taut}{\mathrm{Taut}}

\newcommand{\tD}{\tilde{D}}

\newcommand{\tH}{\tilde{H}}
\newcommand{\tHopp}{{}^{\mathrm{opp}}\tH}
\newcommand{\triv}{\mathrm{triv}}

\newcommand{\vn}{\varnothing}

\newcommand{\Waf}{\hat{\mathfrak{S}}_I}
\newcommand{\Wfin}{\mathfrak{S}_I}
\newcommand{\Xd}{X^\bullet}
\newcommand{\Y}{\mathbb{Y}}
\newcommand{\Yd}{{Y^\bullet}}
\newcommand{\Z}{\mathbb{Z}}

\newcommand{\fsl}{\mathfrak{sl}}
\newcommand{\Utor}{U^{\mathrm{tor}}}
\newcommand{\Uaf}{U^{\mathrm{af}}}
\newcommand{\ad}{\mathrm{ad}}

\newcommand{\gl}{\mathfrak{gl}}
\newcommand{\fS}{\mathfrak{S}}
\newcommand{\loc}{\mathrm{loc}}
\newcommand{\Coh}{\mathrm{Coh}}
\newcommand{\Frac}{\mathrm{Frac}}

\newcommand{\miki}{\varsigma}

\newcommand{\pair}[2]{\langle #1,\,#2\rangle}
\newcommand{\pairqt}[2]{\pair{#1}{#2}_{q,t}}
\newcommand{\pairP}[2]{{\pairqt{#1}{#2}^P}}

\newcommand{\ya}{\ydiagram{1}}
\newcommand{\yb}{\ydiagram{2}}
\newcommand{\yaa}{\ydiagram{1,1}}
\newcommand{\cd}{\cdot}

\newcommand{\cG}{\mathcal{G}}
\newcommand{\Tor}{\mathrm{Tor}}
\newcommand{\ch}{\mathrm{ch}}
\newcommand{\pairTor}[2]{\pair{#1}{#2}_{\Tor}}
\newcommand{\Sym}{\mathrm{Sym}}
\newcommand{\Hom}{\mathrm{Hom}}

\newcommand{\heis}{\mathfrak{H}}

\newcommand{\bra}[1]{\langle #1|}
\newcommand{\ket}[1]{|#1\rangle}

\newcommand{\fixit}[1]{{\color{red}\texttt{*** #1 ***}}}

\newcommand{\comment}[1]{}

\newtheorem{thm}{Theorem}[section]
\newtheorem{lem}[thm]{Lemma}
\newtheorem{prop}[thm]{Proposition}
\newtheorem{cor}[thm]{Corollary}
\newtheorem{conj}[thm]{Conjecture}

\theoremstyle{definition}

\newtheorem{ex}[thm]{Example}

\theoremstyle{remark}
\newtheorem{rem}[thm]{Remark}

\numberwithin{equation}{section}

\begin{document}

\title[Wreath Macdonald polynomials]{Wreath Macdonald polynomials, a survey}


\author{Daniel Orr}
\address{}
\curraddr{}
\email{dorr@math.vt.edu}
\thanks{}

\author{Mark Shimozono}
\address{}
\curraddr{}
\email{mshimo@math.vt.edu}
\thanks{}


\date{\today}

\begin{abstract}

Wreath Macdonald polynomials arise from the geometry of $\Gamma$-fixed loci of Hilbert schemes of points in the plane, where $\Gamma$ is a finite cyclic group of order $r\ge 1$. For $r=1$, they recover the classical (modified) Macdonald symmetric functions through Haiman's geometric realization of these functions. The existence, integrality, and positivity of wreath Macdonald polynomials for $r>1$ was conjectured by Haiman and first proved in work of Bezrukavnikov and Finkelberg by means of an equivalence of derived categories. Despite the power of this approach, a lack of explicit tools providing direct access to wreath Macdonald polynomials---in the spirit of Macdonald's original works---has limited progress in the subject.

A recent result of Wen provides a remarkable set of such tools, packaged in the representation theory of quantum toroidal algebras. In this article, we survey Wen's result along with the basic theory of wreath Macdonald polynomials, including its geometric foundations and the role of bigraded reflection functors in the construction of wreath analogs of the $\nabla$ operator. We also formulate new conjectures on the values of important constants arising in the theory of wreath Macdonald $P$-polynomials. A variety of examples are used to illustrate these objects and constructions throughout the paper.

\end{abstract}

\maketitle

\section{Introduction}

\subsection{Classical Macdonald theory} The Macdonald polynomials $P_\mu$ of \cite[VI]{Mac} were originally conceived as a distinguished family of symmetric orthogonal polynomials in several (or infinitely many) variables and depending on two parameters, $q$ and $t$. They were designed to interpolate, through specialization of the parameters, between the major families of multivariable symmetric orthogonal polynomials arising in representation theories associated with the general linear groups $GL_n$ and symmetric groups $\fS_n$---namely, the Schur, Hall-Littlewood, and Jack polynomials. The existence of the $P_\mu$ was originally derived from the digaonalization of explicit $q$-difference operators \cite{Mac} or, equivalently, the symmetric function vertex operators of \cite{GH}.

Macdonald's famous integrality and positivity conjectures \cite[VI.8]{Mac} drove the study of Macdonald polynomials into deeper territory and inspired many tremendous efforts toward their proof. Among these, Haiman's proof \cite{H:pos} of the Macdonald positivity conjecture established a remarkable connection between Macdonald polynomials and the geometry of Hilbert schemes of points in the plane. Central to Haiman's approach is an equivalence of derived categories of equivariant coherent sheaves
\begin{align}\label{E:Haiman equiv intro}
D^b(\Coh^\T(\Hilb_n))\cong D^b(\Coh^{\T\times\fS_n}(\C^{2n})),
\end{align}
where $\Hilb_n=\Hilb_n(\C^2)$ is the Hilbert scheme of $n$ points in the plane $\C^2$ and $\T=(\C^\times)^2$ is a two-dimensional torus acting naturally on $\Hilb_n$ and on  $\C^{2n}=(\C^2)^n$. At the level of equivariant $K$-theory, the equivalence \eqref{E:Haiman equiv intro} leads to an isomorphism of vector spaces
\begin{align}\label{E:Haiman iso intro}
\Phi: \bigoplus_{n\ge 0} K^{\T}(\Hilb_n)_\loc \overset{\sim}{\longrightarrow} \Lambda,
\end{align}
where $\Lambda$ denotes the $\K=\Q(q,t)$-algebra of symmetric functions in infinitely many variables and $\loc$ stands for localization with respect to the $\T$-action. (This localization is achieved by extending of scalars to the field of fractions $\K$ of the representation ring $K^\T(\mathrm{pt}) \cong R(\T)\cong \Z[q^{\pm 1},t^{\pm 1}]$ of the torus $\T$.)

For each $n\ge 0$ and each partition $\mu$ of size $n$, we have a $\T$-fixed point $I_\mu\in\Hilb_n$ given by a monomial ideal, and the isomorphism $\Phi$ sends the corresponding class $[I_\mu]$ in $K^{\T}(\Hilb_n)_\loc$ to the modified Macdonald symmetric function $\tH_\mu$ (which is a cousin of $P_\mu$). Moreover, the equivalence \eqref{E:Haiman equiv intro} realizes $\tH_\mu$ as the Frobenius character of a bigraded $\fS_n$-module, namely the fiber $P_n|_{I_\mu}$ of a rank $n!$ vector bundle $P_n$ on $\Hilb_n$. This vector bundle $P_n$, which induces \eqref{E:Haiman equiv intro} and has many special properties, is known as the Procesi bundle.

It was later realized that the isomorphism $\Phi$ carries further representation-theoretic significance. In particular, both sides of \eqref{E:Haiman iso intro} admit natural actions of the quantum toroidal $\gl_1$ algebra $\Utor_1$, also known as the elliptic Hall algebra \cite{BS}. The space on the left side of \eqref{E:Haiman iso intro} affords a well-known geometrically-defined action of $\Utor_1$ \cite{SV2,FT}, which is constructed along the lines of earlier work of Nakajima and Grojnowski. The action of $\Utor_1$ on the right-hand side $\Lambda$ is expressed via vertex operators going back to \cite{GH,GT} and arises as limit of polynomial representations of type $GL$ spherical double affine Hecke algebras \cite{SV1}. A result of \cite{SV2,FT} then asserts that the isomorphism $\Phi$ arising from Haiman's theorem interwines these two actions of $\Utor_1$, up to an automorphism of $\Utor_1$ known as the Miki automorphism \cite{Miki2} (or, equivalently, the limit of Cherednik's Fourier transform \cite{C}). 
This perspective on $\Phi$ provided by the methods of \cite{SV1,SV2,FT} has proven to be tremendously fruitful in Macdonald theory; see, for instance, the works \cite{GN} and \cite{BHMPS}, both of which also employ the powerful shuffle realization of $\Utor_1$ \cite{SV2,FT,Neg1}.

The particular developments in Macdonald theory which we have described above will serve as a model for our treatment of wreath Macdonald polynomials. To summarize, these developments proceeded as follows:
\begin{enumerate}
\item \textit{Eigenoperators.} Macdonald polynomials $P_\mu$ (and their variants $\tH_\mu$) are constructed as orthogonal polynomials by means of explicit eigenoperators.
\item \textit{Geometry.} Haiman's theorem on the Hilbert scheme realization of Macdonald polynomials implies the Macdonald positivity conjecture.
\item \textit{Representation theory.} The isomorphism \eqref{E:Haiman iso intro} arising from Haiman's theorem is understood as an intertwining map between two representations of the quantum toroidal $\gl_1$ algebra.
\end{enumerate}

\subsection{Wreath Macdonald theory}
Our aim in this article is to describe developments in the theory of \textit{wreath Macdonald polynomials}. These are a generalization of Macdonald polynomials arising from the geometry of the $\Gamma$-fixed loci $\Hilb_n^\Gamma=\Hilb_n(\C^2)^\Gamma$ for a finite subgroup $\Gamma\subset SL_2(\C)$. Although the original framework for wreath Macdonald theory proposed by Haiman \cite{H} allows for an arbitrary finite subgroup $\Gamma$, we shall assume throughout that $\Gamma\cong\Z/r\Z$ is a finite cyclic group, as essentially all developments in the subject to date are limited to this case. 

While we shall use the developments in classical Macdonald theory described above as a model, it is interesting that wreath Macdonald theory has proceeded in a different order, beginning with geometry. In broad strokes, its development has proceeded to date as follows:
\begin{enumerate}
\item \textit{Geometry.} Following Haiman's proposal of the geometric foundations for wreath Macdonald theory \cite{H}, Bezrukavnikov and Finkelberg \cite{BF} prove that a wreath analog of the equivalence \eqref{E:Haiman equiv intro} leads to a geometric realization of wreath Macdonald polynomials via the analog of the map \eqref{E:Haiman iso intro}. This establishes the existence, integrality, and positivity of wreath Macdonald polynomials, all conjectured by Haiman \cite{H}. This was later established and extended by Losev \cite{L1,L2}.
\item \textit{Representation theory.} Wen \cite{Wen} proves that the wreath analog of the isomorphism \eqref{E:Haiman iso intro} is, up to nontrivial scalars, an intertwining map between two representations of the quantum toroidal $\gl_r$ algebra $\Utor_r$. This involves a twist by the remarkable \textit{Miki automorphism} of $\Utor_r$. Wen's result, which builds upon on earlier work of Negut~\cite{Neg} and Tsymbaliuk~\cite{T}, constitutes an highly nontrivial extension of the result of \cite{SV2,FT} described above.
\item \textit{Eigenoperators.} By virtue of \cite{Wen}, the horizontal Heisenberg subalgebra of $\Utor_r$ acts diagonally on wreath Macdonald polynomials. This leads to explicit vertex operators having the wreath Macdonald polynomials as their joint eigenbasis.
\end{enumerate}

In somewhat more detail, for each element $\beta$ in the root lattice $Q$ of $\fsl_r$, one has a basis of wreath Macdonald polynomials $\{\tH_\mud^{t_{-\beta^\vee}}\}_\mud$ for the $r$-fold tensor power $\Lambda^{\otimes r}$ of symmetric functions. These are indexed arbitrary by $r$-tuples of partitions $\mud=(\mu^{(0)},\dotsc,\mu^{(r-1)})$, with $\tH_\mud^{t_{-\beta^\vee}}$ homogeneous of total degree $|\mud|=\sum_i|\mu^{(i)}|$. The root lattice element $\beta$ corresponds to an $r$-core partition $\gamma$, and once this is fixed we can think of the multipartition $\mud$ as the $r$-quotient of an arbitrary (single) partition of size $N=|\gamma|+r|\mud|$ (see \S\ref{SS:core-quot}).

The wreath analog of \eqref{E:Haiman equiv intro} is an equivalence
\begin{align}\label{E:wreath Haiman equiv intro}
D^b(\Coh^\T(\fM_{\beta,n}))\cong D^b(\Coh^{\T\times\Gamma_n}(\C^{2n}))
\end{align}
studied in \cite{BF} and arising from more general result of Bezrukavnikov and Kaledin \cite{BK}. Here $\Gamma_n=\fS_n \wr \Gamma$ is the wreath product group, $\beta\in Q$ is an element of root lattice of $\fsl_r$, and $\fM_{\beta,n}$ is a Nakajima quiver variety for the cyclic quiver with $r$ vertices and one-dimensional framing space. As we recall in detail below, $\fM_{\beta,n}$ is isomorphic to an irreducible component of $\Hilb_N^\Gamma$ for $N$ as above.

The analog of \eqref{E:Haiman iso intro} is then an isomorphism 
\begin{align}\label{E:wreath Haiman iso intro n}
\Phi_{\beta,n} : K^{\T}(\fM_{\beta,n})_\loc \overset{\sim}{\longrightarrow} (\Lambda^{\otimes r})_n,
\end{align}
where $(\Lambda^{\otimes r})_n$ is the component of $\Lambda^{\otimes r}$ consisting of multisymmetric functions of total degree $n$. Bezrukavnikov and Finkelberg \cite{BF} established that the isomorphism $\Phi_{\beta,n}$ sends the fixed-point basis $\{[I_\mud^{t_{-\beta^\vee}}]\}$ of $K^{\T}(\fM_{\beta,n})_\loc$, where $\mud$ runs over $r$-multipartitions of total size $n$, to the sought-after wreath Macdonald basis $\{\tH_\mud^{t_{-\beta^\vee}}\}$ of $(\Lambda^{\otimes r})_n$ conjectured to exist in \cite{H}. In this way, the wreath Macdonald polynomial $\tH_\mud^{t_{-\beta^\vee}}$ is realized as the bigraded $\Gamma_n$-Frobenius character of the fiber of a rank $r^nn!$ vector bundle (the wreath Procesi bundle) over the corresponding $\T$-fixed point, establishing integrality and positivity of wreath Macdonald-Kostka coefficients. 

Taking the direct sum of all maps $\Phi_{\beta,n}$ from \eqref{E:wreath Haiman iso intro n}, one obtains an isomorphism
\begin{align}\label{E:wreath Haiman iso intro}
\Phi: \bigoplus_{\substack{\beta\in Q\\n\ge 0}} K^{\T}(\fM_{\beta,n})_\loc \overset{\sim}{\longrightarrow} \bigoplus_{\beta\in Q} \Lambda^{\otimes r}
\end{align}
of $\K$-vector spaces. This is where the representation theory of the quantum toroidal algebra $\Utor_r$ enters the picture. In particular, \eqref{E:wreath Haiman iso intro} is the setting for Wen's work \cite{Wen}, which involves a comparison between the geometrically-defined action of $\Utor_r$ on the left-hand side \cite{VV,Nak} and Saito's vertex representation \cite{Saito} on the right-hand side, twisted by the Miki automorphism. The main result of \cite{Wen} can be interpreted as the assertion that the map \eqref{E:wreath Haiman iso intro} arising from the equivalence \eqref{E:wreath Haiman equiv intro} agrees with Tsymbaliuk's isomorphism \cite{T} between two realizations of the Fock representation of $\Utor_r$, up to a rescaling of the fixed-point basis. 

Under the $\Utor_r$ action of \cite{VV,Nak}, it is easy to identify eigenoperators for the fixed-point basis on the left-hand side of \eqref{E:wreath Haiman iso intro}. A key point of Wen's result is that it allows one to transfer these to the right-hand side, thus obtaining eigenoperators for wreath Macdonald polynomials $\{\tH_\mud^{t_{-\beta^\vee}}\}$ in the vertex representation \cite{Saito}. What makes this transfer highly nontrivial and difficult to discover directly is that it must go through the Miki automorphism of $\Utor_r$.

More generally, for any element $w=ut_{-\beta^\vee}\in \fS_r \ltimes Q$ in the affine Weyl group of type $A_{r-1}$, one has a wreath Macdonald basis $\{\tH_\mud^{w}\}$ of $\Lambda^{\otimes r}$ indexed by multipartitions $\mud$ of arbitrary size. There are corresponding equivalences \eqref{E:wreath Haiman equiv intro} and isomorphisms $\Phi^w_n : K^{\T}(\fM_{\beta,n})_\loc \overset{\sim}{\longrightarrow} (\Lambda^{\otimes r})_n$ realizing each $\tH_{\mud}^w$ as the bigraded $\Gamma_n$-Frobenius character of a fiber of a different wreath Procesi bundle (depending on the finite permutation $u$). These more general Procesi bundles were introduced and classifed by Losev \cite{L1,L2}, who gave an independent proof of the main result of \cite{BF}. At present, we do not know if it is possible to characterize the more general wreath Macdonald polynomials $\tH^w_\mud$ for $w=ut_{-\beta^\vee}$ with $u\neq 1$ using quantum toroidal algebras.

In forthcoming work \cite{OS,OSW}, we use Wen's result to extract explicit $q$-difference eigenoperators for the finite variable specializations of wreath Macdonald polynomials, providing a characterization of these polynomials true to the spirit of Macdonald's original approach to the subject.

\subsection{Overview}
We provide a survey of the developments described above, culminating in a statement of Wen's result and a completely concrete description of the vertex eigenoperators it provides for the wreath Macdonald polynomials $\tH^{t_{-\beta^\vee}}_\mud$ indexed by translations. Along the way, and with the intention of making the subject more widely accessible, we provide an explicit and detailed exposition of fundamental properties and symmetries of wreath Macdonald polynomials $\tH^w_\mud$ for arbitrary affine Weyl group elements $w$. We explain the role of bigraded reflection functors in the theory (as explained to us by Haiman \cite{H:private}), and in particular how these lead to wreath analogs of the $\nabla$ operator from classical Macdonald theory. We illustrate these constructions with a variety of examples. As an original contribution, we also formulate new conjectures on important constants arising in the theory of wreath Macdonald $P$-polynomials (Conjectures \ref{CJ:P pairing} and \ref{CJ:J to P coefficient}).

To set the stage, we begin with a review of some major results from classical Macdonald theory which have influenced these developments.

\section*{Acknowledgements}

We thank Mark Haiman for helpful discussions and for sharing his unpublished results on wreath Macdonald polynomials \cite{H:private}. We also thank Joshua Wen for many fruitful discussions and related collaborations. D.O. gratefully acknowledges support from the Simons Foundation (Collaboration Grant for Mathematicians, 638577) and the Max Planck Institute for Mathematics (MPIM Bonn).

\section{Classical Macdonald theory}

\subsection{Partitions}

Let $\Y$ be Young's lattice of partitions. 
The diagram of a partition $\mu=(\mu_1,\mu_2,\dotsc)\in \Y$ is the subset of $\Z_{\ge0}^2$ given by
\begin{align*}
D(\mu)=\{(a,b)\mid 0 \le a < \mu_{b+1} \}.
\end{align*}
Its elements are called cells.
We often write $\mu$ to mean $D(\mu)$. 
For any $\mu\in\Y$, let $\mu^t\in\Y$ be the transposed partition, defined by $(a,b)\in \mu^t$ if and only if $(b,a)\in\mu$. Let $|\mu|=\sum_i \mu_i$. Let $\unrhd$ be the dominance partial order on $\Y$ \cite[\S I.1]{Mac}. Let $\Y_N\subset \Y$ be the set of partions with at most $N$ parts.

For any $\mu\in\Y$, define the following quantities:
\begin{align}\label{E:Bmu}
B_\mu = B_\mu(q,t) &= \sum_{(a,b)\in\mu} q^a t^b \\
\label{E:Amudef}
A_\mu = A_\mu(q,t) &= 1-(1-q)(1-t)B_\mu.
\end{align}
One has
\begin{align*}
A_\mu =\sum_{(a,b)\in A(\mu)} q^a t^b - qt \sum_{(a,b)\in R(\mu)} q^a t^b,
\end{align*}
where $A(\mu)$ is the set of $\mu$-addable cells, the cells $s\in \Z_{\ge0}^2\setminus D(\mu)$ such that $D(\mu)\cup \{s\}$ is the diagram of a partition, and $R(\mu)$ is the set of $\mu$-removable cells, those cells $s\in D(\mu)$ such that $D(\mu)\setminus \{s\}$ is the diagram of a partition.

\begin{ex} For the partition $\mu=(3,2)$
\begin{align*}
D(\mu)&=\{(0,0),(1,0),(2,0),(0,1),(1,1)\}, \\
A(\mu)&=\{(3,0),(2,1),(0,2)\}, \\
R(\mu)&=\{(2,0),(1,1)\}.
\end{align*}
The addable (resp. removable) cells of $\mu$ are depicted with $+$ (resp. $-$) entries.
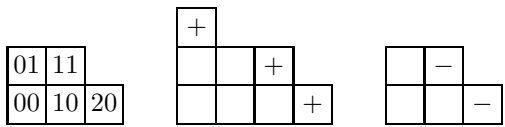
\ytableausetup{aligntableaux=bottom,boxsize=5mm}

\begin{align*}
\begin{ytableau}
01&11\\
00&10&20
\end{ytableau}\qquad
\begin{ytableau}
 +\\
{} & &+ \\
 & & & +\\
\end{ytableau}\qquad
\begin{ytableau}
{}& \text{$-$}\\
{}&{} & \text{$-$}\\
\end{ytableau}
\qquad
\end{align*}
We have
\begin{align*}
B_\mu(q,t)&=1+q+q^2+t+qt\\
A_\mu(q,t)&=q^3+q^2t+t^2-qt(q^2+qt).
\end{align*}
\end{ex}
The arm $a_\mu(s)$ (resp. leg $l_\mu(s)$) of a cell $s=(a,b)\in \mu$ is by definition the number of cells in $\mu$ that are strictly to the right (resp. above) $s$ in its row (resp. column). The hook of $s$ in $\mu$ is given by $h_\mu(s) = 1 + a_\mu(s)+l_\mu(s)$. 

\begin{ex}
For $\mu=(6,4,2,1)$ and $s=(1,1)$ we have $a_\mu(s)=2$, $l_\mu(s)=1$, and $h_\mu(s)=4$.
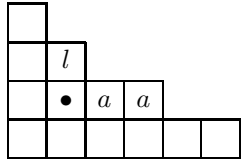
\begin{align*}
\begin{ytableau}
{}\\
{}&{l}\\
{}&{\bullet}&{a}&{a}\\
{}&{}&{}&{}&{}&{}
\end{ytableau}
\end{align*}
\end{ex}

\subsection{Symmetric functions}

For any integer $N\ge 0$, let $\fS_N$ be the symmetric group and let $\Lambda_{N,\Z}=\Z[x_1,\dotsc,x_N]^{\fS_N}$ be the ring of symmetric polynomials in the variables $x_1,\dotsc,x_N$. This is a graded ring: $\Lambda_{N,\Z}=\bigoplus_{n=0}^\infty \Lambda^n_{N,\Z}$ where $\Lambda^n_{N,\Z}\subset \Lambda_{N,\Z}$ consists of symmetric polynomials which are homogeneous of total degree $n$.
 
Let $\Lambda_\Z=\bigoplus_{n=0}^\infty \Lambda^n_\Z$ be the ring of symmetric functions \cite[\S I.2]{Mac}, which is by definition the projective limit of the $\Lambda_{N,\Z}$ in the category of graded rings. In concrete terms, one may regard elements of $\Lambda_\Z$ as formal power series in infinitely many variables $x_1,x_2,\dotsc$ which are symmetric and of bounded degree. For $f\in\Lambda_\Z$, we use the plethystic notation $f[X_N]$, where $X_N=x_1+\dotsc+x_N$, to denote the image of $f$ in $\Lambda_{N,\Z}$.

Let $\{m_\lambda\}_{\lambda\in\Y}$ and $\{s_\lambda\}_{\lambda\in\Y}$ be the bases of monomial and Schur symmetric functions, respectively. We continue to denote the images of $m_\lambda$ and $s_\lambda$ in $\Lambda_{N,\Z}$ by the same letters; for $\lambda\in\Y_N$ we obtain bases of $\Lambda_{N,\Z}$.


For any commutative ring $R$, let $\Lambda_R=\Lambda_\Z\otimes_\Z R$, $\Lambda_R^n=\Lambda_\Z^n\otimes_\Z R$, and similarly define $\Lambda_{N,R}$ and $\Lambda_{N,R}^n$. We have the power sum basis $\{p_\lambda\}_{\lambda\in\Y}$ for $\Lambda_{\Q}$.


Since we will typically work over the field $\K=\Q(q,t)$ of rational functions in indeterminates $q$ and $t$, we reserve the notations $\Lambda$, $\Lambda^n$, $\Lambda_N$, and $\Lambda_N^n$ when working over this base field, i.e., $\Lambda=\Lambda_{\K}$, etc.

For any $A=A(q,t)\in \K^\times$ and $B=B(q,t)\in\K$, let $\psi_{A,B}$ be the  $\K$-algebra automorphism of $\Lambda$ defined by
\begin{align*}
\psi_{A,B}(p_n)=A(q^n,t^n)p_n+B(q^n,t^n)\qquad\text{for $n\in\Z_{>0}$.}
\end{align*}
We also define the plethystic notation $f[AX+B]$ by means of the equality $\psi_{A,B}(f)=f[AX+B]$ for any $f\in\Lambda$, where $X=x_1+x_2+\dotsm$.

The Hall inner product $\langle \cdot,\cdot\rangle$ on $\Lambda_\Z$ (or, by extension of scalars, on any $\Lambda_R$) is defined by $\langle s_\lambda, s_\mu\rangle=\delta_{\lambda\mu}$.
Let $\End(\La)$ denote the $\K$-algebra of $\K$-linear endomorphisms of $\La$.
For $T\in\End(\Lambda)$, let $T^\perp\in\End(\Lambda)$ denote the adjoint of $T$ with respect to the Hall inner product. We often regard $g\in\Lambda$ as the multiplication operator $f\mapsto gf$ on $\Lambda$, and then $g^\perp$ denotes the adjoint of this operator. For instance, we have
\begin{align*}
p_n^\perp = n\frac{\partial}{\partial p_n}\qquad\text{for $n\in\Z_{>0}$.}
\end{align*}

Finally, denoting by $R(\fS_n)$ the representation ring of the symmetric group, let $\mathrm{Frob} : R(\fS_n) \to \Lambda^n_\Z$ be the Frobenius character map \cite[\S I.7]{Mac}, which is a linear isomorphism sending the irreducible character $\chi^\lambda$ of $\fS_n$ for $|\lambda|=n$ to the Schur function $s_\lambda$. For a bigraded $\fS_n$-module $V=\bigoplus_{a,b\in\Z} V_{a,b}$ with each $V_{a,b}$ finite-dimensional, we denote by $\mathrm{Frob}_{q,t}(V)$ the bigraded Frobenius series $\sum_{a,b} \mathrm{Frob}(V_{a,b})q^a t^b$.

\subsection{Macdonald $P$-functions}

The Macdonald symmetric functions $\{P_\mu\}_{\mu\in\Y}$ form an orthogonal basis of $\Lambda$ with respect to Macdonald's $(q,t)$-deformation of the Hall inner product \cite[\S VI]{Mac}. For any $\mu\in\Y$, $P_\mu$ is homogeneous of degree $|\mu|$ and satisfies the unitriangularity (with $\lhd$ meaning $\unlhd$ and not equal):
\begin{align*}
    P_\mu\in m_\mu+\bigoplus_{\la \lhd \mu} \K m_\la.
\end{align*}

The images $P_\mu[X_N]$ for $\mu\in\Y_N$ form a basis for $\Lambda_N$. They can be characterized as the symmetric polynomial eigenfunctions of the Macdonald operator
\begin{align}\label{E:Mac op}
M_N=\sum_{k=1}^N \prod_{\substack{\ell=1\\\ell\neq k}}^N\frac{tx_k-x_\ell}{x_k-x_\ell}T_{q,x_k},
\end{align}
where $T_{q,x_k}f(x_1,\dotsc,x_N) = f(x_1,\dotsc,qx_k,\dotsc,x_N)$, with eigenvalues given by:
\begin{align}\label{E:M-P}
M_N P_\mu[X_N] = \left(\sum_{k=1}^N q^{\mu_k}t^{N-k}\right)P_\mu[X_N].
\end{align}

\begin{ex}
Let $N=2$. The monomial symmetric polynomials $$m_{(0,0)}=1,\quad m_{(1,0)}=x_1+x_2,\quad m_{(1,1)}=x_1x_2,\quad m_{(2,0)}=x_1^2+x_2^2$$
form a basis for the subspace $\Lambda_2^0\oplus\Lambda_2^1\oplus\Lambda_2^2\subset \Lambda_2$.
With respect to this basis, the restriction of the Macdonald operator $M_2$ has the following matrix:
$$ \begin{pmatrix} 1+t & 0 & 0 & 0\\ 0 & 1+qt & 0 & 0\\ 0 & 0 & q(1+t) & (1-q^2)(1-t) \\ 0 & 0 & 0 & 1+q^2t \end{pmatrix}. $$
One finds that
\begin{align*}
P_{(0,0)} &= 1 = m_{(0,0)}\\
P_{(1,0)} &= x_1 + x_2 = m_{(1,0)}\\
P_{(1,1)} &= x_1 x_2 = m_{(1,1)}\\
P_{(2,0)} &= x_1^2+x_2^2 + \frac{(1+q)(1-t)}{1-qt}x_1 x_2 = m_{(2,0)}+\frac{(1+q)(1-t)}{1-qt}m_{(1,1)}.
\end{align*}
\end{ex}

\subsection{Vertex operators}


Let $\Omega$ be the plethystic exponential.\footnote{For an introduction to $\Omega$ in the context of symmetric functions, see \cite[\S 5.3]{LR}.} Thus, in the completion $\widehat{\Lambda} = \prod_{n=0}^\infty \Lambda^n$ of symmetric functions, one has the elements
\begin{align*}
\Omega[X] &= \prod_{i=1}^\infty \frac{1}{1-x_i}
= \sum_{k=0}^\infty h_k
= \exp\left(\sum_{k=1}^\infty \frac{p_k}{k}\right) \\
\Omega[-X] &= \prod_{i=1}^\infty (1-x_i)
= \sum_{k=0}^\infty (-1)^k e_k
= \exp \left(-\sum_{k=1}^\infty \frac{p_k}{k}\right)
\end{align*}
where $e_k=s_{(1^k)}$ and $h_k=s_{(k)}$ are the elementary and complete homogeneous symmetric functions, respectively.
Generally, for plethystic alphabets $X$ and $Y$, one has the identity 
$\Omega[X+Y]=\Omega[X]\Omega[Y]$.

By a symmetric function \textit{vertex operator}, we mean a formal operator series of the form
\begin{align*}
V(z)=\sum_{k=-\infty}^\infty V_k z^k = \Omega[zAX]\Omega[z^{-1}BX]^\perp\in \End(\Lambda)[[z,z^{-1}]]
\end{align*}
for some $A,B\in\K$. Here we regard the series
\begin{align*}
\Omega[zAX] = \sum_{k=0}^\infty h_k[AX]z^k = \exp\left(\sum_{k=1}^\infty \frac{p_k[AX]}{k}z^k\right)
\end{align*}
as an element of $\End(\Lambda)[[z]]$ by viewing each symmetric function $h_k[AX]$ (or $p_k[AX]$) as a multiplication operator. For the other ``half'' of $V(z)$, we have the operator series
\begin{align*}
\Omega[z^{-1}BX]^\perp 
= \sum_{k=0}^\infty h_k[BX]^\perp z^{-k}
= \exp\left(\sum_{k=1}^\infty \frac{p_k[BX]^\perp}{k}z^{-k}\right)\in \End(\Lambda)[[z^{-1}]].
\end{align*}
Formally multiplying these two operator series, one obtains the following formula for the coefficients of $V(z)$:
\begin{align*}
V_k = \sum_{\substack{\ell,m\ge 0\\\ell-m=k}} h_\ell[AX]h_m[BX]^\perp.
\end{align*}
Since $g^\perp \Lambda^n=0$ for any $g\in\Lambda^m$ such that $m>n$, the expression for $V_k$ above reduces to a finite sum when acting on any fixed $\Lambda^n$. In particular, each $V_k$ for $k\in\Z$ is a well-defined endomorphism of $\Lambda$ (of degree $k$).

By acting on the reproducing kernel $\Omega[XY]$ for the Hall pairing, one can also show that the annihilating half of a vertex operator acts on a symmetric function $f\in\Lambda$ by the plethystic formula:
\begin{align*}
\Omega[z^{-1}BX]^\perp\cdot f &= f[X+z^{-1}B].
\end{align*}


\subsection{Macdonald vertex operator}

In the context of Macdonald $P$-functions, the vertex operator
\begin{align}\label{E:MacD}
D(z)=\sum_{n=-\infty}^\infty D_k z^{-k}=\Omega[z(t-1)X]\Omega[z^{-1}\frac{q-1}{t}X]^\perp
\end{align}
is of particular importance. Here we have the two halves
\begin{align*}
\Omega[z(t-1)X] &= \prod_{i=1}^\infty \frac{1-zx_i}{1-tzx_i}
= \exp\left(\sum_{k=1}^\infty \frac{p_k}{k}(t^k-1)z^k\right)\\
\Omega[z^{-1}\frac{q-1}{t}X]^\perp &= \exp\left(\sum_{k=1}^\infty \frac{p_k^\perp}{k}\frac{q^k-1}{t^k}z^{-k}\right).
\end{align*}

\begin{thm}[{\cite[Theorem 3.2]{GH}}]\label{T:GH}
The coefficient $D_0$ of $z^0$ in the vertex operator $D(z)$ acts diagonally on Macdonald symmetric functions
\begin{align}\label{E:D0-P}
D_0 P_\mu &= A_\mu(q,t^{-1})P_\mu,
\end{align}
with $A_\mu$ given by \eqref{E:Amudef}.
\end{thm}

\begin{rem}
The operator $D_0$ arises as the $N\to\infty$ limit of the operator $t^{-N}(1+(1-t)M_N)$, and this is essential in the proof of Theorem~\ref{T:GH}.
The renormalization of $M_N$ is necessary for compatibility with the projective limit defining $\Lambda$ (see \cite[VI.4]{Mac}). One sees this most transparently when deriving the eigenvalues in \eqref{E:D0-P} from those in \eqref{E:M-P}.
\end{rem}

\begin{ex}
For the purpose of illustration, let us verify \eqref{E:D0-P} for the Macdonald symmetric function $P_{(1,0)} = p_1$. First, we compute:
$$
\Omega[z^{-1}\frac{q-1}{t}X]^\perp\cdot P_{(1,0)} = p_1+z^{-1}\frac{q-1}{t}.
$$
Since $\Omega[z(t-1)X]=1+z(t-1)p_1+\dotsm$, we then find:
\begin{align*}
D_0P_{(1,0)} &= \left(1+(t-1)\frac{q-1}{t}\right)p_1\\
&= (1-(1-q)(1-t^{-1}))p_1\\
&= A_{(1,0)}(q,t^{-1})p_1.
\end{align*}
\end{ex}

\subsection{Modified Macdonald functions}
The modified Macdonald symmetric functions $\{\tH_\mu\}_{\mu\in\Y}$ are defined by
\begin{align*}
\tH_\mu &= t^{n(\mu)} (J_\mu[X/(1-t)]|_{t\mapsto t^{-1}}),
\end{align*}
where $n(\mu) = \sum_i (i-1)\mu_i$ and $J_\mu = \prod_{s\in\lambda}(1-q^{a_\mu(s)}t^{l_\mu(s)+1}) P_\mu.$

The $\tH_\mu$ are uniquely characterized by the following three conditions \cite[\S 3.5]{H}:
\begin{align}
\label{E:tH-def-q}
\tH_\mu[(1-q)X] &\in \bigoplus_{\lambda\unrhd \mu} \K s_\lambda\\
\label{E:tH-def-t}
\tH_\mu[(1-t)X] &\in \bigoplus_{\lambda\unrhd \mu^t} \K s_\lambda\\
\label{E:tH-def-normalization}
\langle s_n, \tH_\mu\rangle &= 1.
\end{align}

\begin{rem}
One can write \eqref{E:tH-def-t} in the equivalent form
\begin{align}
\label{E:tH-def-t-inv}
\tH_\mu[(1-t^{-1})X] &\in \bigoplus_{\lambda\unlhd \mu} \K s_\lambda,
\end{align}
since $\tH_\mu$ is homogeneous, $s_\lambda[-X]=(-1)^{|\lambda|}s_{\lambda^t}$, and $\lambda \unrhd \mu^t$ if and only if $\lambda^t\unlhd \mu$.
\end{rem}

The $\tH_\mu$ are often preferred over $P_\mu$ due to their remarkable integrality and positivity properties. In particular, defining the Macdonald-Kostka coefficients $\tK_{\lambda\mu}(q,t)$ through the expansion
\begin{align*}
\tH_\mu = \sum_\la \tK_{\la\mu}(q,t)s_\la,
\end{align*}
the Macdonald positivity conjecture \cite[VI.8]{Mac} (cf. Theorem~\ref{T:Haiman} below) asserts that
\begin{align}\label{E:Mac-pos}
\tK_{\lambda\mu}(q,t)\in\Z_{\ge 0}[q,t].
\end{align}

\begin{ex}
One has
\begin{align*}
J_{(2)}&=(1-t)(1-qt)P_{(2)}\\
&=(1-t)(1-qt)m_{(2)}+(1+q)(1-t)^2m_{(1,1)}\\
&=\frac{(1-q)(1-t^2)}{2}p_2+\frac{(1+q)(1-t)^2}{2}p_1^2
\end{align*}
and the corresponding Macdonald symmetric function
\begin{align*}
\tH_{(2)}&=t^0\left.\left(\frac{1-q}{2}p_2+\frac{1+q}{2}p_1^2\right)\right|_{t\mapsto t^{-1}}\\
&=\frac{1-q}{2}p_2+\frac{1+q}{2}p_1^2\\
&=s_{(2)}+q s_{(1,1)}.
\end{align*}
We observe \eqref{E:Mac-pos} in this case. We also see that \eqref{E:tH-def-q}-\eqref{E:tH-def-normalization} hold, with the only nontrivial remaining computation being:
\begin{align*}
\tH_{(2)}[(1-q)X] &= \frac{(1-q)(1-q^2)}{2}p_2+\frac{(1+q)(1-q)^2}{2}p_1^2\\
&= \frac{(1-q)(1-q^2)}{2}(p_2+p_1^2)\\
&= (1-q)(1-q^2)s_{(2)}.
\end{align*}
\end{ex}

Conjugating the vertex operator of Theorem~\ref{T:GH} by the transformation from $P_\mu$ to $\tH_\mu$ and making use of the identity $P_\mu|_{q\mapsto q^{-1},t\mapsto t^{-1}}=P_\mu$ \cite[VI, (4.14)]{Mac}, we arrive at the following:
\begin{cor}\label{C:tH-eigen}
The coefficients $\tD_0$ and $\tD_0^*$ of $z^0$ in the vertex operators
\begin{align}
\label{E:tD}
\tD(z)&=\sum_{n=-\infty}^\infty \tD_k z^{-k}=\Omega[-zX]\Omega[z^{-1}(1-q)(1-t)X]^\perp\\
\label{E:tD-star}
\tD^*(z)&=\sum_{n=-\infty}^\infty \tD_k^* z^{-k}=\Omega[zX]\Omega[-z^{-1}(1-q^{-1})(1-t^{-1})X]^\perp
\end{align}
act diagonally on modified Macdonald functions:
\begin{align*}
\tD_0 \tH_\mu &= A_\mu(q,t) \tH_\mu\\
\tD_0^* \tH_\mu &= A_\mu(q^{-1},t^{-1}) \tH_\mu.
\end{align*}
\end{cor}

\subsection{Haiman's Positivity Theorem}\label{SS:Haiman}

Let $\Hilb_n=\Hilb_n(\C^2)$ be the Hilbert scheme of $n$ points in the plane $\C^2$. This is a resolution of the symplectic quotient singularity $\C^{2n}/\fS_n$; in particular, $\Hilb_n$ is a smooth variety of dimension $2n$. Its points can be described by ideals $I\subset \C[x,y]$ such that $\dim \C[x,y]/I = n$. The natural action of $\T=(\C^\times)^2$ on $\C^2$ induces an action of $\T$ on $\Hilb_n$ with finitely many fixed points---namely, the monomial ideals $I_\mu=(y^{\mu_1},xy^{\mu_2},\dotsc)$ for $\mu\in \Y$ with $|\mu|=n$.

To connect with Macdonald theory, we identify $\Z[q^{\pm 1},t^{\pm 1}]$ with the representation ring $R(\T)$ and the field $\K=\Q(q,t)$ with $\Frac(R(\T))$. Here we follow the conventions of \cite{H}, so that the standard coordinate functions $x(a_1,a_2)=a_1$ and $y(a_1,a_2)=a_2$ on $\C^2$ have $\T$-weights $t$ and $q$, respectively.
 

Haiman's proof of the Macdonald positivity conjecture \eqref{E:Mac-pos} centers on establishing the existence of a rank $n!$ vector bundle $P_n$ on $\Hilb_n$, the \textit{Procesi bundle}, with the following special properties:
\begin{enumerate}
\item $P_n$ is $(\T\times \fS_n)$-equivariant, with respect to the trivial $\fS_n$-action on $\Hilb_n$
\item each fiber of $P_n$ is isomorphic as an $\fS_n$-representation to the regular representation\label{I:reg}
\item each fiber of $P_n$ is an $\fS_n$-equivariant quotient of the polynomial ring $\C[x_1,\dotsc,x_n,y_1,\dotsc,y_n]$ with respect to the diagonal action of $\fS_n$ and with $\T$ acting so that each $x_i$ (resp. each $y_i$) has weight $t$ (resp. $q$)\label{I:pol}
\end{enumerate}
In particular, any fiber $P_n|_{I_\mu}$ over a $\T$-fixed point in $\Hilb_n$ carries an action of $\T\times \fS_n$; equivalently, one may regard $P_n|_{I_\mu}$ as a \textit{bigraded} $\fS_n$-module.

\begin{thm}[\cite{H:pos}]\label{T:Haiman}
There exists a vector bundle $P_n$ on $\Hilb_n$ with the properties described above. This vector bundle induces an equivalence between derived categories of coherent sheaves
\begin{align}\label{E:Haiman equivalence}
D^b(\Coh^\T(\Hilb_n)) \cong D^b(\Coh^{\T\times \fS_n}(\C^{2n}))
\end{align}
which, at the level of $K$-theory, realizes the modified Macdonald function $\tH_\mu$ as the bigraded Frobenius character of the fiber $P_n|_{I_\mu}$.
\end{thm}

Let us describe how the equivalence \eqref{E:Haiman equivalence} connects to symmetric functions in more detail. Passing to equivariant $K$-theory, i.e., applying the $K_0$-functor to \eqref{E:Haiman equivalence}, one obtains an isomorphism of $R(\T)$-modules 
$K^\T(\Hilb_n)\cong K^{\T\times \fS_n}(\C^{2n})$.
As explained in \cite[\S 5.4.3]{H}, there is an injection $K^{\T\times \fS_n}(\C^{2n})\hookrightarrow \Lambda^n$ of $R(\T)$-modules given by the global sections map followed by the bigraded Frobenius series $\mathrm{Frob}_{q,t}$. The composite map $\Phi_n : K^\T(\Hilb_n)\hookrightarrow \Lambda^n$ is given explicitly by $\Phi_n([\cF])=\mathrm{Frob}_{q,t}(R\Gamma([\cF]\otimes P_n))$ for any $\T$-equivariant coherent sheaf $\cF$ on $\Hilb_n$. Letting $[I_\mu]$ denote the class of the skyscraper sheaf at a $\T$-fixed point $I_\mu$ in $\Hilb_n$, the second assertion in Theorem~\ref{T:Haiman} is then stated precisely as follows: 
\begin{align*}
\Phi_n([I_\mu])=\mathrm{Frob}_{q,t}(P_n|_{I_\mu})=\tH_\mu.
\end{align*}
With $\tH_\mu$ realized as a bigraded Frobenius character, one immediately deduces that  $\tK_{\lambda\mu}(q,t)$ is a \textit{Laurent} polynomial in $q$ and $t$ with nonnegative integer coefficients for all $\lambda,\mu\in \Y$. Property \eqref{I:pol} of the vector bundle $P_n$ then implies the full assertion \eqref{E:Mac-pos} of the Macdonald positivity conjecture.





\begin{ex}
Let us observe Macdonald positivity \eqref{E:Mac-pos} in conjunction with property \eqref{I:reg} of the Procesi bundle for $|\mu|=3$:
\begin{align*}
\mathrm{Frob}(\C\fS_3) &= s_{(3)} + 2 s_{(2,1)} + s_{(1,1,1)}\\
\tH_{(3)} &= s_{(3)}+(q+q^2)s_{(2,1)}+q^3s_{(1,1,1)}\\
\tH_{(2,1)} &= s_{(3)}+(q+t)s_{(2,1)}+qts_{(1,1,1)}\\
\tH_{(1,1,1)} &= s_{(3)}+(t+t^2)s_{(2,1)}+t^3s_{(1,1,1)}.
\end{align*}
\end{ex}

\subsection{Quantum toroidal $\gl_1$}

The maps $\Phi_n$ arising from Theorem~\ref{T:Haiman} admit a further representation-theoretic interpretation involving an algebra $\Utor_1$ known as the elliptic Hall algebra or quantum toroidal $\gl_1$ algebra. This result---which beautifully unifies all of the material in this section---is due independently to Schiffmann and Vasserot \cite{SV2} and Feigin and Tsymbaliuk \cite{FT}.

Let $\Phi : \bigoplus_{n\ge 0} K^\T(\Hilb_n)_\loc \to \Lambda$ be the direct sum of the maps $K^\T(\Hilb_n)_\loc \overset{\sim}{\to}\Lambda^n$ obtained from the $\Phi_n$ by extension of scalars to the localized equivariant $K$-groups $K^\T(\Hilb_n)_\loc=K^\T(\Hilb_n)\otimes_{R(T)}\K$.
Then $\Phi$ is the isomorphism of $\K$-vector spaces given by $\Phi([I_\mu])=\tH_\mu$ for all $\mu\in\Y$. The result of \cite{FT,SV2} asserts the isomorphism $\Phi$ intertwines two actions of $\Utor_1$. Without giving a complete definition of $\Utor_1$, let us briefly describe the two actions:

\begin{itemize}

\item The action of $\Utor_1$ on $\Lambda$ is generated by all coefficients $\tD_k$ and $\tD_k^*$ of the vertex operators $\tD(z)$ and $\tD^*(z)$ from Corollary~\ref{C:tH-eigen}, together with all multiplication opertors $f$ and all skewing operators $f^\perp$ for $f\in\Lambda$. In fact, the four operators $p_1, p_1^\perp, \tD_0$, and $\tD_0^*$ are sufficient to generate the entire action of $\Utor_1$.

\item The action of $\Utor_1$ on $\bigoplus_{n=0}^\infty K^\T(\Hilb_n)_\loc$ is defined geometrically, along the lines of earlier constructions of Nakajima and Grojnowski. This involves the nested Hilbert schemes $\Hilb_{n,n+1}\subset \Hilb_n \times \Hilb_{n+1}$ parametrizing pairs of ideals $(I,J)\in\Hilb_n \times \Hilb_{n+1}$ such that $J\subset I$. Each nested Hilbert scheme is equipped with natural projections $p_{n,n+1} : \Hilb_{n,n+1} \to \Hilb_n$ and $q_{n,n+1} : \Hilb_{n,n+1}\to \Hilb_{n+1}$, as well as a line bundle $L_{n,n+1}$ whose fiber over $(I,J)$ is the quotient space $I/J$. 

Define convolution operators $e_k : K^\T(\Hilb_n) \to K^\T(\Hilb_{n+1})$ and $f_k : K^\T(\Hilb_n) \to K^\T(\Hilb_{n+1})$ for all $k\in\Z$ by
\begin{align*}
e_k([\cF])&=(q_{n,n+1})_*([L_{n,n+1}^{\otimes k}]\cdot p_{n,n+1}^*([\cF]))\\
f_k([\cF])&=(p_{n,n+1})_*([L_{n,n+1}^{\otimes (k-1)}]\cdot q_{n,n+1}^*([\cF])).
\end{align*}
Note that $f_k$ is defined to be $0$ on $K^\T(\Hilb_0)$. We continue denote the induced operators on $\bigoplus_{n=0}^\infty K^\T(\Hilb_n)_\loc$ by $e_k$ and $f_k$.

Finally, on each Hilbert scheme $\Hilb_n$ there is a tautological vector bundle $T_n$ whose fibers are given by $T_n|_I=\C[x,y]/I$. For any symmetric function $f\in\Lambda$, define $h_f$ and $h_f^*$ to be the operators on $\bigoplus_{n=0}^\infty K^\T(\Hilb_n)_\loc$ which on each $K^\T(\Hilb_n)_\loc$ are given by $K$-theory multiplication with the virtual equivariant bundle $f[1-(1-q)(1-t)T_n]$ and its dual, respectively. The $\T$-character of $1-(1-q)(1-t)T_n$ over a fixed point $I_\mu$ is precisely $A_\mu(q,t)$ from \eqref{E:Amudef}, since that of $T_n$ is given by $B_\mu(q,t)$. The $\T$-character of $f[1-(1-q)(1-t)T_n]|_{I_\mu}$ is given by the plethysm $f[A_\mu(q,t)]$.

The operators $e_k,f_k,h_f$, and $h_f^*$ for $k\in\Z$ and $f\in \Lambda$ generate the action of $\Utor_1$ on $\bigoplus_{n=0}^\infty K_\T(\Hilb_n)_\loc$.

\end{itemize}

The following result gives the connection these two $\Utor_1$-actions:

\begin{thm}[\cite{FT,SV2}]\label{T:FT-SV}
The isomorphism $\Phi$ intertwines the two actions of $\Utor_1$ described above as follows: the matrix coefficients of the operators $e_{-1}, f_1, h_{p_1}, h_{p_1}^*$ in the basis $\{[I_\mu]\}_{\mu\in\Y}$ are equal to those of $\frac{1}{(1-q)(1-t)}p_1, p_1^\perp, \tD_0, \tD_0^*$ in the basis $\{\tH_\mu\}_{\mu\in\Y}$, respectively.
\end{thm}

The nontrivial identification of generators arising in Theorem~\ref{T:FT-SV} is given by the \textit{Miki automorphism} (or Fourier transform \cite{SV1}) of $\Utor_1$. This automorphism exists for all quantum toroidal algebras $\Utor_r$ of type $\gl_r$ (see \S\ref{SS:miki} for the $r\ge 3$ case) and it plays a key role in Wen's extension of Theorem~\ref{T:FT-SV} to the wreath setting (Theorem~\ref{T:eigen}).

\section{Wreath Macdonald Theory}

The definition of wreath Macdonald polynomials has two main ingredients.
The first is a partial order $\dom_w$ on $\Y^I$ that depends on an affine Weyl group element $w\in\Waf$ (definitions and notation appear later in this section). This comes from the geometry of quiver varieties.
The second is the construction of matrix plethysms $\cP_M$, which are endomorphisms of the $I$-fold tensor product of symmetric functions $\La^{\otimes I}$ determined by an $I\times I$ matrix $M$ with coefficients in the base field $\K=\Q(q,t)$. The specific matrix plethysms come from tensoring with certain representations of the wreath product $\Gamma_n = \fS_n \wr \Z/r\Z$ via a wreath Frobenius map.

\subsection{Combinatorics of the partial orders}
Fix an integer $r\ge1$ and let $I=\Z/r\Z$, which is regarded as the set of nodes of the affine Dynkin diagram of type $A_{r-1}^{(1)}$. 

\subsubsection{Maya diagrams, 1-runner abaci, and charged partitions}
\label{SS:Maya}

A Maya diagram or edge sequence is a function $b: \Z\to \{0,1\}$ such that $b(i)=0$ for $i\ll0$ and $b(i)=1$ for $i\gg 0$. It can be viewed as a 1-runner abacus with integer positions which places a bead at $i$ if $b(i)=0$ and a hole at $i$ if $b(i)=1$. For $c\in\Z$ let $\vn_c$ be the Maya diagram defined by 
$\vn_c(i) = 0$ for $i< c$ and $\vn_c(i)=1$ for $i\ge c$.
The charge of $b$ is the unique integer $c$ such that
\begin{align}\label{E:charge}
| \{ k\in\Z\mid \text{$k < c$ and $b(k)=1$} \} | = | \{ k\in\ \mid \text{$k \ge c$ and $b(k)=0$} \}|,
\end{align}
that is, the number of holes at positions before $c$, is equal to the number of beads at positions at least $c$. Equivalently, starting with $b$ and repeatedly exchanging beads with holes to their left, one reaches an abacus with all beads to the left of all holes, that is, an abacus of the form $\vn_c$. Then $c=\charge(b)$. The Durfee square side length $d(b)$ of $b$ is the common cardinality in \eqref{E:charge}.

\begin{ex} For neatness in the diagram below we write $\overline{k}$ for $-k$. Let $b$ be defined by
\renewcommand{\arraystretch}{1.2}
\[
\begin{array}{|c||c|c|c|c|c|c||c|c|c|c|c|c|c|c|} \hline
i& \dotsm&\ol{5}&\ol{4}&\ol{3}&\ol{2}&\ol{1}&0&1&2&3&4&\dotsm \\ \hline \hline
b(i) &0&0&1&0&1&0&1&1&0&0&1&1  \\ \hline
\end{array} \,\,.
\]
Its abacus is drawn below.
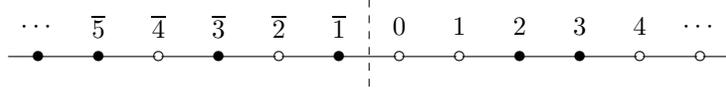
\[
\begin{tikzpicture}[scale=.8]
\node at (-5,0) {$\bullet$};
\node at (-4,0) {$\bullet$};
\node at (-3,0) {$\circ$};
\node at (-2,0) {$\bullet$};
\node at (-1,0) {$\circ$};
\node at (0,0) {$\bullet$};
\node at (1,0) {$\circ$};
\node at (2,0) {$\circ$};
\node at (3,0) {$\bullet$};
\node at (4,0) {$\bullet$};
\node at (5,0) {$\circ$};
\node at (6,0) {$\circ$};
\node at (-5,.5) {$\dotsm$};
\node at (-4,.5) {$\ol{5}$};
\node at (-3,.5) {$\ol{4}$};
\node at (-2,.5) {$\ol{3}$};
\node at (-1,.5) {$\ol{2}$};
\node at (0,.5) {$\ol{1}$};
\node at (1,.5) {$0$};
\node at (2,.5) {$1$};
\node at (3,.5) {$2$};
\node at (4,.5) {$3$};
\node at (5,.5) {$4$};
\node at (6,.5) {$\dotsm$};
\foreach \x in {-5,...,5}
  \draw (\x+.05,0)--(\x+.95,0);
\draw[dashed] (.5,-.5) -- (.5,1);
\draw (-5.5,0)--(-5.05,0);
\draw (6.05,0)--(6.5,0);
\end{tikzpicture}
\]
We have $\charge(b)=0$; the dashed vertical line between $\ol{1}$ and $0$ indicates this.
We have $d(b)=2$.
\end{ex}

There is a partition defined by a Maya diagram $b$ we denote by $\shape(b)$. It is defined as follows. We draw a polygonal path in the plane which follows a sequence of vectors (for increasing $i\in \Z$) in which the $i$-th vector is $(-1,0)$ if $b(i)=0$ and $(0,1)$ if $b(i)=1$, anchored such that the $(c-1)$-th vector enters the point $(d(b),d(b))$ and the $c$-th vector leaves it. 
The path is the top right border of the diagram of the partition which we call $\shape(b)$. Let $\Y$ be the set of partitions and $\cM$ the set of Maya diagrams. There is a bijection 
\begin{align*}
\cM &\to \Y \times \Z \\
b &\mapsto (\shape(b), \charge(b)).
\end{align*}
We call the elements of $\Y\times\Z$ charged partitions.

\begin{ex} The path associated to the Maya diagram $b$ of the previous example is pictured below.
We have $\shape(b)=(4,3,2,2)$.
\[
\begin{tikzpicture}
\node at (-.5,-.5) {$(0,0)$};
\node at (6.5,-.5) {$(6,0)$};
\node at (-.5,6.5) {$(0,6)$};
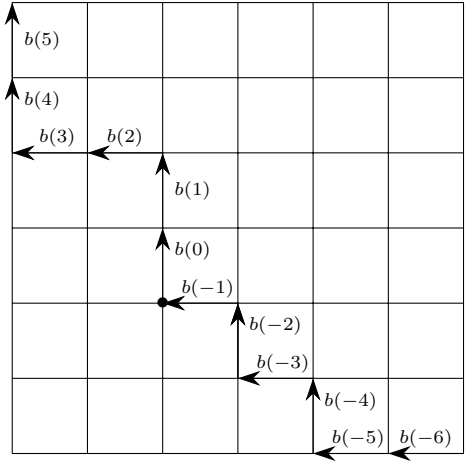
\draw (0,0) grid (6,6);
\node at (2,2) {$\bullet$};
\node at (.4,5.5) {\scriptsize{$b(5)$}};
\node at (.4,4.7) {\scriptsize{$b(4)$}};
\node at (.6,4.2) {\scriptsize{$b(3)$}};
\node at (1.5,4.2) {\scriptsize{$b(2)$}};
\node at (2.4,3.5) {\scriptsize{$b(1)$}};
\node at (2.4,2.7) {\scriptsize{$b(0)$}};
\node at (2.6,2.2) {\scriptsize{$b(-1)$}};
\node at (3.5,1.7) {\scriptsize{$b(-2)$}};
\node at (3.6,1.2) {\scriptsize{$b(-3)$}};
\node at (4.5,.7) {\scriptsize{$b(-4)$}};
\node at (4.6,.2) {\scriptsize{$b(-5)$}};
\node at (5.5,.2) {\scriptsize{$b(-6)$}};
\tikzset{>={Stealth[length=3mm, width=2mm]}}
\draw[->] (6,0)--(5,0);
\draw[->] (5,0)--(4,0);
\draw[->] (4,0)--(4,1);
\draw[->] (4,1)--(3,1);
\draw[->] (3,1)--(3,2);
\draw[->] (3,1)--(3,2);
\draw[->] (3,2)--(2,2);
\draw[->] (2,2)--(2,3);
\draw[->] (2,3)--(2,4);
\draw[->] (2,4)--(1,4);
\draw[->] (1,4)--(0,4);
\draw[->] (0,4)--(0,5);
\draw[->] (0,5)--(0,6);
\end{tikzpicture}
\]
\end{ex}

\begin{rem} Usually the bijection between Maya diagrams and charged partitions is defined in an opposite manner, by arrows pointing in the opposite directions and indexing them so as $i\in\Z$ increases, the arrows proceed from upper left to lower right. There are a number of notational consequences of this difference in indexing which we will need to point out from time to time.
\end{rem}

\subsubsection{Cores and quotients}
\label{SS:core-quot}

Let $\cC=\cC_r\subset\Y$ be the set of $r$-cores, which by definition are the partitions $\la$ with no cell $(a,b)\in\la$ of hook length $h_\la(a,b)=r$. 

Let $\Y^I$ be the set of $I$-multipartitions, tuples $(\mu^{(0)},\mu^{(1)},\dotsc,\mu^{(r-1)})$ of partitions indexed by the affine Dynkin node set $I=\Z/r\Z$.

\subsubsection{Cores and quotients via (un)interleaving Maya diagrams}
\label{SS:core quotient via maya}
Let $b$ be a Maya diagram or equivalently 1-runner abacus. For $i\in I$ let $b^{(i)}$ be the 1-runner abacus obtained by selecting the subsequence of beads in positions in the coset of $i$ mod $r\Z$, that is,
\begin{align*}
b^{(i)}(k) = b(kr+i) \qquad\text{for all $k\in\Z$ and $0\le i\le r-1$.}
\end{align*}
This gives a bijection
\begin{align}
\notag  \cM &\cong \cM^I \\
\label{E:quotient on Maya diagrams}
  b&\mapsto \quot_r(b) = (b^{(0)},b^{(1)},\dotsc,b^{(r-1)}),
\end{align}
which we call the $r$-quotient for Maya diagrams. The inverse bijection is given by interleaving the $r$ Maya diagrams $b^{(0)},\dotsc,b^{(r-1)}$ using the same formula in \eqref{E:quotient on Maya diagrams}. Pictorially we place $b^{(i)}$ as the $i$-th runner in an $r$-runner abacus and read the beads down the columns, proceeding from columns on the left to columns on the right.

\begin{ex} For $b$ in the previous example, the Maya diagram $b^{(i)}$ in $\quot_3(b)$ is pictured as  the $i$-th runner in the $3$-runner abacus pictured below. The vertical dashed lines picture the charge
for each runner: $\charge(b^{(0)})=1$, $\charge(b^{(1)})=-1$, and $\charge(b^{(2)})=0$.
\[
\begin{tikzpicture}[scale=.8]
\node at (-5,2) {runner $0$};
\node at (-5,1) {runner $1$};
\node at (-5,0) {runner $2$};
\node at (-3,2) {$\dotsm$};
\node at (-3,1) {$\dotsm$};
\node at (-3,0) {$\dotsm$};
\node at (-2,2.5) {$\ol{3}$};
\node at (-2,2) {$\bullet$};
\node at (-2,1) {$\bullet$};
\node at (-2,0) {$\bullet$};
\node at (-1,2.5) {$\ol{2}$};
\node at (-1,2) {$\bullet$};
\node at (-1,1) {$\bullet$};
\node at (-1,0) {$\circ$};
\node at (0,2.5) {$\ol{1}$};
\node at (0,2) {$\bullet$};
\node at (0,1) {$\circ$};
\node at (0,0) {$\bullet$};
\node at (1,2.5) {$0$};
\node at (1,2) {$\circ$};
\node at (1,1) {$\circ$};
\node at (1,0) {$\bullet$};
\node at (2,2.5) {$1$};
\node at (2,2) {$\bullet$};
\node at (2,1) {$\circ$};
\node at (2,0) {$\circ$};
\node at (3,2.5) {$2$};
\node at (3,2) {$\circ$};
\node at (3,1) {$\circ$};
\node at (3,0) {$\circ$};
\node at (4,2) {$\dotsm$};
\node at (4,1) {$\dotsm$};
\node at (4,0) {$\dotsm$};
\foreach \y in {2,1,0} {
  \draw (-2.5,\y)--(-1.95,\y);
  \draw (3.05,\y)--(3.5,\y);}
\foreach \x in {-2,...,2}
  \foreach \y in {2,1,0}
  \draw (\x+.05,\y)--(\x+.95,\y);
\draw[dashed] (1.5,1.5)--(1.5,2.5);
\draw[dashed] (-.5,.5)--(-.5,1.5);
\draw[dashed] (.5,-.5)--(.5,.5);
\end{tikzpicture}
\]
\ytableausetup{boxsize=4pt}
Under $\cM \cong \Y\times\Z$ we have $b^{(0)}\mapsto (\ya,1)$,
$b^{(1)}\mapsto (\vn,-1)$, and $b^{(2)}\mapsto (\yb,0)$; here $\vn$ denotes the empty partition.
\end{ex}

We have the following commutative diagram in which all horizontal maps are bijections and vertical maps are obvious inclusions, where $(\vn^I)$ denotes the empty multipartition.
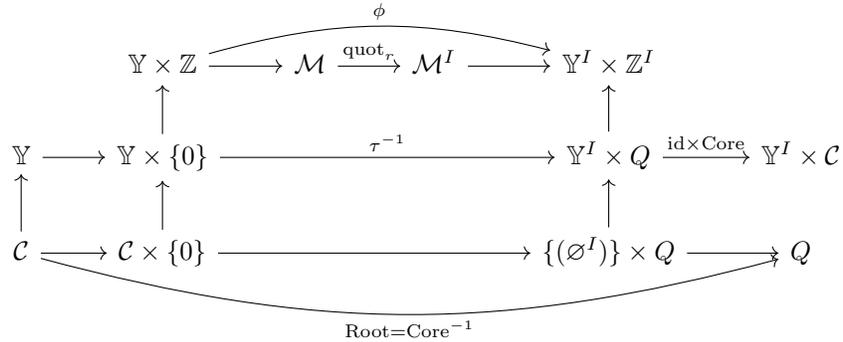
\[
\begin{tikzcd}
&\Y \times \Z \arrow[r] \arrow[rrr,bend left=15,"\phi"]& \cM \arrow[r,"\quot_r"] & \cM^I \arrow[r] & \Y^I \times \Z^I & \\
\Y\arrow[r] &\Y \times \{0\} \arrow[u] \arrow[rrr,"\tau^{-1}"] &&&\Y^I \times Q \arrow[u] \arrow[r,"\id\times \coremap"]& \Y^I \times \cC \\
\cC \arrow[r] \arrow[u] \arrow[rrrrr,bend right=15,swap,"\rootmap=\coremap^{-1}"]& \cC \times \{0\} \arrow[u] \arrow[rrr] &&& \{(\vn^I)\} \times Q \arrow[r] \arrow[u] & Q 
\end{tikzcd}
\]
The top row of maps are given by bijections described above; we use the composite map $\cM^I \cong (\Y\times \Z)^I\cong \Y^I\times \Z^I$ using the $I$-fold product of the bijection $\cM\to \Y\times\Z$ followed by reordering Cartesian factors.
Let $\phi$ be the composite of the top row of maps. For $(\mu,c)\in\Y\times\Z$ let $\phi(\mu,c)=(\mud,\beta)$ where $\mud=(\mu^{(0)},\dotsc,\mu^{(r-1)})\in\Y^I$ and $\beta=(\beta_0,\dotsc,\beta_{r-1})\in\Z^I$. It is straightforward to verify that
\begin{align*}
  c= \sum_{i\in I} \beta_i.
\end{align*}
Restricting $\phi$ to $\Y \times \{0\}$ we get a bijection $\Y\cong \Y^I \times Q$ where $Q\subset \Z^I$ is the root lattice of $\mathfrak{sl}_r$ realized by the zero-sum elements of $\Z^I$.
The simple roots $\alb_i\in Q$ are realized by $\epsilon_{i-1}-\epsilon_i\in\Z^I$ for $1\le i\le r-1$ where $\{\epsilon_i\mid i\in I\}$ is the standard basis of $\Z^I$. The inverse of this bijection we will denote by $\tau: \Y^I\times Q \to \Y$. 

Let $(\vn^I)\in \Y^I$ denote the multipartition consisting of empty partitions. Then restricting $\tau$ to $\{(\vn^I)\}\times Q$ we get a bijection $Q\to \cC$ called $\coremap$ whose inverse bijection we will call $\rootmap: \cC\to Q$. Composing $\tau^{-1}$ with $\id_{\Y^I}\times \coremap$ we obtain a bijection
\begin{align*}
\Y \to \Y^I \times \cC ,
\end{align*}
whose components are known as the $r$-quotient $\quot_r:\Y\to \Y^I$ and the $r$-core $\core_r:\Y\to\cC$.

\begin{rem} The classical $r$-core map agrees with our map $\core_r$,
but the classical $r$-quotient map $\quot'_r$ (see \cite[Example I.1.8]{Mac} with $m$ taken to be a multiple of $r$, or \cite[\S 6.2]{G}) produces the reverse sequence of partitions than ours:
\begin{align*}
  \quot_r(\mu) = w_0 \,\quot'_r(\mu).
\end{align*}
where $w_0$ is the reversing permutation.
This is due to our convention for the bijection $\cM\cong \Y\times\Z$, namely, the direction of tracing of the edge of the partition from bottom right to upper left.
\end{rem}

\begin{ex} Continuing the above example, starting with $((4,3,2,2),0)\in\Y\times\Z$, we map to $\Y^I\times \Z^I$, getting $\mud = (\ya,\vn,\yb)$ and $\beta=(1,-1,0)$.
\end{ex}

For $\la\in\Y$ a $\la$-removable $r$-ribbon is a set of cells of the form $D(\la)\setminus D(\mu)$ for some $\mu\in\Y$ with $D(\mu)\subset D(\la)$, which is rookwise connected and has exactly one cell of each residue. 
Then $\la$ is an $r$-core if and only if it has no removable $r$-ribbon.

One way to obtain $\core_r(\mu)$ is to repeatedly remove (removable) $r$-ribbons starting with $\mu$ until an $r$-core is reached. One obtains the same $r$-core independently of the order of removal of $r$-ribbons. This resulting $r$-core is $\core_r(\mu)$.

\subsubsection{Transposing partitions and the core-quotient bijection}
For all $\la\in\Y$ we have
\begin{align*}
  \core_r(\la^t) &= \core_r(\la)^t \\
  \quot_r(\la^t) &= \quot_r(\la)^*,
\end{align*}
where $*$ denotes the reverse of the componentwise transpose of a multipartition.

\subsubsection{Residues}\label{SS:par}
We shall only require residues for boxes with respect to a partition $\shape(b)$ associated with a Maya diagram $b$ of charge $0$. In such a diagram the $i$-th step tracing the border of the partition $\shape(b)$, starts on the diagonal $y-x=i$. Therefore it is consistent to name this diagonal line as the $i$-th diagonal, that is, to define the content or diagonal index of $(x,y)$ by $y-x$ 
and to define the residue of $(x,y)\in\Z^2$ as the residue of $y-x$ in $\Z/r\Z$. \footnote{This is the negative of the standard convention.} 
For a partition $\mu\in\Y$ denote by $A_i(\mu)\subset A(\mu)$ and $R_i(\mu)\subset R(\mu)$ the subsets of $\mu$-addable and $\mu$-removable cells having residue $i$. 

\subsubsection{Cores to root lattice, reprised}
The bijection $\rootmap:\cC\to Q$ was defined previously in terms of the bijection $\phi$.
It can be computed directly as follows.
Let $Q_\af = \bigoplus_{i\in I} \Z \alpha_i$ be the root lattice of 
the affine Kac-Moody algebra $\hat{\mathfrak{sl}}_r$ with simple roots $\{\alpha_i\mid i\in I\}$. Define the map $\kappa:\Y\to Q_\af^+=\bigoplus_{i\in I} \Z_{\ge0} \alpha_i$ by
\begin{align*}
  \kappa(\mu) = \sum_{(x,y)\in\mu} \alpha_{y-x},
\end{align*}
where the subscripts of the affine simple roots are taken modulo $r$ as usual.
Let $\cl: Q_\af\to Q$ be the restriction map. 
We have $\cl(\al_i)=\alb_i$ for $1\le i\le r-1$ while $\cl(\alpha_0) = - \theta = -\sum_{i=1}^{r-1} \alb_i$ where $\theta\in Q$ is the highest root. The kernel of $\cl$ is 
$\Z\delta$ where $\delta=\sum_{i\in I} \alpha_i\in Q_\af$ is the null root. Define the map
$\kb: \Y \to Q$ by 
\begin{align}\label{E:kb}
\kb(\mu) = - \cl(\kappa(\mu)).
\end{align}
Since each $r$-ribbon contributes $\delta$ in the computation of $\kappa$, it follows from \cite[\S 6.5,7.5]{G} that the restriction of $\kb$ to $\cC$, which we shall by abuse of notation also call $\kb$, defines a bijection $\cC\to Q$.
The reason we use the negative sign in \eqref{E:kb} is because with this definition, $\kb$ and $\rootmap$ define the same bijection $\cC\to Q$.
\begin{ex}\label{X:core to root}
The boxes of the partition $\mu=(4,3,2,2)$ and its 3-core $\gamma=(2)$ 
are labeled by their residues in $\Z/3\Z$.
\ytableausetup{boxsize=normal,aligntableaux=bottom}
\[
\begin{ytableau}
0 & 2\\
 2 & 1\\
1 & 0&2 \\
 0& 2& 1& 0\\
\end{ytableau}
\qquad \quad\begin{ytableau} 0 & 2\end{ytableau}
\]
We have $\kb(\mu)=\kb(\gamma) = - \cl(\alpha_0+\alpha_2) = \alb_1$.
\end{ex}

\subsubsection{Affine Weyl action on partitions}\label{SS:root to core}
We now give another way to compute the bijection $Q \to \cC$ inverse to $\kb$, based on affine crystal graph computations. Let $\Waf$ be the affine Weyl group for $\hat{\mathfrak{sl}}_r$, with Coxeter generators $s_i$ for $i\in I$.

There is an action of the quantum affine $\fsl_r$ algebra on a certain Fock space $\mathcal{F}$,\footnote{This corresponds to the \textit{horizontal} quantum affine algebra action discussed in \S\ref{SS:Fock}.} inside which is the basic representation. The crystal graph of $\mathcal{F}$ may be realized by the set $\Y$ of all partitions. 
The empty partition $\vn$ corresponds to the vacuum vector. The orbit of $\vn$ under the affine crystal reflection operators $s_i$ for $i\in I$, generates the extremal weight vectors in the 
subcrystal given by the basic representation; see \cite[\S 2]{MM} \cite[\S 5]{LLT}. Moreover the orbit $\Waf\cdot \vn$ is precisely the set $\cC$ of $r$-core partitions.

The affine Weyl group $\Waf\cong Q^\vee \rtimes \Wfin$ acts on $\Y$: for all $i\in I$, $s_i$ acts on $\la\in\Y$ by removing every removable cell of residue $i$ and adding every addable cell of residue $i$. We have already seen this action: using the bijection between partitions and $r$-runner abaci described above, this action has the following pleasant description. The group $\Wfin$ of permutations of the set $I$, acts on $r$-runner abaci by permuting the runners, which are indexed by $I$. We adopt the convention that the simple reflection $s_i$ for $1\le i\le r-1$, acts on $I$ by exchanging $i-1$ and $i$. We use $Q\cong Q^\vee$ for the translation lattice.
The translation element $t_{\beta^\vee}$ for $\beta=(\beta_0,\beta_1,\dotsc,\beta_{r-1})\in Q\subset \Z^I$ acts on the $r$-runner abaci by shifting the $i$-th runner by $\beta_i$ positions for $i\in I$. 
Standard formulas for the affine Weyl group such as $s_0 = t_{\theta^\vee} s_\theta$ are consistent with our conventions. 

We have $\cC = \Waf\cdot \vn$ and the stabilizer of $\vn$ in $\Waf$ is $\Wfin$. Since there is a unique translation element in every coset of $\Waf/\Wfin$, there is a bijection $Q\to \cC$ which agrees with the previously-defined map $\coremap$ with
\begin{align*}
\coremap(\beta)=t_{\beta^\vee}\cdot \vn.
\end{align*}
This bijection coincides with $\rootmap^{-1}=\kb^{-1}$ \cite[Lemma 7.5]{G}.

Let $w_0\in \Wfin$ be the long element, sending $i$ to $r-1-i$ for all $i\in I$. 
We have
\begin{align*}
\coremap(\beta)^t = \coremap(-w_0(\beta))\qquad\text{for all $\beta \in Q$.}
\end{align*}
Finally, let $w\mapsto w^*$ be the involutive group automorphism of $\Waf$ sending $s_i$ to $s_{-i}$ for all $i\in I$. For $u\in\Wfin$ and $\beta^\vee\in Q^\vee$ we have $u^*=w_0uw_0$ and $t_{\beta^\vee}^* = t_{-w_0\beta^\vee}$ which is not equal to $w_0 t_{\beta^\vee}w_0 = t_{w_0\beta^\vee}$. We have
\begin{align*}
(w \cdot \la)^t = w^* \cdot \la^t\qquad\text{for all $w\in\Waf$ and $\la\in\Y$.}
\end{align*}


\begin{ex} \label{X:root to core} Let $r=3$ and $\beta=\alb_1$. 
We have $t_{\alb_1^\vee} =s_2s_0s_2s_1$.
Computing in the crystal graph we obtain the 3-core
\ytableausetup{aligntableaux=bottom,boxsize=4mm}
\begin{align*}
s_2s_0s_2s_1\cdot \vn = s_2s_0\cdot \vn =s_2 \cdot \begin{ytableau} 0 \end{ytableau}
= \begin{ytableau}  0&2 \end{ytableau}.
\end{align*}
So $\core(\alb_1)=(2)$. Compare this with Example \ref{X:core to root}.
\end{ex}

\subsubsection{Partial orders on multipartitions}
The previously-defined action of $\Waf$ on partitions or $r$-runner abaci is given explicitly on $\Y^I\times Q$ as follows: for $\beta^\vee\in Q^\vee$ and $u\in \Wfin$, define
\begin{align}\label{E:Waf action}
t_{\beta^\vee} u \cdot (\mud, \alpha) &= (u(\mud),\beta+u(\alpha))\qquad\text{for all $(\mud,\alpha)\in \Y^I\times Q$.}
\end{align}
The bijection $\tau$ of \S \ref{SS:core quotient via maya} is $\Waf$-equivariant using the $\Waf$-actions in \eqref{E:Waf action} and \S \ref{SS:root to core}.

For $w\in \Waf$ let $\tau_w:\Y^I\to\Y$ be defined by 
\begin{align*}
\tau_w(\mud) = \tau(w^{-1}(\mud,0)).
\end{align*}
Define the partial order $\dom_w$ on $\Y^I$ as in \cite{G}\footnote{Gordon uses different conventions.}:
\begin{align}\label{E:partial order} 
\lad \dom_w \mud \Leftrightarrow \tau_w(\lad) \dom \tau_w(\mud)
\end{align}
where $\dom$ is the usual dominance order on $\Y$.

It is immediate from the definitions that (see \cite[Lemma 7.11]{G})
\begin{align*}
  \lad\dom_w\mud \Leftrightarrow u(\lad) \dom_{uw} u(\mud) \qquad\text{for all $u\in\Wfin$.}
\end{align*}
Since for $\la,\mu\in\Y$, $\la\dom\mu$ if and only if $\la^t\rdom \mu^t$, 
it follows that
\begin{align*}
\lad \dom_w \mud &\Leftrightarrow (\lad)^* \rdom_{w^*} (\mud)^* \\
&\Leftrightarrow \lad^t \rdom_{w_0w^*} \mud^t.
\end{align*}

To set some notation, let 
\begin{align*}
w = u t_{-\beta^\vee} \in \Waf \qquad\text{with $u\in \Wfin$ and $\beta^\vee\in Q^\vee$.}
\end{align*}

\begin{rem} We will need to apply $w^{-1}=t_{\beta^\vee} u^{-1}$ to the weight $\La_0-n\delta$ and want the core created by $w^{-1}\cdot\vn$ to be associated with $\beta$ rather than $-\beta$. This explains our preference for the minus sign in the translation part of $w$.
\end{rem}

Let
\begin{align*}
\mu &:= \tau_w(\mud) = \tau(u^{-1}\mud,\beta).
\end{align*}
Equivalently 
we have
\begin{align*}
\mud &= u \,\quot_r(\mu) \\
\beta &= \kb(\core_r(\mu)).
\end{align*}

\ytableausetup{boxsize=1mm}

\begin{ex} \label{X:order} 
Let $r=3$, $n=1$, $w=t_{-\alb_1^\vee}$. We consider the partial order $\dom_w$ on the multipartitions of size $1$. We have $u=\id$ and $\beta=\alb_1$.
We compute $t_{\beta^\vee}\vn = t_{\alb_1}\vn = s_2s_0s_2s_1\vn = \yb$.
The partitions with $3$-core $\yb$ with quotients of size $1$ are $(5), (2,2,1), (2,1,1,1)$.
We list them in reverse lex order, a total order which refines dominance order. The corresponding ($u$-permuted) quotients are listed.
\begin{align*}
\begin{array}{|c|||c|c|c|} \hline
\mu  & \ydiagram{5} &  \ydiagram{1,2,2} & \ydiagram{0,1,1,1,2} \\[1mm] \hline
\quot_r(\mu) & (\cd,\ya,\cd)&(\cd,\cd,\ya)&(\ya,\cd,\cd) \\ \hline
u\,\quot_r(\mu) & (\cd,\ya,\cd)&(\cd,\cd,\ya)&(\ya,\cd,\cd) \\ \hline
 \end{array}
\end{align*}
The table gives 
\begin{align*}
(\cd,\ya,\cd) \dom_w (\cd,\cd,\ya) \dom_w (\ya,\cd,\cd)\qquad \text{for $w = t_{-\alb_1^\vee}$.}
\end{align*}

Now let $w = s_2s_1 t_{-\alb_1^\vee}$. Then $u=s_2s_1$ and 
\begin{align*}
\begin{array}{|c|||c|c|c|} \hline
\mu  & \ydiagram{5} &  \ydiagram{1,2,2} & \ydiagram{0,1,1,1,2} \\[1mm] \hline
\quot_r(\mu) & (\cd,\ya,\cd)&(\cd,\cd,\ya)&(\ya,\cd,\cd) \\ \hline
u \,\quot_r(\mu) &  (\ya,\cd,\cd)&(\cd,\ya,\cd)&(\cd,\cd,\ya) \\ \hline
 \end{array}
\end{align*}
Therefore
\begin{align}\label{E:example order}
(\ya,\cd,\cd) \dom_w (\cd,\ya,\cd) \dom_w (\cd,\cd,\ya)\qquad \text{for $w = s_2s_1 t_{-\alb_1^\vee}$.}
\end{align}

\end{ex}

\subsection{Multisymmetric functions}
\label{SS:multisymmetric}

Let $\La^{\otimes I}$ be the $I$-fold tensor power of $\Lambda$ over $\K=\Q(q,t)$. For $f\in \La$ we write $f[X^{(i)}]$ to indicate an element in $\La^{\otimes I}$ with $1$ in tensor factors $j\ne i$ and $f$ in factor $i$. 

Then $\La^{\otimes I}$ is a polynomial $\K$-algebra with generators $p_k[X^{(i)}]$ for $k\in \Z_{>0}$ and $i\in I$ where $p_k$ is the power sum \cite[\S I.2]{Mac}.

We will sometimes write $f[\Xd]$ for an element $f\in \La^{\otimes r}$ to remind the reader that $f$ is symmetric separately in $r$ sets of variables $X^{(0)},X^{(1)},\dotsc,X^{(r-1)}$.

The space $\Lambda^{\otimes I}$ acquires a grading by summing the degree of each factor in a tensor product of homogeneous elements. We write $(\Lambda^{\otimes I})_n$ for the subspace of $\Lambda^{\otimes I}$ consisting of elements of degree $n$.

\subsubsection{Tensor Schur basis}
For $\lad=(\la^{(0)},\la^{(1)},\dotsc,\la^{(r-1)})\in\Y^I$ define the tensor Schur function $s_\lad=\bigotimes_{i\in I} s_{\la^{(i)}}=\prod_{i\in I} s_{\la^{(i)}}[X^{(i)}]$. The tensor Hall pairing is the pairing on $\La^{\otimes I}$ defined by
\begin{align*}
  \pair{s_\lad}{s_\mud} = \delta_{\lad,\mud}\qquad\text{for all $\lad,\mud\in \Y^I$.}
\end{align*}
The Hall pairing on $\La$ has reproducing kernel
\begin{align*}
\Omega[XY] := \prod_{i,j\ge1} (1-x_iy_j)^{-1}=\sum_{\la\in\Y} s_\la[X]s_\la[Y].
\end{align*}
Let us regard $\Xd$ and $\Yd$ as column vectors with $i$-th components $X^{(i)}$ and $Y^{(i)}$ for $i\in I$. Then the tensor Hall pairing has reproducing kernel
\begin{align*}
\sum_{\lad\in\Y^I} s_\lad[\Xd] s_\lad[\Yd] 
&= \Omega[\Yd^t \Xd] = \Omega[\sum_{i\in I} X^{(i)} Y^{(i)}].
\end{align*}

\subsubsection{Vector and matrix plethysms}\label{SS:matrix plethysm}

Let $\Mat_{I\times I}(\K)$ be the algebra of $I\times I$ matrices over $\K$. To any $M\in M_{I\times I}(\K)$ with entries $M_{ij}(q,t)$ and $g\in\La^{\otimes I}$ let $g[M \Xd]$ denote the image of $g$ under the $\K$-algebra endomorphism of  $\La^{\otimes I}$ given by
\begin{align*}
  p_k[X^{(i)}] \mapsto \sum_{j\in I} M_{ij}(q^k,t^k) p_k[X^{(j)}]\qquad\text{for all $k\in \Z_{>0}$, $i\in I$.}
\end{align*}
Let $\cP_M\in\End(\La^{\otimes I})$ be the $\K$-algebra homomorphism $g[\Xd]\mapsto g[M^t \Xd]$ where $M^t$ is the transpose of $M$. The transpose makes the construction covariantly functorial:
\begin{align*}
  \cP_{MM'} = \cP_M \circ \cP_{M'}
\qquad\text{for $M,M'\in \Mat_{I\times I}(\K).$}
\end{align*}

\begin{ex}\label{X:antipode} Let $-\id$ be the negative of the $I\times I$ identity matrix. 
We have $\cP_{-\id}(p_k[X^{(i)}]) = - p_k[X^{(i)}]$ for all $k>0$ and $i\in I$. That is, $\cP_{-\id}$ is the $I$-fold tensor product of the antipode map.
\end{ex}

\begin{ex}\label{X:permutation plethysm}
Suppose $M$ is the matrix of a permutation $u$ of the set $I$. Our convention is that this matrix sends the $i$-th standard basis column vector to the $u(i)$-th for all $i$. Then, writing $\cP_u=\cP_M$, we have
\begin{align*}
  \cP_u(p_k[X^{(i)}])=p_k[X^{u(i)}]\qquad\text{for all $k\in \Z_{>0}, i\in I$.}   
\end{align*}
Examples of permutations of $I$ that we will use are:
\begin{itemize}
    \item $\negate$ (negate): $i\mapsto \negate(i)=i^*:=r-i$;
    \item $w_0$ (reversal): $i\mapsto r-1-i$;
    \item $\chi$ (rotation): $i\mapsto i+1$.
\end{itemize}
Often we just write $u\in\Wfin$ for the automorphism of $\La^{\otimes I}$ rather than $\cP_u$. 
\end{ex}

\begin{ex}\label{X:The Matrix Plethysm} The following matrix plethysm is used in the definition of wreath Macdonald polynomials:
\begin{align*}
  \cP_{\id-q\chi^{-1}}(p_k[X^{(i)}]) = p_k[X^{(i)}] - q^k p_k[X^{(i-1)}]\qquad\text{for all $k\in \Z_{>0}$, $i\in I$.}
\end{align*}
\end{ex}

\subsection{Wreath Macdonald polynomials} 
We now define Haiman's wreath Macdonald polynomials \cite{H,H:private}. 
For each $w\in \Waf$, there is a wreath Macdonald basis $\{\tH^w_\mud \mid \mud\in\Y^I\}$ of $\La^{\otimes I}$. By definition the wreath Macdonald polynomial
$\tH^w_\mud$ is the unique element of $\La^{\otimes I}$ satisfying
\begin{align}
\label{E:q triangularity}
    \cP_{\id-q\chi^{-1}} (\tH^w_\mud) &\in \bigoplus_{\lad\dom_w \mud} \K s_\lad \\
\label{E:t triangularity}
    \cP_{\id-t^{-1}\chi^{-1}}(\tH^w_\mud) &\in \bigoplus_{\lad\rdom_w\mud} \K s_\lad \\
\label{E:normalization}
    \pair{s_n[X^{(0)}]}{\tH^w_\mud} &= 1 
\end{align}
where $n=|\mud|=\sum_{i\in I} |\mu^{(i)}|$ is the total size of $\mud$.

As in the ordinary Macdonald case ($r=1$), the conditions \eqref{E:q triangularity} and \eqref{E:t triangularity} overdetermine the $\tH^w_\mud$ and their existence therefore requires proof. Existence was conjectured by Haiman \cite{H} and first proved by Bezrukavnikov and Finkelberg \cite{BF} using powerful geometric methods originating from earlier work of Bezrukavnikov and Kaledin \cite{BK}; a second proof of existence was given later by Losev \cite{L2}. We formulate these results in more detail in \S\ref{SS:wreath Haiman} below. 

It is not difficult to show that the conditions above uniquely determine the $\tH^w_\mud$. Moreover, using properties the inner product $\pair{\cdot}{\cdot}_{q,t}$ introduced below (see \S\ref{SS:ortho-tH}), one can show that the leading terms in \eqref{E:q triangularity} and \eqref{E:t triangularity} are nonzero.

\begin{rem}
For $r=1$, $\mud$ is a $1$-tuple of partitions $(\mu)$ where $\mu$ is a partition, 
the affine Weyl group is the identity, and $\dom_\id$ becomes dominance order.
The matrix plethysm $\cP_{\id-q\chi^{-1}}$ becomes the plethysm $f\mapsto f[(1-q)X]$ so that
\eqref{E:q triangularity} matches \eqref{E:tH-def-q},
$\cP_{\id-t^{-1}\chi^{-1}}$ becomes the plethysm $f\mapsto f[(1-t^{-1})X]$ so that 
\eqref{E:t triangularity} matches \eqref{E:tH-def-t-inv}, and 
\eqref{E:normalization} matches \eqref{E:tH-def-normalization}.
Hence the wreath Macdonald polynomial $\tH^{\id}_{(\mu)}$ for $r=1$ is the usual 
modified Macdonald polynomial $\tH_\mu$.
\end{rem}

\begin{ex} \label{X:wreath Mac triangularity}
Using the order from Example \ref{X:order} with $w = s_2s_1 t_{-\alb_1^\vee}$ we give a few wreath Macdonald polynomials and verify some of their triangularity properties using the order \eqref{E:example order}.
The corresponding wreath Macdonalds are given explicitly and the nontrivial triangularity conditions are verified.
\begin{align*}
\tH^w_{(\ya,\cd,\cd)} &= s_1[X^{(0)}] + q^2 s_1[X^{(1)}] + q s_1[X^{(2)}] \\
\cP_{\id-q \chi^{-1}} \tH^w_{(\cd,\cd,\ya)} &= (1-q^3) s_1[X^{(0)}]
\end{align*}
\begin{align*}
\tH^w_{(\cd,\ya,\cd)} &= s_1[X^{(0)}] + t s_1[X^{(1)}] + q s_1[X^{(2)}] \\
\cP_{\id-q \chi^{-1}}\tH^w_{(\cd,\ya,\cd)}  &= (1-qt) s_1[X^{(0)}] + (t-q^2) s_1[X^{(1)}] \\
\cP_{\id-t^{-1} \chi^{-1}}\tH^w_{(\ya,\cd,\cd)}  &=  (t-qt^{-1}) s_1[X^{(1)}] + (q-t^{-1})s_1[X^{(2)}]
\end{align*}
\begin{align*}
\tH^w_{(\cd,\cd,\ya)} &= s_1[X^{(0)}] + t s_1[X^{(1)}] + t^2 s_1[X^{(2)}] \\
\cP_{\id-t^{-1} \chi^{-1}}\tH^w_{(\cd,\cd,\ya)}  &=(t^2-t^{-1}) s_1[X^{(2)}]
\end{align*}
\end{ex}

\subsection{Wreath Macdonald-Kostka coefficients}
Define the wreath Macdonald-Kostka coefficients $\tK^w_{\lad,\mud}(q,t)\in \K$ by
\begin{align*}
  \tH^w_\mud = \sum_{\lad} \tK^w_{\lad,\mud}(q,t) s_{\lad}.
\end{align*}
The wreath analog of Haiman's Theorem (Theorem~\ref{T:wreath Haiman} below) implies the Laurent polynomiality and positivity of wreath Macdonald-Kostka coefficients:

\begin{thm}[\cite{BF}]\label{T:wreath pos}
$\tK^w_{\lad,\mud}(q,t)\in \Z_{\ge0}[q^{\pm 1},t^{\pm 1}]$.
\end{thm}

Another immediate consequence is the following dimension formula for the 
$\Gamma_n$-isotypic components of fibers of the wreath Procesi bundle over torus fixed points.

\begin{thm} 
$\tK^w_{\lad,\mud}(1,1)$ is the number of standard multitableaux $f^{\lad}$ of shape $\lad$. 
\end{thm}
\begin{proof} 
Theorem~\ref{T:wreath Haiman} below realizes $\tH^w_\mud$ as the $\Gamma_n$-Frobenius character of a wreath Procesi bundle fiber $P^w_n|_{I_\mud^w}$, where $n=|\mud|$. The important point here is that, after forgetting the $\T$-grading, $P^w_n|_{I_\mud^w}$ affords a copy of the regular representation of $\Gamma_n$ and its irreducible decomposition is given by $\sum_{|\lad|=n} \tK^w_{\lad,\mud}(1,1) \chi^\lad$ where $\chi^\lad$ denotes the $\Gamma_n$-irreducible representation with Frobenius character $s_{\lad}$ (see \eqref{E:wreath Frob}). Taking the coefficient of $\chi^\lad$ and using the fact that the multiplicity of an irreducible representation in the regular representation is equal to its dimension, we obtain $\tK^w_{\lad,\mud}(1,1)=\dim \chi^\lad$. The dimension is given by multitableaux.
\end{proof}

\begin{rem}
If $|\la^{(i)}|=n_i$, $n=\sum_{i\in I} n_i$, and $f^\la$ is the the number of standard tableaux of shape $\la\in\Y$, then 
\begin{align*}
    f^{\lad} = \binom{n}{n_0,n_1,\dotsc,n_{r-1}} \prod_{i\in I} f^{\la^{(i)}}.
\end{align*}
\end{rem}

The notion of $\Gamma$-degree is defined in \S \ref{SS:McKay and cyclic quiver}.
The following says that only certain Laurent monomials in $q$ and $t$ may occur in a given wreath Macdonald-Kostka coefficient.

\begin{prop}[\cite{H:private}]\label{P:Gamma grading}
The wreath Macdonald-Kostka coefficient $\tK^w_{\lad,\mud}(q,t)$ is homogeneous of $\Gamma$-degree
$\sum_{i\in I} |\la^{(i)}| i$.
\end{prop}

\begin{ex} We see that Theorem \ref{T:wreath pos} and Proposition \ref{P:Gamma grading} holds for the tensor Schur coefficients of the wreath Macdonald polynomials appearing in Example \ref{X:wreath Mac triangularity}.
\end{ex}

\subsection{Wreath Macdonald symmetries}
The wreath Macdonald polynomials satisfy a number of beautiful symmetries.
Some of these symmetries generalize properties of modified Macdonald polynomials for the $r=1$ case, while others are only visible in the wreath setting when $r\ge 2$.

For $f,g\in \La^{\otimes I}$ let $f\equiv g$ mean that 
$g=cf$ for $c\in \K^\times$. In all of the following cases of wreath Macdonald symmetries presented below, the constant of proportionality is a Laurent monomial which may easily be deduced 
using the normalization condition \eqref{E:normalization}. The following symmetries can be readily proved using the definitions.

\subsubsection{Symmetry swapping $q$ and $t$}
Let ``$\swap$" be the $\Q$-algebra automorphism of $\K$ given by exchanging $q$ and $t$. It extends to a $\Q$-algebra automorphism of $\La^{\otimes I}$. 
The combination of $\swap$ and $\negate$ is compatible with $\Gamma$-degree 
(see \S \ref{SS:McKay and cyclic quiver})
in the sense that $\resGamma(\swap(f)) = \negate (\resGamma(f))$ for all $\Gamma$-homogeneous $f\in R(\T)$.

\begin{prop} \label{P:swapping symmetry}
For all $w\in \Waf$ and $\mud\in \Y^I$,
\begin{align*}
  \swap\,\negate\, \tH^w_{\mud} &= \tH^{w^*}_{\mud^*}.
\end{align*} 
\end{prop}

\begin{rem} For $r=1$, $\tH_\mu[X;t,q] = \tH_{\mu^t}[X;q,t]$.
\end{rem}

\subsubsection{Inversion symmetry}
\label{SS:inv}
Let ``$\inv$" be the $\Q$-algebra automorphism of $\K$ given by sending $q$ and $t$ to their 
reciprocals. It extends to a $\Q$-algebra automorphism of $\La^{\otimes I}$ by acting on coefficients.
Let $\omega$ be the $\K$-algebra involutive automorphism of $\La$ given by 
$s_\la\mapsto s_{\la^t}$. Consider the involutive $\K$-algebra automorphism of $\La^{\otimes I}$ given by the $I$-th tensor power of $\omega$;
by abuse of notation we also call this $\omega$. 
For $f\in \La^{\otimes I}$ homogeneous of degree $d$, we have
\begin{align*}
  \omega(f) = (-1)^d \cP_{-\id}(f).
\end{align*}
It follows that $\omega$ commutes with matrix plethysms:
\begin{align*}
   \omega \,\cP_M = \cP_M \,\omega\qquad\text{for all $M\in M_{I\times I}(\K)$.}
\end{align*}

Define the operator $\downarrow$ on $\La^{\otimes I}$ by
\begin{align*}
\downarrow \,= \inv\, \negate\, \omega,
\end{align*}
in analogy with the operator $\downarrow$ of \cite[Prop. 1.1]{GHT}. It is an involution.

\begin{prop} \label{P:inversion symmetry}
For all $w\in \Waf$ and $\mud\in \Y^I$,
\begin{align*}
   \downarrow \tH^w_{\mud} &\equiv \tH^{w_0 w}_{w_0\mud}.
\end{align*} 
\end{prop}

\begin{rem} For $r=1$, $q^{n(\mu^t)} t^{n(\mu)} \inv \,\omega\,\tH_\mu[X;q,t]= \tH_{\mu}[X;q,t]$.
\end{rem}

\begin{rem} Let $c$ be the constant such that 
$\downarrow \tH^w_\mud = c \tH^{w_0w}_{w_0\mud}$. With $n=|\mud|$ we have
\begin{align*}
c &= \pair{s_n[X^{(0)}]}{\downarrow \tH^w_\mud} \\
&= \inv\, \pair{e_n[X^{(0)}]}{\tH^w_\mud} \\
&= \inv\, e^{(0),w}_\mud, 
\end{align*}
where $e^{(0),w}_\mud$ is the eigenvalue of Haiman's operator $\nabla^{(0)}$ on $\tH^w_\mud$; see \S \ref{SS:wreath nabla}. It is a Laurent monomial so  $\inv(e^{(0),w}_\mud)=1/e^{(0),w}_\mud$. It follows that 
\begin{align}
\downarrow (\nabla^{(0)})^{-1} \tH^w_\mud &= \tH^{w_0w}_{w_0\mud}.
\end{align}
We deduce that $\downarrow (\nabla^{(0)})^{-1}$, $\nabla^{(0)}\downarrow$, and $\downarrow \nabla^{(0)}$ are all involutions. In particular $\downarrow \nabla^{(0)} \downarrow = (\nabla^{(0)})^{-1}$; see \cite[Prop. 1.1(b)]{GHT}.
\end{rem}

\subsubsection{Rotational symmetry}
The following is peculiar to the wreath setting for $r\ge 2$.
The quotient construction in \cite[Ex. I.1.8]{Mac}
depends on a sufficiently large integer $m$. Increasing $m$ has the effect of rotating the quotient. This leads to the following symmetry.

\begin{thm} For all $w\in\Waf$ and $\mud\in\Y^I$,
\begin{align*}
  \chi\tH^w_\mud \equiv \tH^{\chi w}_{\chi \mud}.
\end{align*}
\end{thm}

\subsubsection{Forward rotation in matrix plethysms}
\label{SS:forward wreath Mac} 
Another variant of wreath Macdonald polynomials that only appears for $r\ge2$ is the following. Let $\hH^w_\mud$ be the version of wreath Macdonald polynomials defined by replacing $\chi^{-1}$ by $\chi$ in \eqref{E:q triangularity} and \eqref{E:t triangularity}.

\begin{prop}\label{P:backward H} For all $w\in\Waf$ and $\mud\in\Y^I$,
\begin{align*}
  \hH^w_\mud \equiv w_0 \tH^{w_0w}_{w_0\mud}.
\end{align*}
\end{prop}

Note how closely related this is with the inversion symmetry.

\subsubsection{Dualizing the partial order}
For $w\in \Waf$ let $\domop_w$ be the partial order dual to $\dom_w$ on $\Y^I$: $\lad \domop_w \mud$ if and only if $\mud \dom_w \lad$.
Let $\tHopp^w_\mud$ be the family of symmetric functions obtained by using the definition of wreath Macdonald polynomial with $\domop_w$ instead of $\dom_w$. Replacing $\dom_w$ by $\domop_w$ and sending $(q,t)$ to $(t^{-1},q^{-1})$ leaves the definition of wreath Macdonald polynomial invariant. This implies the following.

\begin{prop} \label{P:dual partial order} For all $w\in \Waf$ and $\mud\in\Y^I$,
\begin{align*}
\tHopp^w_\mud = \inv \,\swap\, \tH^w_\mud.
\end{align*}
\end{prop}

\subsection{Orthogonality of $\tH^w_\mud$}
\label{SS:ortho-tH}

Define $A\in \Mat_{I\times I}(\K)$ by 
\begin{align}
A = (\id - q \chi^{-1})^{-1} (\id - t \chi)^{-1} \negate,
\end{align}
where $\negate$ is regarded as a permutation matrix. Let $\swap$ and $\inv$ act on matrices by acting on their entries.

\begin{lem}
\begin{align*}
A^t &= A\\
\swap \,\inv (A) &= qt \, A \\
A^{-1} \inv\,\negate &= qt\,\inv\,\negate \,A^{-1}.
\end{align*}
\end{lem}

Let $\pairqt{\cdot}{\cdot}$ be the pairing on $\La^{\otimes I}$ defined by
\begin{align*}
\pairqt{f[\Xd]}{g[\Xd]} &= \pair{f[\Xd]}{g[A^{-1} \Xd]}.
\end{align*}

The following are straightforward consequences of the definition of wreath Macdonald polynomials and the pairing $\pair{\cdot}{\cdot}_{q,t}$.

\begin{prop} \label{P:qt pairing}
\begin{enumerate}
\item 
The reproducing kernel for $\pairqt{\cdot}{\cdot}$ is 
\begin{align*}
\Omega_{q,t} = \Omega[\Yd^t A X].
\end{align*}
\item For all $f,g\in \La^{\otimes I}$, $\pairqt{f}{g}=\pairqt{g}{f}$.
\item $s_\lad[A \Xd]$ and $s_\lad[\Xd]$ are $\pairqt{\cdot}{\cdot}$-dual bases.
\item 
$\cP_{(\id-q\chi^{-1})^{-1}} s_\lad[\Xd]=s_\lad[(\id-q\chi)^{-1}\Xd]$ and 
$\cP_{\negate\,(\id-t\chi^{-1})^{-1}} s_\lad[\Xd]=s_\lad[(1-t\chi)^{-1}\negate\,\Xd]$ are $\pairqt{\cdot}{\cdot}$-dual bases.
\end{enumerate}
\end{prop}

\begin{prop} \label{P:H orthogonal}
\begin{align*}
\pairqt{\tH^w_\mud}{\inv\,\negate\,\tH^w_\nud} = 0 \qquad\text{if $\mud\ne\nud$.}
\end{align*}
\end{prop}

For $w\in \Waf$ and $\mud\in \Y^I$ define
\begin{align}\label{E:H pairing constant}
  b^w_\mud(q,t) &= \pairqt{\tH^w_\mud}{\inv\,\negate\,\tH^w_\nud}
\end{align}
Since $\inv\,\negate$ is an automorphism, $\{\inv\,\negate\,\tH^w_\mud\mid \mud\in \Y^I\}$ is a basis of 
$\La^{\otimes I}$. We conclude that $b^w_\mud(q,t)\ne0$. It follows that 
$(b^w_\mud(q,t))^{-1} \inv\,\negate\,\tH^w_\mud$ is the $\pairqt{\cdot}{\cdot}$-dual basis to 
$\tH^w_\mud$.

\begin{thm}\label{T:H pairing constant}
The value of the pairing in \eqref{E:H pairing constant} is
\begin{align}\label{E:H pairing constant formula}
  b^w_\mud(q,t) &= b_{\tau(w^{-1} (\mud,0))}(q,t) \qquad\text{where} \\
\notag
  b_\mu(q,t) &= \prod_{\substack{s\in \mu \\ h_\mu(s)\equiv\, 0 \,\text{mod}\, r}} 
  (1 - q^{1+a_\mu(s)}t^{-l_\mu(s)})(1-t^{1+l_\mu(s)}q^{-a_\mu(s)}).
\end{align}
\end{thm}

\begin{rem}
\begin{enumerate}
    \item We give a proof of Theorem~\ref{T:H pairing constant} in \S\ref{SS:norms} using the wreath analog of Theorem~\ref{T:wreath Haiman}.
    \item In terms of the locus of the Hilbert scheme fixed under $\Gamma\cong\Z/r\Z$ (see \S\ref{S:quiver varieties}), one has $b_\mu(q,t)^{-1} = \mathrm{ch}_{\T}\, \mathrm{Sym}\, T^*\Hilb_{|\mu|}(\C^2)^\Gamma|_{I_\mu}$.
   \item $(qt)^n b_\mud^w(q^{-1},t^{-1}) = b_\mud^w(q,t)$ for $n=|\mud|$.
\end{enumerate}
\end{rem}

\subsection{Wreath Macdonald $P$ polynomials}

As known to Haiman \cite{H:private} there is a wreath analogue of the Macdonald $P$ polynomial.
The analogies in this subsection are not perfect in comparison to the $r=1$ case; the variations will be pointed out.

Define $B\in \Mat_{I\times I}(\K)$ by
\begin{align*}
B = (\id- q \chi^{-1})^{-1} (\id - t^{-1}\chi^{-1})\,\negate.
\end{align*}
Define the pairing $\pairP{\cdot}{\cdot}$ on $\La^{\otimes I}$ by
\begin{align*}
    \pairP{f[\Xd]}{g[\Xd]} &= \pair{f[\Xd]}{g[B^{-1}\Xd]} 
\end{align*}
It has reproducing kernel
\begin{align*}
    \Omega^P_{q,t} &= \Omega[\Yd^t B \Xd].
\end{align*}

\begin{lem}\label{L:Omega qt to P}
\begin{align*}
\Omega^P_{q,t} &= \cP^X_{\id-t^{-1}\chi^{-1}} \cP^Y_{\id-t\chi} \Omega_{q,t}.
\end{align*}
\end{lem}

\begin{prop}\label{P:wreathP}
For $w\in\Waf$ there is a unique basis $\{P^w_\mud\mid \mud\in\Y^I\}$ of $\La^{\otimes I}$ called the wreath Macdonald $P$ polynomials, such that
\begin{align*}
&P^w_\mud \in s_\mud + \bigoplus_{\lad\rdomstrict_w \mud}  \K s_\lad \\
&\pairP{P^w_\mud}{\inv\,\negate\,P^w_\nud} = 0 \qquad \text{for $\mud\ne\nud$.}
\end{align*}
\end{prop}

For $w\in\Waf$ and $\mud\in\Y^I$ define the wreath Macdonald $J$ polynomial by
\begin{align*}
J^w_\mud = \cP_{\id - t^{-1}\chi^{-1}} \tH^w_\mud.
\end{align*}

It follows directly from the definitions that there is a nonzero constant
$c^w_\mud\in\K$ such that 
\begin{align*}
  J^w_\mud \in c^w_\mud s_\mud + \bigoplus_{\lad \rdomstrict \mud} \K s_\lad.
\end{align*}
Then
\begin{align*}
  P^w_\mud = (c^w_\mud)^{-1} J^w_\mud.
\end{align*}

\begin{conj} \label{CJ:P pairing} For $\mu=\tau(w^{-1}(\mud,0))$
 \begin{align*}
    \pairP{P^w_\mud}{\inv\,\negate\,P^w_\mud} = \prod_{\substack{s\in\mu\\ h(s)\equiv 0\,\text{mod}\,\, r}} \dfrac{1-q^{1+a(s)}t^{-l(s)}}{1-q^{a(s)} t^{-1-l(s)}}.
 \end{align*}
\end{conj}

\begin{rem} In $r=1$ Macdonald theory, to go from $\tH_\mu$ to $J_\mu$ one first passes to $H_\mu$ by inverting $t$ (and rescaling) and then applies the plethystic automorphism. In the above definition there is no scaling. For $r=1$ we have $P^{\id}_{(\mu)} = P_\mu|_{t\mapsto t^{-1}}$.

One obtains another interesting basis of $\La^{\otimes I}$ by using $q$ instead of $t^{-1}$ in the definition of $J^w_\mud$.
\end{rem}

\begin{conj}\label{CJ:J to P coefficient}
Let $f^w_\mud\in\K$ be the sum of terms in the ``diagonal" wreath Macdonald-Kostka polynomial $\tK^w_{\mud,\mud}(q,t)$ with maximum power of $t$. Then $f^w_\mud$ is a single Laurent monomial in $q$ and $t$ and
\begin{align*}
c^w_\mud = f^w_\mud \prod_{\substack{s\in \mu \\ h_\mu(s)\equiv 0\, \text{mod}\,\, r}}
(1-q^{a(s)}t^{-1-l(s))}),
\end{align*}
where $\mu=\tau(w^{-1} (\mud,0))$.
\end{conj}


\begin{prop} $P^w_\mud[\Xd;q,q^{-1}] = s_\mud$.
\end{prop}

\section{Nakajima quiver varieties}
\label{S:quiver varieties}

In this section, we explain the geometric realization of wreath Macdonald polynomials, as formulated conjecturally by Haiman \cite{H} and later established by Bezrukavnikov and Finkelberg \cite{BF} and Losev \cite{L1,L2}. We use this realization to prove the norm formula \eqref{E:H pairing constant formula} and to construct wreath analogs of the $\nabla$ operator.

\subsection{The cyclic group $\Gamma$}
\label{SS:McKay and cyclic quiver}

Let $\Gamma\subset \C^\times$ be the cyclic group of $r$-th roots of unity; we fix the generator $\zeta = \exp(2\pi \sqrt{-1}/r)$ for $\Gamma$. We shall use the embedding
$\Gamma\to \T \subset GL(2,\C)$ determined by $\zeta\mapsto \diag(\zeta^{-1},\zeta)$, where $\T$ is the torus of \S\ref{SS:Haiman} regarded as a subgroup of diagonal matrices in $GL(2,\C)$ in the standard way, via $(t_1,t_2)\mapsto\diag(t_1,t_2)$.

The representation ring $R(\Gamma) = \Z[\chi]$ is generated by the character $\chi$ sending $\zeta = \exp(2\pi \sqrt{-1}/r)$ to itself. $R(\Gamma)$ has basis $\chi^i$ for $i\in I=\Z/r\Z$. The McKay correspondence asserts a bijection between the affine Dynkin node set $I=\Z/r\Z$ (vertex set of cyclic quiver) and the irreducible characters of $\Gamma$. It is given by $i\mapsto\chi^i$ for $i\in I$.\footnote{The reason that we used the notation $\chi$ for the rotational permutation in $\Wfin$, was due to the fact that multiplication by $\chi$ acts by rotation on the ordered basis $\{\chi^0,\chi^1,\dotsc,\chi^{r-1}\}$ of $R(\Gamma)$.}

Recall that we have made the identification $R(\T)=\Z[q^{\pm 1}, t^{\pm 1}]$, with the standard coordinate functions $x(a_1,a_2)=a_1$ and $y(a_1,a_2)=a_2$ on $\C^2$ having $\T$-weights $t$ and $q$, respectively.
Let us call $\resGamma$ the restriction map $R(\T) \to R(\Gamma)$. 
We have $\resGamma(q)=\chi^{-1}$ and $\resGamma(t)=\chi$. 
We say that $f\in R(\T)$ is homogeneous of $\Gamma$-degree $i\in\Z/r\Z$ if 
$\resGamma(f)\in \Z \chi^i$.

It is sometimes enlightening to include the $\Gamma$-grading of $R(\T)$ explicitly in formulas. This is accomplished by applying the ring homomorphism $R(\T) \to R(\T\times \Gamma)$
given by $q\mapsto q\chi^{-1}$ and $t\mapsto t\chi$.

\subsection{Quiver varieties}
Let $V(\La_0)$ be the irreducible integrable highest weight $\hat{\fsl}_r$-module of highest weight $\La_0$, i.e., the basic representation. By \cite{Nak1}, the weight spaces of $V(\La_0)$ may be realized by the top degree homology of various Nakajima quiver varieties $\fM$ for the cyclic quiver with one-dimensional framing space. These quiver varieties provide the geometric setting for wreath Macdonald theory \cite{H}. Their equivariant $K$-theory groups $K_\T^*(\fM)$ realize the weight spaces in an irreducible quantum toroidal algebra module which we shall consider in the next section (see Remark~\ref{R:Fock geom} below). 

\subsubsection{Affine root lattice and dimension vectors for cyclic quiver}
In this article we always use a one-dimensional framing at node $0$ only.
Given $\bv=(v_i\mid i\in I)\in \Z_{\ge0}^I$, let $V_i$ be an $v_i$-dimensional vector space 
over $\C$. This dimension vector is written $\sum_{i\in I} v_i\alpha_i\in Q_\af^+$.
Since the null root equals $\delta=\sum_{i\in I} \alpha_i$, the expression $n\delta$ will sometimes be used for the vector having dimension $n$ at each quiver vertex. 

\subsubsection{Stability conditions}
Let $G=\prod_{i\in I} GL(V_i)$ and let $\det_i$ be the determinant character on $GL(V_i)$ for $i\in I$. For each $\theta=(\theta_i\mid i\in I)\in \Z^I$, define the following character of $G$: $\chi^\theta = \prod_{i\in I} \det_i^{\theta_i}$.

\subsubsection{Quiver variety}
Let $\fM_\theta(\bv)$ be the Nakajima quiver variety for the cyclic quiver with 
one-dimensional framing at the node $0$, quiver dimension vector $\bv$, 
and stability condition given by $\chi^\theta$ for $\theta\in\Z^I$.
For the standard stability condition $\theta=(1,1,\dotsc,1)\in\Z^I$ the top degree homology of $\fM_\theta(\bv)$ realizes the $(\La_0-\sum_{i\in I} v_i \alpha_i)$-weight space in $V(\La_0)$ \cite{Nak1}.

\subsubsection{Stability alcoves and the affine Weyl group}
We consider how the quiver varieties $\fM_\theta(\bv)$ change as the stability condition $\theta$ changes. The variety $\fM_{k\theta}(\bv)$ is invariant up to isomorphism for $k\in \Z_{>0}$. Therefore we may work with a parameter $\theta\in \Theta=\Q^I$, modulo scaling by $\Q_{>0}$. Let $X^\af_\Q = \Q\delta \oplus \bigoplus_{i\in I} \Q \La_i$ be the affine weight lattice tensored with $\Q$, with null root $\delta$ and fundamental weights $\La_i$. 
We use a linear isomorphism $\Theta \cong X^\af_\Q/\Q\delta$ defined by
$\sum_{i\in I} \theta_i \epsilon_i \mapsto \sum_{i\in I} \theta_i \La_i$. 
The level of $\theta\in\Theta$ is by definition $\pair{\theta}{c}=\sum_{i\in I}\theta_i$, where $c$ is the canonical central element. Let $\Theta_\ell$ be the elements of $\Theta$ of level $\ell\in \Q$.
Let us assume $\theta\ne0$. By scaling we can assume $\theta\in \Theta_{\pm1}$.

Let $X_\Q = \bigoplus_{i\in I\setminus\{0\}} \Q \omega_i$ be the finite weight lattice tensored with $\Q$ where $\omega_i$ is the $i$-th finite fundamental weight.
Let $\cl: X^\af_\Q \to X_\Q$ be the restriction map. It has kernel $\Q\delta\oplus \Q\La_0$. Consider the lift $\omega_i\mapsto\La_i-\La_0$ for $i\in I\setminus\{0\}$. Via this lift we have
$\Theta_{\pm1} \cong \pm \La_0 \oplus X_\Q$.

Define $\theta_{\pm} = \pm (1/r) \sum_{i\in I} \La_i$. Note that the vector $\theta_\pm$ has level $\pm 1$ and that $\theta_+$ is a $\Q_{>0}$-multiple of the standard stability condition.

There is a natural action of $\Waf$ on $\Theta_\ell$ for each $\ell\in \Q$.
For every $i\in I$ and $m\in \Z$ let $L_{i,m}$ be the hyperplane in $\Q^I$ consisting of the vectors $\beta$ such that $\pair{\beta}{\alpha_i^\vee}=m$. The affine hyperplane arrangement is the set of hyperplanes of the form $w(L_{i,m})$ for $w\in \Waf$, $i\in I$, and $m\in\Z$.
Let $\Theta_{\pm1}^\reg$ denote the subset of $\Theta_{\pm1}$ that avoids all the hyperplanes in the affine hyperplane arrangement. The affine Weyl group $\Waf$ acts on $\Theta_{\pm1}^\reg$, permuting the connected (for the usual topology on Euclidean space) components of $\Theta_{\pm1}^\reg$, which are called alcoves. There are exactly two orbits of alcoves, $\Waf \cdot\theta_+$ and $\Waf\cdot \theta_-$.
By \cite{G} the quiver variety $\fM_\theta(\bv)$ is smooth and stays the same up to isomorphism as $\theta$ varies over an alcove. Therefore we have a stability condition for every pair $(w,+)$ or $(w,-)$ for $w\in\Waf$.

\begin{rem} Gordon \cite{G} uses two parametrizations of stability condition. One is 
$(h,H_1,\dotsc,H_{r-1})\in \Q^r$, which in our notation is $-h\La_0 + \sum_{i=1}^{r-1} H_i\omega_i$.
The other is $\theta\in \Q^\ell$ which we have identified with $\sum_{i\in I} \theta_i\La_i$.
\end{rem}

\subsubsection{Weight spaces}

The nonzero weight spaces in $V(\La_0)$ have weights of the form $t_{\beta^\vee}\La_0-n\delta$ for 
$(\beta,n)\in Q\times \Z_{\ge0}$ \footnote{We shall use the type A isomorphism $Q\mapsto Q^\vee$ sending $\alb_i\mapsto \alb_i^\vee$ for $i\in I\setminus\{0\}$ without further mention.}.
Subtracting this weight from the highest weight we get an element 
$\La_0-t_{\beta^\vee}\La_0+n\delta \in \bigoplus_{i\in I} \Z_{\ge0}\alpha_i$. 
By \cite[Eq. (6.5.2)]{Kac} we have 
\begin{align*}
  t_{\beta^\vee} \cdot \La_0 = \La_0 + \beta - \dfrac{(\beta,\beta)}{2} \delta.
\end{align*}

Let $\fM_{\beta,n} = \fM_{\theta_+}(\La_0-t_{\beta^\vee}\La_0+n\delta) = \fM_{\theta_+}(-\beta+((1/2)(\beta,\beta)+n)\delta).$

\subsubsection{$\Gamma$-fixed locus in $\Hilb_N$}
It is well-known (see, e.g., \cite[Theorem 12.13]{Kirbk}) that $\fM_{\beta,n}$ is isomorphic to the irreducible component of the $\Gamma$-fixed locus $\Hilb_N^\Gamma=\Hilb_N(\C^2)^\Gamma$ consisting of $\Gamma$-invariant ideals $I\subset \C[x,y]$ such that the equality $\Gamma$-characters $\mathrm{ch}_\Gamma \C[x,y]/I = \sum_{i\in I} v_i \chi^i$ holds. Here 
$\bv$ is determined by $t_{\beta^\vee}\cdot\Lambda_0-n\delta=\Lambda_0-\sum_{i\in I}v_i \alpha_i$, $N=\sum_{i\in I}v_i$, and $\Gamma$ acts on $\Hilb_N$ via the embedding $\Gamma\to\T$. In terms of the $r$-core $\gamma = \kb^{-1}(\beta)\in\cC$, one has $N = |\gamma| + rn$. Moreover, this isomorphism is $\T$-equivariant for a natural $\T$-action on $\fM_{\beta,n}$. The $\T$-fixed points of $\fM_{\beta,n}$ correspond precisely to the monomial ideals $I_\lambda$ such that $|\lambda|=N$ and $\core(\lambda)=\gamma$ \cite[Prop. 7.2.18]{H}.

\subsubsection{Correspondence between $\T$-fixed points}
For any $w=ut_{-\beta^\vee}\in\Waf$, where $u\in \Wfin$ and $\beta\in Q$, the composition of Nakajima reflection functors $R_{w^{-1}}$ \cite{Nak3} defines an isomorphism
\begin{align}\label{E:rf-iso}
\fM_{w\theta_+}(n\delta) \to \fM_{\beta,n}.
\end{align}
By \cite{G} (see also \cite{Nag1,P}), the $\T$-fixed points $I_{\mud}^w$ of $\fM_{w\theta_+}(n\delta)$ are naturally labeled by the set $\Y^I_n$ of $r$-multipartitions $\mud$ of total size $n$, so that $I_{\mud}^w$ corresponds to the $\T$-fixed point $I_{\tau(w^{-1}(\mud,0))}$ in $\Hilb_N^\Gamma$, where $\tau: \Y^I\times Q \to \Y$ is the bijection of \S \ref{SS:core-quot}.


\begin{rem} 
The combinatorial partial order $\dom_w$ on multipartitions defined in \eqref{E:partial order} refines 
the geometric partial order defined by attracting sets with respect to the hyperbolic $\mathbb{C}^*$-action on $\fM_{\beta,n}$ \cite{G} (see also \cite[\S 2.3]{BF} and \cite{P}). This $\mathbb{C}^*$-action corresponds to the subgroup of matrices $\diag(t_1,t_1^{-1})\in\T$ acting on $\Hilb_N^\Gamma$.
\end{rem}


\subsubsection{Small tautological bundle and bigraded reflection functors}
The following is due to Haiman \cite{H,H:private}.
Let $M^w_n$ be the bundle of tautological quiver data over $\fM_{w\theta_+}(n\delta)$. We call this the small tautological bundle; it has rank $rn$. Using the isomorphism \eqref{E:rf-iso}, we regard $M^w_n$ as a vector bundle on $\fM_{\beta,n}$. 
In contrast the (``big") tautological bundle $\Taut_N$ over $\fM_{\beta,n}$, has rank $N=rn+|\kb^{-1}(\beta)|$.

For $\mud\in \Y^I_n$ define
\begin{align*}
  \cB^w_\mud(q,t,\chi) = \mathrm{ch}_{\Gamma\times\T} M^w_n|_{I_{\mud}^w}\in R(\Gamma\times\T).
\end{align*}

We use $R(\Gamma\times\T)= \Z[q^{\pm1},t^{\pm1}][\chi]$; it has $R(\T)$-basis $\chi^i$ for $i\in I$.
For $i\in I$ and $f\in R(\Gamma\times\T)$ let $[\chi^i]f$ denote the coefficient of $\chi^i$ in $f$.

The following operator computes the change in the $\Gamma\times\T$-grading of the fibers of tautological quiver data bundles over $\T$-fixed points under the application of a Nakajima reflection functor $R_i$.
For $i\in I$, define the operator $R_i^*$ on $R(\Gamma\times\T)$ by
\begin{align*}
  R_i^*(f) = f + \left(q^{-1}t^{-1}\delta_{i0} - [\chi^i] (\id-q^{-1}\chi)(\id-t^{-1}\chi^{-1}) f\right) \chi^i\quad\text{for $i\in I$.}
\end{align*}

If $f=\sum_{i\in I} f^{(i)} \chi^i$ with $f^{(i)}\in R(\T)$ and $g = R_i^*(f)=\sum_{i\in I} g^{(i)} \chi^i$ then for $j\in I$ we have
\begin{align*}
  g^{(j)} &= \begin{cases}
      f^{(j)} & \text{if $j\ne i$} \\
      q^{-1}t^{-1} \delta_{i0} - q^{-1}t^{-1} f^{(i)} + t^{-1} f^{(i+1)}+q^{-1} f^{(i-1)} & \text{if $j=i$.}
  \end{cases}
\end{align*}

\begin{thm}[\cite{H:private}]
\begin{enumerate}
    \item $R_i^* R_{i+1}^* R_i^* = R_{i+1}^* R_i^* R_{i+1}^*$ for all $i\in I$.
    \item $(R_i^*-1)(R_i^*+q^{-1}t^{-1})=0$ for all $i\in I$.
    \item Let $w\in \Waf$ and $n\in\Z_{\ge0}$. Then for all $\mud\in \Y^I$ 
\begin{align*}
  \cB^w_\mud(q,t,\chi) &= R^*_{i_1} R^*_{i_2}\dotsm R^*_{i_\ell} B_\mu(q\chi^{-1},t\chi)
\end{align*}
where $w=s_{i_1}s_{i_2}\dotsm s_{i_\ell}$ is a reduced decomposition, $\mu=\tau(w^{-1}(\mud,0))$, 
and $B_\mu(q,t)$ is defined in \eqref{E:Bmu}.
\end{enumerate}
\end{thm}

\begin{rem} 
The operators $R_i^*$ define a representation of the affine Hecke algebra.
\end{rem}

\begin{ex} \label{X:quiver data}
Let $r=3$, $w= t_{-\alb_2^\vee} = s_2 s_1s_0s_1$, and $\mud = (\ya,\cd,\ya)$.
We have $u=\id$, $\beta=\alb_2$, 
$t_{\beta^\vee}\cdot\vn=t_{\alb_2^\vee}\cdot\vn=s_1s_0s_1s_2\cdot\vn=\yaa$ and $u^{-1}\mud=(\ya,\cd,\ya)$.
The partition with 3-core $\yaa$ and 3-quotient $(\ya,\cd,\ya)$ is $\mu=(4,4)$.
\begin{align*}
\ytableausetup{boxsize=1.2cm}
\begin{ytableau}
t\chi & qt & q^2t\chi^2&q^3t\chi\\
1 & q\chi^2& q^2\chi & q^3 
\end{ytableau}
\end{align*}
\ytableausetup{boxsize=1mm}

Starting with $B_\mu(q\chi^{-1},t\chi)$ we have
\begin{align*}
\begin{tikzcd}[ampersand replacement=\&]
(q^3+qt+1)\chi^0 + (q^3t + q^2 + t)\chi^1 + (q^2 t+q)\chi^2 \arrow[d,"R_1^*"] \& (\ya,\cd,\ya)  \\
(q^3+qt+1)\chi^0+ (q^2+t)\chi^1 + (q^2t+q)\chi^2 \arrow[d,"R_0^*"] \&
(\cd,\ya,\ya)  \\
(qt+1)\chi^0 + (q^2+t)\chi^1 + (q^2t+q)\chi^2 \arrow[d,"R_1^*"] \& 
(\ya,\ya,\cd)  \\
(qt+1)\chi^0 + (q^2+t)\chi^1 + (q^2t+q)\chi^2 \arrow[d,"R_2^*"] \& 
(\ya,\ya,\cd)  \\
(qt+1)\chi^0 + (q^2+t)\chi^1 + (tq^{-1}+q)\chi^2 \& (\ya,\cd,\ya).
\end{tikzcd}
\end{align*}
The result is $\cB_\mud^w$.
\end{ex}

\subsection{Wreath analog of Haiman's Positivity Theorem}\label{SS:wreath Haiman}

We are now prepared to state the wreath analog of Theorem~\ref{T:Haiman}.

\begin{thm}[\cite{BF,L2}]\label{T:wreath Haiman}
For each $\beta\in Q$, $n\in\Z_{\ge 0}$, and $w=u t_{-\beta^\vee}\in \Waf$, there exists an equivalence 
\begin{align}\label{E:BK equivalence}
D^b(\Coh^\T(\fM_{\beta,n})) \cong D^b(\Coh^{\T\times \Gamma_n}(\C^{2n}))
\end{align}
between derived categories of equivariant coherent sheaves. Moreover, at the level of $K$-theory, this equivalence realizes the wreath Macdonald polynomial $\tH_{\mud}^w$ as the bigraded $\Gamma_n$-Frobenius character of the fiber $P^w_n|_{I_\mud^w}$ of a rank $r^nn!$ vector bundle $P^w_n$ on $\fM_{\beta,n}$ whose fibers afford the regular representation of $\Gamma_n$.
\end{thm}

The equivalence \eqref{E:BK equivalence} originates from a more general result of Bezrukavnikov and Kaledin \cite{BK}; the work \cite{BF} connects it to wreath Macdonald theory. The vector bundle $P^w_n$ on $\fM_{\beta,n}$ described above is the wreath Procesi bundle. These were first considered for arbitrary $u\in\Wfin$ by Losev \cite{L1}, who also classified Procesi bundles on symplectic resolutions of quotient singularities in greater generality. The Procesi bundle we denote by $P^w_n$ is normalized so that 
\begin{align}\label{E:wreath Procesi normalization}
(P^w_n)^{\Gamma_{n-1}} = M^w_n.
\end{align}
We note that the methods of \cite{BF,L2} provide alternative proofs of Haiman's Theorem (Theorem~\ref{T:Haiman}) in the $r=1$ case.

As in the case of Haiman's Positivity Theorem, we pass from the setting of Theorem~\ref{T:wreath Haiman} to (multi-)symmetric functions by applying the $K_0$ functor, which gives an isomorphism
\begin{align*}
K^{\T}(\fM_{\beta,n}) \cong K^{\T\times\Gamma_n}(\C^{2n})
\end{align*}
of $R(\T)$-modules. We proceed as before, but replacing the $\fS_n$-Frobenius character by the $\Gamma_n$-Frobenius character \cite[I, Appendix B]{Mac} given by the map
\begin{align}\label{E:wreath Frob}
\mathrm{Frob}^\Gamma : R(\Gamma_n) \to (\Lambda^{\otimes I}_\Z)_n
\end{align}
sending an irreducible character $\chi^{\lad}$ of $\Gamma_n$ to the tensor Schur function $s_{\lad}$; we denote the bigraded $\Gamma_n$-Frobenius series by $\mathrm{Frob}_{q,t}^\Gamma$. 

Altogether we have an isomorphism
\begin{align*}
\Phi^w_n : K^{\T}(\fM_{\beta,n})_\loc \to (\Lambda^{\otimes I})_n
\end{align*}
of $\K$-vector spaces, which by Theorem~\ref{T:wreath Haiman} satisfies 
\begin{align}\label{E:BK iso image}
\Phi^w_n([I_\mud^w])=\mathrm{Frob}_{q,t}^\Gamma(P^w_n|_{I^w_\mud})=\tH^w_\mud.
\end{align}
The Laurent integrality and positivity of wreath Macdonald-Kostka coefficients (Theorem \ref{T:wreath pos}) follows directly from \eqref{E:BK iso image}.

\begin{rem}
In contrast to the $r=1$ case, the wreath Procesi fibers $P^w_n|_{I^w_\mud}$ are not quotients of $\C[x_1,\dotsc,x_n,y_1,\dotsc,y_n]$ in general. For instance, it is impossible for either of the $r=3$ examples
\begin{align*}
\tH^{s_1}_{(\ya,\cd,\cd)} &= s_{(\ya,\cd,\cd)}+qt^{-1}s_{(\cd,\ya,\cd)}+q s_{(\cd,\cd,\ya)}\\
\tH^{t_{-\al_1^\vee-\al_2^\vee}}_{(\cd,\ya,\cd)} &= 
s_{(\ya,\cd,\cd)}+q^{-1}s_{(\cd,\ya,\cd)}+t^{-1}s_{(\cd,\cd,\ya)} 
\end{align*}
to be a bigraded $\Gamma_n$-Frobenius character of a module $\C[x_1,\dotsc,x_n,y_1,\dotsc,y_n]/I$ for some $\Gamma_n$-invariant homogeneous ideal $I$. These examples do not have a unique minimal bi-degree, which the image of $1$ would provide in such a module.
\end{rem}

\begin{rem}
According to \cite{L1}, there are exactly $2|\Wfin|$ non-equivalent Procesi bundles on $\fM_{\beta,n}$. In our formulation of Theorem~\ref{T:wreath Haiman}, we only consider half of these. The other half is obtained by applying the dual. The resulting multisymmetric functions are precisely the orthogonal basis $\inv\,\negate\,\tH^w_\mud$ appearing in Proposition~\ref{P:H orthogonal}.
\end{rem}

Let us work out the condition that \eqref{E:wreath Procesi normalization} imposes on 
the wreath Macdonald-Kostka coefficients.
Under the bigraded wreath Frobenius map, the wreath Procesi fiber maps to the wreath Macdonald polynomial: $P_n^w|_{I_{\mud}^w}\mapsto\tH^w_\mud$. Restriction to $\Gamma_{n-1}$ is achieved on $\La^{\otimes I}$ by the operator $\sum_{i\in I} s_1[X^{(i)}]^\perp$ where $\perp$ is the adjoint of multiplication with respect to the Hall scalar product on $\La$. Taking $\Gamma_{n-1}$-invariants is accomplished by taking the coefficient of $s_{n-1}[X^{(0)}]$, the image of the trivial $\Gamma_{n-1}$ character under wreath Frobenius. Thus \eqref{E:wreath Procesi normalization} is equivalent to the following.

\begin{prop}[\cite{H:private}]
For all $w\in\Waf$ and $\mud\in\Y^I$ with $|\mud|=n$ we have
\begin{align*}
  [\chi^i] \cB_{\mud}^{w} = 
  \begin{cases}
  \pair{s_{(n-1)}[X^{(0)}] s_1[X^{(i)}]}{\tH^{w}_{\mud}}& \text{for $i\in I\setminus \{0\}$} \\
  1 + \pair{s_{(n-1,1)}[X^{(0)}]}{\tH^{w}_{\mud}}& \text{for $i=0$.} 
  \end{cases}
\end{align*}
Here for $n=1$ the $s_{(n-1,1)}$ should be interpreted as $0$.
\end{prop}

\begin{ex} Let $r$, $w$, and $\mud$ be as in Example \ref{X:quiver data}.
Using the reverse lex order on big partitions, the induced total order on multipartitions (which refines the order $\dom_w$) is
\[
\begin{array}{|c|c|c|c|c|c|c|c|c|}\hline
71 & 62 & 44 & 422 & 41^4 & 3221 & 321^3 & 22211 & 1^8 \\ \hline
(\yb,\cd,\cd) & 
(\cd,\cd,\yb) &
(\ya,\cd,\ya) & 
(\yaa, \cd,\cd) &
(\ya,\ya,\cd) & 
(\cd,\cd,\yaa) & 
(\cd,\ya,\ya) & 
(\cd,\yb,\cd) & 
(\cd,\yaa,\cd) \\ \hline
\end{array}.
\]

We have
\begin{align*}
\tH^w_\mud &= s_{(\yb,\cd,\cd)} + qt \,s_{(\yaa,\cd,\cd)} + (q^2+t)s_{(\ya,\ya,\cd)} + (tq^{-1}+q) s_{(\ya,\cd,\ya)} \\
&+ q \,s_{(\cd,\yb,\cd)} + q^2 t \,s_{(\cd,\yaa,\cd)} + (qt+1)s_{(\cd,\ya,\ya)} + q^{-1}\, s_{(\cd,\cd,\yb)} + t \,s_{(\cd,\cd,\yaa)}.
\end{align*}
We have 
\begin{align*}
[\chi^0]\cB_{\mud}^w&= 1 + qt = 1 + \pair{s_{(\yaa,\cd,\cd)}}{\tH^{w}_{\mud}} \\
[\chi^1]\cB_\mud^w &= q^2+t = \pair{s_{(\ya,\ya,\cd)}}{\tH^w_\mud} \\
[\chi^2]\cB_\mud^w &= tq^{-1} + q = \pair{s_{(\ya,\cd,\ya)}}{\tH^w_\mud}.
\end{align*}
\end{ex}

\subsection{Geometric pairing and norms}
\label{SS:norms}

Adapting the approach of \cite[\S 5.4.3]{H} to the wreath setting, we can now prove formula \eqref{E:H pairing constant formula} from Theorem~\ref{T:H pairing constant} for the norms of wreath Macdonald polynomials using the isomorphisms $\Phi^w_n$ arising from Theorem~\ref{T:wreath Haiman}. 

For $\T$-equivariant coherent sheaves $\cF,\cG$ on $\fM_{\beta,n}$, define the pairing
\begin{align*}
\pairTor{\cF}{\cG} = \sum_{i=0}^\infty (-1)^i \ch_\T \mathrm{Tor}_i(\cF,\cG)\in R(\T),
\end{align*}
where $\ch_\T$ denotes the $\T$-character. This extends to a well-defined and nondegenerate $\K$-valued inner product on $K_\T(\fM_{\beta,n})_\loc$. By \cite[Proposition 5.4.10]{CG}, it is given on the fixed point basis by
\begin{align*}
\pairTor{[I^w_\mud]}{[I^w_\nud]}&=\delta_{\mud\nud}\ch_\T(\Sym (T^* \fM_{\beta,n}|_{I^w_\mud}))^{-1}\\
&=\delta_{\mud\nud}\ch_\T(\Sym (T^* \Hilb_N^\Gamma|_{I_{\tau(w^{-1}(\mud,0))}}))^{-1},
\end{align*}
where $N=|\tau(w^{-1}(\mud,0))|$.

We also consider the pairing on $K_\T(\fM_{\beta,n})_\loc$ given by
\begin{align}
\pairTor{[\cF]}{[\cG]}'&= \sum_\mud c_\mud(q^{-1},t^{-1}) \pairTor{[I^w_\mud]}{[\cG]}
\end{align}
for $[\cF] = \sum_\mud c_\mud(q,t) [I^w_\mud]$. Note that this is conjugate $\K$-linear in the first argument and $\K$-linear in the second.

\begin{prop}
\begin{align*}
\pairTor{[\cF]}{[\cG]}' &= \pairqt{\inv\,\negate\, \Phi^{w}_n([\cF])}{\Phi^w_n([\cG])}.
\end{align*}
\end{prop}

\begin{proof}
Since both sides are conjugate linear in $[\cF]$ and linear in $[\cG]$, it suffices to check the assertion on $\K$-basis elements. We evaluate $\pair{[\cF]}{[\cG]}_{\mathrm{Tor}}'$ on the $\K$-bases $\{[I^w_\mud]\}_{\mud}$ and $\{[P^{w,*}_\lad]\}_{\lad}$ of $K_\T(\fM_{\beta,n})_\loc$, where
$$P^{w,*}_\lad=\Hom_{\Gamma_n}(V^{\lad*},(P^w_n)^*).$$ 
By \cite[\S4.1]{BF} and \cite[\S5.1.1]{L2}, we know that $\Phi^w_n([P^{w,*}_\lad])$ coincides with the bigraded $\Gamma_n$-Frobenius character of $V^{\lad}\otimes\C[x_1,\dotsc,x_n,y_1,\dotsc,y_n]$, namely:
\begin{align*}
\Phi^w_n(P^{w,*}_\lad) &= s_{\lad}[(1-q\chi)^{-1}(1-t\chi^{-1})^{-1}\Xd].
\end{align*}

The computation then proceeds as follows:
\begin{align*}
\pairTor{[I^w_\mud]}{[P^{w,*}_\lad]}' 
&=\pairTor{[I^w_\mud]}{[P^{w,*}_\lad]}\\ &=\ch_\T(P^{w,*}_\lad|_{I^w_\mud})\\
&=\ch_\T(\Hom(V^{\lad *},P^{w,*}_\lad|_{I^w_\mud}))\\
&=\pair{\inv\,\negate\, H^w_\mud}{\negate\,s_\lad[\Xd]}\\
&=\pair{\inv\,\negate\, H^w_\mud}{\,s_\lad[\negate\, \Xd]}\\
&=\pairqt{\inv\,\negate\, H^w_\mud}{\,s_\lad[(1-q\chi)^{-1}(1-t\chi^{-1})^{-1}\Xd]}\\
&=\pairqt{\inv\,\negate\,\Phi^w_n([I^w_\mud])}{\Phi^w_n([P^{w,*}_\lad])},
\end{align*}
where we use that $A=(1-q\chi^{-1})^{-1}(1-t\chi)^{-1}\negate$ satisfies $(1-q\chi)^{-1}(1-t\chi^{-1})^{-1}A^{-1}=\negate$.
\end{proof}

We deduce \eqref{E:H pairing constant formula}:
\begin{cor}
\begin{align*}
\pairqt{\inv\,\negate\, \tH^w_\mud}{\tH^w_\mud} &= \ch_\T(\Sym (T^* \Hilb_N^\Gamma|_{I_{\tau(w^{-1}(\mud,0))}}))^{-1}.
\end{align*}
\end{cor}

\begin{rem}
As shown in \cite[\S 5.4.3]{H}, one can eliminate $\inv$ (and $\negate$) from the above when $r=1$ by means of an isomorphism $P_n^*\cong P_n\otimes\varepsilon\otimes\cO(-1)$, where $\varepsilon$ denotes the sign representation of $\fS_n$ and $\cO(-1)$ is the dual of the top exterior power of the tautological bundle (corresponding to the operator $\nabla^{-1}$ on symmetric functions).
\end{rem}


\subsection{Wreath nabla}
\label{SS:wreath nabla}

For $w\in\Waf$, $\mud\in\Y^I_n$, and $i\in I$ define the Laurent monomial
\begin{align*}
    e^{(i),w}_\mud &= \ch_\T \bigwedge^n ([\chi^i]M^w_n)|_{I_\mud^w} \\
    &= e_n([\chi^i] \cB^w_\mud(q,t,\chi)).
\end{align*}
This value is the product of the $\T$-weights of the $\chi^i$ component of the fiber of $
M^w_n$ over $I_\mud^w$. The following result asserts that the monomial $e^{(i),w}_\mud$ records the bi-degree of the unique copy of the $i$-th sign representation of $\Gamma_n$ in $P^w_n|_{I_\mud^w}$.


\begin{thm}[\cite{H:private}]\label{T:quiver data} 
Let $\mud\in \Y^I$ have total size $n$. Then
\begin{align*}
\pair{e_n[X^{(i)}]}{\tH^w_\mud} &= e^{(i),w}_\mud.
\end{align*}
\end{thm}

\begin{ex} With $w$ and $\mud$ from the previous example, we have
\begin{align*}
    \pair{e_2[X^{(0)}]}{\tH^w_\mud} &= qt = e_2(qt+1) \\
    \pair{e_2[X^{(1)}]}{\tH^w_\mud} &= q^2 t = e_2(q^2+t) \\
    \pair{e_2[X^{(2)}]}{\tH^w_\mud} &= t = e_2(tq^{-1}+q).
\end{align*}
The input variables to the $e_n$ on the right are the summands of the coefficients of powers of $\chi$ in $\cB_\mud^w$.
\end{ex}

For $i\in I$ and $w\in \Waf$, the wreath nabla operator $\nabla^{(i)}_w$ is defined as the linear operator on $\La^{\otimes I}$ that acts diagonally on the wreath Macdonald basis $\tH^w_\mud$ with eigenvalue $e^{(i),w}_\mud$. (Recall that $\nabla^{(0)}_w$ made an appearance in \S\ref{SS:inv} above.)
 
\subsection{Factorization for generic stability condition}
We recall a factorization property observed by Haiman \cite{H:private}.

Fix $n\in \Z_{\ge0}$ and $u,v\in \Wfin$.
Let $\beta^\vee\in Q^\vee$ be sufficiently antidominant with respect to $n$, that is, $\pair{\beta^\vee}{\alb_i} \ll 0$ for all $1\le i\le r-1$. Let 
$w=ut_{v \beta^\vee}\in\Waf$. Then
\begin{itemize}
    \item $\fM_{w\theta_+}(n\delta)$ does not depend on $\beta^\vee$.
    \item It is isomorphic to the Hilbert scheme $\Hilb^n(\widehat{\C^2/\Gamma})$ of points on the minimal resolution $\widehat{\C^2/\Gamma}$ of the quotient singularity $\C^2/\Gamma$ \cite{Kuz}.
\end{itemize}
As a consequence the wreath Macdonald polynomial $\tH^w_\mud$ factorizes into modified Macdonald polynomials with appropriate substitutions of variables and parameters.

In this situation we write $\tH^{u,v}_\mud=\tH^w_\mud$ and call it a generic wreath Macdonald polynomial.

For $i\in I$ let $\sq^{(i)}\in\Y^I$ be the multipartition that has a single box partition in position $i$ and empty partitions elsewhere. Let
$M^{u,v}\in \Mat_{I\times I}(\K)$ be the matrix
\begin{align}
  M^{u,v}_{ij} &= \pair{\tH^{u,v}_{\sq^{(i)}}}{s_1[X^{(j)}]}.
\end{align}

\begin{ex}  Let $M^- = M^{\id,\id}$. For this stability condition the partial order $\dom_w$ on $\Y^I_n$ is refined by the following total order $\ge$: for distinct $\lad,\mud\in\Y^I_n$, let $0\le k\le r-1$ be minimum such that $\la^{(k)}\ne \mu^{(k)}$. Then $\lad>\mud$ if $\la^{(k)}$ is reverse lex larger than $\mu^{(k)}$ (the first part that disagrees is larger in $\la^{(k)}$). See \cite[Cor. 7.13]{G} for a precise characterization of the partial order in this case. All of the generic orders are permutations of each other.

For $r=3$ and $n=1$ we have
\begin{align*}
(\ya,\cd,\cd) \ge (\cd,\ya,\cd) \ge (\cd,\cd,\ya).
\end{align*}
For $n=2$ 
\begin{align*}
(\yb,\cd,\cd) \ge (\yaa,\cd,\cd) \ge (\ya,\ya,\cd) \ge (\ya,\cd,\ya) \ge (\cd,\yb,\cd)\ge
(\cd,\yaa,\cd) \ge (\cd,\ya,\ya) \ge (\cd,\cd,\yb) \ge (\cd,\cd,\yaa).
\end{align*}
Using this for $n=1$ one may compute
\begin{align*}
  M^-_{ij} &= \begin{cases}
      t^j & \text{if $i\ge j$} \\
      q^{r-j}&\text{if $i<j$.}
  \end{cases}
\end{align*}
\end{ex}

For a symmetric function $f\in \La$ and a row vector $(a_j\mid j\in I)\in \K^I$,
we write $f[\sum_{j\in I} a_j X^{(j)}]$ to mean the image of 
$f$ under the $\K$-algebra homomorphism $\La\to\La^{\otimes I}$ sending $p_k$ to $\sum_{j\in I} a_j(q^k,t^k) p_k(X^{(j)})$.
This is called a vector plethysm.

Let $\Xd$ be the column vector with $i$-th coordinate $X^{(i)}$ for $i\in I$ and
$\epsilon_i$ the $i$-th standard basis column vector. Then
\begin{align*}
\epsilon_i^t M^{u,v}\Xd = \sum_{j\in I} M^{u,v}_{ij} X^{(j)}.
\end{align*}

Define the vectors $q_\bullet,t_\bullet\in \K^I$ by 
$q_i = q^{r-i} t^{-i}$ and $t_i = q^{-i} t^{r-i}$ for $0\le i\le r-1$.

\begin{prop}[\cite{H:private}] \label{P:factor}  The wreath Macdonald polynomial factors:
\begin{align*}
  \tH^{u,v}_{\mud} = \prod_{i\in I} \tH^{u,v}_{\chi^i \mu^{(i)}}
\end{align*}
where the special notation $\chi^i \nu$ (not used elsewhere and not to be confused with the rotation of multipartitions) denotes the multipartition which has empty components except the $i$-th, which is $\nu\in\Y$. Moreover we have
\begin{align*}
  \tH^{u,v}_{\chi^i \nu} = \tH_{\nu}(\epsilon_i^t M^{u,v}\Xd,q_{v^{-1}(i)},t_{v^{-1}(i)})
\end{align*}
where the right hand side is a usual modified Macdonald polynomial $\tH_\nu$ with parameters
$(q,t)=(q_{v^{-1}(i)},t_{v^{-1}(i)})$ to which is applied the given vector plethysm. 
\end{prop}

\begin{ex} Let $r=3$, $i=2$, $\nu=(2)$, $u=v=\id$.
We have $\tH_{(2)} = s_2 + q s_{11}$. First the parameters 
are substituted: $q\mapsto q_2=q/t^2$ and $t\mapsto t_2=t^3$, giving $s_2 + (q/t^2) s_{11}$.
The vector plethysms of $s_2$ and $s_{11}$ using $Z=X^{(0)}+t X^{(1)}+t^2 X^{(2)}$ are:
\begin{align*}
s_2[Z]  &= s_2[X^{(0)}] + t^2 s_2[X^{(1)}] + t^4 s_2[X^{(2)}] \\
 \qquad &+ t s_1[X^{(0)}]s_1[X^{(1)}] + t^2 s_1[X^{(0)}]s_1[X^{(2)}] + t^3 s_1[X^{(1)}]s_1[X^{(2)}] \\
s_{11}[Z]   &= s_{11}X^{(0)}] + t^2 s_{11}[X^{(1)}]+ t^4 s_{11}[X^{(2)}] \\
  &+ t s_1[X^{(0)}]s_1[X^{(1)}]+ t^2 s_1[X^{(0)}]s_1[X^{(2)}] + t^3 s_1[X^{(1)}]s_1[X^{(2)}].
\end{align*}
One may verify that
\begin{align*}
    \tH^{\id,\id}_{(\cd,\cd,(2))} &= s_2[Z] + (q/t^2) s_{11}[Z].
\end{align*}

\end{ex}

\section{Quantum toroidal algebras and wreath Macdonald eigenoperators}

Throughout this section, we assume $r \ge 3$. Our goal is to explain how the quantum toroidal algebra provides a wreath generalization of the Macdonald eigenoperators from Corollary~\ref{C:tH-eigen} for wreath Macdonald functions $\tH^w_\mud$ with $w=t_{-\beta^\vee}$ a translation element. See Remark~\ref{R:2} for a brief discussion of the $r=2$ case and the connection to classical Macdonald theory for $r=1$.

\subsection{Quantum toroidal algebras}

Quantum toroidal algebras were introduced in \cite{GKV} for arbitrary simply-laced Lie types. Here we consider only the type $A$ case and follow the conventions of \cite{T,FJMM}. In this case there is a remarkable symmetry of the quantum toroidal algebra due to Miki \cite{Miki}, which plays a key role in the construction of wreath Macdonald eigenoperators.

\subsubsection{Setup}
We enlarge the base field to $\F=\Q(s,u)$, with $\K=\Q(q,t)$ embedded into $\F$ via $q=s^2u^2$ and $t=s^2u^{-2}$. We set $p=s^2$ and $d=u^2$.

Let $A=(a_{ij})_{i,j\in I}$ be the Cartan matrix of $\hat{\fsl}_r$ and let $M=(m_{ij})_{i,j\in I}$ be the skew-symmetric matrix given by:
\begin{align*}
m_{ij} = \begin{cases}\pm 1&\text{if $i=j\pm 1$}\\0 &\text{otherwise.}\end{cases}
\end{align*}

\subsubsection{Definition}
The quantum toroidal algebra $\Utor_r=\Utor_{s,u}(\gl_r)$ is the associative $\F$-algebra generated by
\begin{align*}
&C_h^{\pm 1/2}\,,\, C_v^{\pm 1/2}\,,\, e_{i,k}\, ,\, f_{i,k}\, ,\, \psi_{i,l}^\pm \qquad \text{for $i\in I\,,\, k\in\Z\,,\, \pm l\in\Z_{\ge 0}$}
\end{align*}
subject to the relations given below.\footnote{For the purposes of our discussion, we ignore the additional generators $q^{d_1},q^{d_2}$.} To state these, it is convenient to collect the generators into currents:
\begin{align*}
e_i(z) &= \sum_{k\in\Z}e_{i,k}z^{-k},\quad
f_i(z) = \sum_{k\in\Z}f_{i,k}z^{-k},\quad
\psi_i^\pm(z) = \sum_{\pm l\ge 0}\psi_{i,l}^\pm z^{-l}.
\end{align*}
The defining relations of $\Utor_r$ are then as follows for all $i,j\in I$:
\begin{align*}
&\text{$C_h^{1/2},C_v^{1/2}$ are central}\\
&\psi_{i,0}^+ \psi_{i,0}^- = \psi_{i,0}^- \psi_{i,0}^+ = 1\\
&\psi_i^\pm(z)\psi_j^\pm(w) = \psi_j^\pm(w)\psi_i^\pm(z)\\
&C_h=(C_h^{1/2})^2=\prod_{i\in I}\psi_{i,0}\\
&\theta_{a_{ij}}(C_v^{-1}d^{m_{ij}}z/w)\psi_i^+(z)\psi_j^-(w) = \theta_{a_{ij}}(C_v d^{m_{ij}}z/w)\psi_j^-(w)\psi_i^+(z)\\
&\psi_i^\pm(z)e_j(w) = \theta_{a_{ij}}(C_v^{\pm\frac{1}{2}}d^{m_{ij}}z/w)e_j(w)\psi_i^\pm(z)\\
%
%
&\psi_i^\pm(z)f_j(w) = \theta_{a_{ij}}(C_v^{\mp\frac{1}{2}}d^{m_{ij}}z/w)^{-1}f_j(w)\psi_i^\pm(z)\\
%
%
&e_i(z)e_j(w) = \theta_{a_{ij}}(d^{m_{ij}}z/w)e_j(w)e_i(z)\\
%
&f_i(z)f_j(w) = \theta_{a_{ij}}(d^{m_{ij}}z/w)^{-1}f_j(w)f_i(z)\\
&[e_i(z),f_j(w)] = \frac{\delta_{ij}}{p-p^{-1}}\left(\delta(C_v w/z)\psi_i^+(C_v^{\frac{1}{2}}w)-\delta(C_v z/w)\psi_i^-(C_v^{\frac{1}{2}}z)\right)\\
&[e_i(z_1),[e_i(z_2),e_{i\pm 1}(w)]_p]_{p^{-1}}+[e_i(z_2),[e_i(z_1),e_{i\pm 1}(w)]_p]_{p^{-1}} = 0\\
&[f_i(z_1),[f_i(z_2),f_{i\pm 1}(w)]_p]_{p^{-1}}+[f_i(z_2),[f_i(z_1),f_{i\pm 1}(w)]_p]_{p^{-1}} = 0
\end{align*}
where $\theta_m(z) = \frac{p^m z-1}{z-p^m}$ for any $m\in\Z$, $\delta(z)=\sum_{n\in\Z} z^n$, and $[x,y]_c = xy-cyx$.

Instead of the generators $\psi_{i,l}^\pm$ for $i\in I$ and $\pm l\ge 0$, we may use
\begin{align}\label{E:Utor-gens-heis}
\psi_{i,0}^{\pm}\,,\, h_{i,l} \qquad \text{for $i\in I$ and $l\in\Z_{\neq 0}$}
\end{align}
where the $h_{i,l}$ are defined by
\begin{align}\label{E:h-def}
\psi_i^\pm(z) = \psi_{i,0}^\pm \exp\left(\pm (p-p^{-1})\sum_{\pm l>0} h_{i,l}z^{-l}\right).
\end{align}



\subsubsection{Mode form of relations}
\label{SS:comp-rel}
In terms of individual generators, the relations of $\Utor_r$ can be fully unpacked as follows:
\begin{align}
\notag&\text{$C_h^{1/2},C_v^{1/2}$ are central}\\
\notag&\psi_{i,0}^+ \psi_{i,0}^- = \psi_{i,0}^- \psi_{i,0}^+ = 1\\
\notag&\psi_{i,k}^\pm\psi_{j,l}^\pm = \psi_{j,l}^\pm\psi_{i,k}^\pm\\
\notag&C_h=(C_h^{1/2})^2=\prod_{i\in I}\psi_{i,0}\\
%
&[h_{i,k},h_{j,-l}]
=\delta_{k,l}\frac{d^{-km_{ij}}}{k}\frac{(p^{ka_{ij}}-p^{-ka_{ij}})(C_v^k-C_v^{-k})}{(p-p^{-1})^2}\quad (k,l\neq 0)\label{E:h-comp}\\
\notag &\psi_{i,0}e_{j,k} = p^{a_{ij}}e_{j,k}\psi_{i,0}\\
\notag &\psi_{i,0}f_{j,k} = p^{-a_{ij}}f_{j,k}\psi_{i,0}\\
\notag &[h_{i,k},e_{j,l}] = \frac{C_v^{-|k|/2}d^{-m_{ij}k}}{k}\frac{p^{ka_{ij}}-p^{-ka_{ij}}}{p-p^{-1}} e_{j,k+l}\quad (k\neq 0)\\
\notag &[h_{i,k},f_{j,l}] = -\frac{C_v^{|k|/2}d^{-m_{ij}k}}{k}\frac{p^{ka_{ij}}-p^{-ka_{ij}}}{p-p^{-1}} f_{j,k+l}\quad (k\neq 0)\\
%
%
\notag &d^{m_{ij}}[e_{i,k+1},e_{j,l}]_{p^{a_{ij}}}+[e_{j,l+1},e_{i,k}]_{p^{a_{ij}}}=0\\
%
\notag &[f_{i,k},f_{j,l+1}]_{p^{a_{ij}}}+d^{m_{ij}}[f_{j,l},f_{i,k+1}]_{p^{a_{ij}}}=0\\
\notag &[e_{i,k},f_{j,l}] = \delta_{ij}\frac{C_v^{(k-l)/2}\psi_{i,k+l}^+-C_v^{-(l-k)/2}\psi_{i,k+l}^-}{p-p^{-1}}\\
\notag&[e_{i,k_1},[e_{i,k_2},e_{i\pm 1,l}]_p]_{p^{-1}}+[e_{i,k_2},[e_{i,k_1},e_{i\pm 1,l}]_p]_{p^{-1}} = 0\\
\notag&[f_{i,k_1},[f_{i,k_2},f_{i\pm 1,l}]_p]_{p^{-1}}+[f_{i,k_2},[f_{i,k_1},f_{i\pm 1,l}]_p]_{p^{-1}} = 0
\end{align}
where $i,j\in I$ and $k,l\in\Z$ and we set $\psi_{i,l}^+ = \psi_{i,-l}^- = 0$ for all $i\in I$ and $l<0$.

\subsubsection{Involution on a quotient}
\label{SS:involution}

Let $\Utor_{r,0}=\Utor_r/\langle C_v^{1/2}-1\rangle$. Observe that there exists a unique $\Q(u)$-linear involutive automorphism $\dag$ of $\Utor_0$ extending the assignment:
\begin{align*}
s&\mapsto s^{-1}\\
e_i(z)&\mapsto f_i(z)\\
f_i(z)&\mapsto e_i(z)\\
\psi_i^\pm(z)&\mapsto \psi_i^\pm(z)
\end{align*}
for all $i\in I$. To see this, one uses that $\theta_m(x)^{-1}=\theta_{-m}(x)$.

Observe that $q^\dag=t^{-1}$ and $t^\dag=q^{-1}$.

\subsubsection{Quantum affine subalgebras and the Miki automorphism}\label{SS:miki}

The quantum toroidal algebra $\Utor_r$ contains two distinguished copies of the quantum affine algebra $\Uaf_r=U_p'(\hat{\fsl}_r)$.

The \textit{horizontal} quantum affine subalgebra $U^h\subset \Utor_r$ is the $\Q(p)$-subalgebra generated by $e_{i,0},f_{i,0}$, and $\psi_{i,0}^{\pm 1}$ for all $i\in I$. An isomorphism $h: \Uaf_r\to U^h$ of $\Q(p)$-algebras is given by identifying these generators with the Chevalley generators in the Drinfeld-Jimbo presentation of $\Uaf_r$.

The \textit{vertical} quantum affine subalgebra $U^v\subset \Utor_r$ is the $\Q(p)$-subalgebra generated by $C_v$ and $d^{ik}e_{i,k},d^{ik}f_{i,k}$, and $d^{ik}C_v^{k/2}\psi^{\pm}_{i,k}$ for $1\le i\le r-1$ and $k\in\Z$. An isomorphism $v: \Uaf_r\to U^h$ of $\Q(p)$-algebras is given by identifying these generators with the generators in the Drinfeld loop presentation of $\Uaf_r$.

\begin{thm}[\cite{Miki}]\label{T:miki}
There exists an automorphism $\miki : \Utor_r \to \Utor_r$ such that 
\begin{align}\label{E:miki-h-v}
\text{$\miki(U^h)=U^v$ \ and \ $\miki(U^v)=U^h$.}
\end{align}
Moreover, $\miki$ is uniquely determined by compatibility with the isomorphisms $h : \Uaf_r\to U^h$ and $v: \Uaf_r\to U^v$.
\end{thm}

We call $\miki$ the \textit{Miki automorphism}. We note that $\miki$ is not uniquely determined by \eqref{E:miki-h-v} alone (for instance, $\miki\neq \miki^{-1}$). The compatibility of $\miki$ with the isomorphisms $h: \Uaf_r \to U^h$ and $v: \Uaf_r \to U^v$, which we will not state explicitly, is asymmetric with respect to $h$ and $v$; see \cite[Theorem 1]{Miki} or \cite[Theorem 2.2]{FJMM} for details. 

It is in general difficult to describe the action of $\miki$ on elements of $\Utor_r$. For our purposes, the following explicit formulas will suffice to characterize $\miki$:

\begin{prop}[\cite{T,FJMM}]\label{P:miki-formulas}
The Miki automorphism $\miki : \Utor_r\to \Utor_r$ is uniquely determined by the following formulas, where $1\le i\le r-1$:
\begin{align*}
\miki(C_v^{1/2}) &= C_h^{-1/2}\\
\miki(C_h^{1/2}) &= C_v^{1/2}\\
\miki(\psi_{i,0}) &= \psi_{i,0}\\
\miki(\psi_{0,0}) &= \psi_{0,0}C_h^{-1}C_v\\
\miki(e_{i,0}) &= e_{i,0}\\
\miki(f_{i,0}) &= f_{i,0}\\
\miki(e_{0,-1}) &= (-1)^r u^{2r}e_{0,1}\\
\miki(f_{0,1}) &= (-1)^r u^{-2r}f_{0,-1}
\end{align*}
and
\begin{align*}
\miki(h_{i,1})
= (-1)^{i+1} u^{-2i}C_h^{-1/2}
f_{0,0} \cdot
\begin{bmatrix}
f_{r-1,0} & \dotsm & f_{i+1,0} & f_{1,0} & \dotsm & f_{i-1,0} & f_{i,0}\\
p & \dotsm & p & p & \dotsm & p & p^2
\end{bmatrix}\\
\miki(h_{i,-1}) = (-1)^{i+1} u^{2i}C_h^{1/2}
\begin{bmatrix}
e_{i,0} & e_{i-1,0} & \dotsm & e_{1,0} & e_{i+1,0} & \dotsm & e_{r-1,0} \\
p^{-2} & p^{-1} & \dotsm & p^{-1} & p^{-1} & \dotsm & p^{-1}
\end{bmatrix}\cdot e_{0,0}
\end{align*}
\begin{align*}
\miki(h_{0,1}) = (-1)^r u^{2-2r}C_h^{-1/2}f_{1,1}\cdot
\begin{bmatrix} 
f_{2,0} & \dotsm & f_{r-1,0} & f_{0,-1} \\
p & \dotsm & p & p^2 
\end{bmatrix}\\
\miki(h_{0,-1}) = (-1)^r u^{2r-2}C_h^{1/2}
\begin{bmatrix}
e_{0,1} & e_{r-1,0} & \dotsm & e_{2,0} \\
p^{-2} & p^{-1} & \dotsm & p^{-1}
\end{bmatrix}
\cdot e_{1,-1},
\end{align*}
in terms of the iterated commutators:
\begin{align*}
x \cdot \begin{bmatrix} y_1 & \dotsm & y_m \\ c_1 & \dotsm & c_m\end{bmatrix} &= [\dotsc[[x,y_1]_{c_1},y_2]_{c_2}\dotsc,y_m]_{c_m}\\
\begin{bmatrix} x_m & \dotsm & x_1 \\ c_m & \dotsm & c_1\end{bmatrix}\cdot y &= [x_m,\dotsc[x_2,[x_1,y]_{c_1}]_{c_2}\dotsc]_{c_m}.
\end{align*}
\end{prop}

\begin{rem}
We extracted the formulas above from \cite{FJMM} via the following identification:
\begin{center}
\begin{tabular}{|c||c|c|c|c|c|c|c|c|c|}
\hline
& $\miki$ & $C_h$ & $C_v$ & $e_{i,k}$ & $f_{i,k}$ & $\psi^{\pm}_{i,0}$ & $h_{i,l}$ & $p$ & $d$\\
\hline
\cite{FJMM} & $\theta^{-1}$ & $\kappa$ & $q^c$ & $E_{i,k}$ & $F_{i,k}$ & $K_i^{\pm 1}$ & $q^{-lc/2}H_{i,l}$ & $q$ & $d$\\
\hline
\end{tabular}.
\end{center}
\end{rem}

\subsubsection{Heisenberg subalgebras}

Let $\heis_v\subset\Utor_r$ be the $\F$-span of $C_v^{\pm 1}$ and $h_{i,l}$ for $i\in I$ and $l\neq 0$. The subspace $\heis_v$ is closed under the Lie bracket and is isomorphic to a $p$-deformed Heisenberg Lie algebra with relations given by \eqref{E:h-comp}. We call $\heis_v$ the \textit{vertical} Heisenberg (Lie) subalgebra of $\Utor_r$.

The \textit{horizontal} Heisenberg subalgebra of $\Utor_r$ is the Lie subalgebra $\heis_h=\miki(\heis_v)$. As a vector space, $\heis_h$ is spanned by $C_h^{\pm 1}$ and $\miki(h_{i,l})$ for $i\in I$ and $l\neq 0$. The horizontal Heisenberg elements $\miki(h_{i,\pm 1})$ are given explicitly by Proposition~\ref{P:miki-formulas}. 

Acting in the vertex representation $\cV$ of $\Utor_r$ introduced below, the elements $\miki(h_{i,l})$ will lead to wreath Macdonald eigenoperators.

\subsection{Fock representation}

The eigenvalues of wreath Macdonald eigenoperators are found in the Fock representation of $\Utor$.

\subsubsection{Lowest weight Fock representation}\label{SS:Fock}

The Fock space $\cF_\xi=\bigoplus_{\mu\in\Y} \F \ket{\mu}$ affords an action of $\Utor_r$ depending on a parameter $\xi\in\F^\times$ \cite{VV:double, STU}, which simultaneously extends the well-known level 1 action of $\Uaf_r$ on this space \cite{KMS} and the level 0 action of $\Uaf_r$ due to \cite{TU} (via the horizontal and vertical embeddings, respectively). Here we follow the conventions of \cite{FJMM,T}.

To describe this action explicitly, let $\cF_\xi^\vee=\oplus_{\mu\in\Y}\F \bra{\mu}$ be the finite dual space, where $\langle \mu | \nu\rangle = \delta_{\mu\nu}$, and define the rational function $\phi(z) = \frac{p-p^{-1}z}{1-z}$.

\begin{thm}[{\cite[Proposition 3.3]{FJMM}}]
The space $\cF_\xi$ affords an action of the quantum toroidal algebra $\Utor_r$ whose nonzero matrix coefficients are as follows:
\begin{align*}
\bra{\mu} C_h^{1/2}\ket{\mu} &= s^{-1}\\
\bra{\mu} C_v^{1/2}\ket{\mu} &= 1\\ 
\bra{\nu} e_i(z) \ket{\mu} &= 
\prod_{\substack{b<y\\(a,b)\in A_{i}(\mu)}}\phi(t^{a-x}q^{b-y})
\prod_{\substack{b<y\\(a,b)\in R_{i}(\mu)}}\phi(t^{x-a}q^{y-b})
\delta(t^{-x}q^{-y}\xi/z)\\
\bra{\mu} f_i(z) \ket{\nu} &= 
\prod_{\substack{b>y\\(a,b)\in A_{i}(\mu)}}\phi(t^{a-x}q^{b-y})
\prod_{\substack{b>y\\(a,b)\in R_{i}(\mu)}}\phi(t^{x-a}q^{y-b})
\delta(t^{-x}q^{-y}\xi/z)\\
\bra{\mu} \psi_i^{\pm}(z) \ket{\mu} &= 
\prod_{(a,b)\in A_{i}(\mu)}\phi(t^{-a}q^{-b}qt\xi/z)^{-1}
\prod_{(a,b)\in R_{i}(\mu)}\phi(t^{-a}q^{-b}\xi/z)
\end{align*}
for all $\mu\in\Y$, $i\in I$, and all $\nu\in\Y$ such that $\nu/\mu$ is a single cell $(x,y)$ satisfying $y-x\equiv i \ (\mathrm{mod}\ r)$. 

In particular, $\cF_\xi$ is an irreducible lowest weight module over $\Utor_{r,0}$, with lowest weight vector $\ket{\varnothing}$ and lowest weight $(\phi(z/\xi)^{\delta_{i0}})_{i\in I}$.
\end{thm}

Here the matrix coefficients for $\psi^\pm_i(z)$ are obtained by expanding the rational functions in $z^{\mp 1}$, respectively. The lowest weight condition in $\cF_\xi$ reads as follows:
\begin{align*}
\psi^\pm_i(z)\ket{\varnothing} &= \phi(z/\xi)^{\delta_{i0}}\ket{\varnothing}\\
f_i(z)\ket{\varnothing} &= 0 \quad \text{for all $i\in I$.}
\end{align*}

\begin{rem}
More generally, there is a Fock representation of $\Utor_r$ for each $p\in I$, corresponding to a fundamental weight of $\hat{\fsl}_r$. We only consider the $p=0$ case, which matches the choice of vertex supporting the one-dimensional framing space in the quiver varieties $\fM_{\beta,n}$ from the previous section.
\end{rem}

\subsubsection{Twisting}

To obtain a highest weight $\Utor_{r,0}$-module $\cF^\dag_\xi$, we twist the Fock module by the involutive automorphism $\dag$ of $\Utor_{r,0}$ from \S\ref{SS:involution}. More precisely, let $\rho_{\cF_\xi} : \Utor_{r,0} \to \End_\F(\cF_\xi)\cong \Mat_{\Y}(\F)$ be the $\F$-algebra homomorphism determined by the Fock representation $\cF_\xi$, where $\Mat_{\Y}(\F)$ is the algebra of $\Y\times\Y$-matrices over $\F$ whose columns have finite support. We identify $\End_\F(\cF_\xi)\cong \Mat_{\Y}(\F)$ by means of the basis $\{\ket{\mu}\}_{\mu\in\Y}$.

Then, we define $\rho_{\cF_\xi^\dag}$ as the composition:
\begin{align*}
\rho_{\cF_\xi^\dag}: \Utor_{r,0}\overset{\dag}{\longrightarrow}\Utor_{r,0}\overset{\rho_{\cF_\xi}}{\longrightarrow}\Mat_{\Y}(\F) \overset{\dag}{\longrightarrow} \Mat_{\Y}(\F)
\end{align*}
where, by an abuse of notation, the map $\Mat_{\Y}(\F) \overset{\dag}{\longrightarrow} \Mat_{\Y}(\F)$ is given by entry-wise application of the automorphism $a(s,u)\mapsto a(s^{-1},u)$ of $\F$. We use the same vector space $\cF_\xi^\dag=\oplus_{\mu\in\Y} \F \ket{\mu}$ and regard the images of $\rho_{\cF_\xi^\dag}$ as matrices with respect to the basis $\{\ket{\mu}\}_{\mu\in\Y}$.

\begin{cor}\label{C:hwF}
The homomorphism $\rho_{\cF_\xi^\dag}$ determines an action of $\Utor_r$ on $\cF_\xi^\dag$ whose nonzero matrix coefficients are as follows:
\begin{align}
\notag
\bra{\mu} C_h^{1/2}\ket{\mu} &= s\\
\notag
\bra{\mu} C_v^{1/2}\ket{\mu} &= 1\\ 
%
\notag
\bra{\mu} e_i(z) \ket{\nu} &= 
\prod_{\substack{b>y\\(a,b)\in A_{i}(\mu)}}\phi(q^{a-x}t^{b-y})
\prod_{\substack{b>y\\(a,b)\in R_{i}(\mu)}}\phi(q^{x-a}t^{y-b})
\delta(q^{x}t^{y}\xi^\dag/z)\\
%
\notag
\bra{\nu} f_i(z) \ket{\mu} &= 
\prod_{\substack{b<y\\(a,b)\in A_{i}(\mu)}}\phi(q^{a-x}t^{b-y})
\prod_{\substack{b<y\\(a,b)\in R_{i}(\mu)}}\phi(q^{x-a}t^{y-b})
\delta(q^{x}t^{y}\xi^\dag/z)\\
\label{E:psi-current-act}
\bra{\mu} \psi_i^{\pm}(z) \ket{\mu} &= 
\prod_{(a,b)\in A_{i}(\mu)}\phi(q^{-a}t^{-b}qtz/\xi^\dag)^{-1}
\prod_{(a,b)\in R_{i}(\mu)}\phi(q^{-a}t^{-b}z/\xi^\dag)
\end{align}
for all $\mu\in\Y$, $i\in I$, and all $\nu\in\Y$ such that $\nu/\mu$ is a single cell $(x,y)$ satisfying $y-x\equiv i \ (\mathrm{mod}\ r)$.

In particular, $\cF_\xi$ is an irreducible highest weight module over $\Utor_{r,0}$, with highest weight vector $\ket{\varnothing}$ and highest weight $(\phi(\xi^\dag/z)^{\delta_{i0}})_{i\in I}$.
\end{cor}

To compute these matrix coefficients, we use that $q^\dag = t^{-1}$, $t^\dag = q^{-1}$, and
\begin{align*}
\phi(c/z)^\dag &=\left(\frac{p-p^{-1}c/z }{1-c/z}\right)^\dag=\frac{p^{-1}-p c^\dag/z }{1-c^\dag/z}=\phi(z/c^\dag)
\end{align*}
for any $c\in\F$.

The highest weight condition in $\cF^\dag_\xi$ reads as follows:
\begin{align*}
\psi^\pm_i(z)\ket{\varnothing} &= \phi(\xi^\dag/z)^{\delta_{i0}}\ket{\varnothing}\\
e_i(z)\ket{\varnothing} &= 0 \quad \text{for all $i\in I$.}
\end{align*}
From \eqref{E:h-def} and \eqref{E:psi-current-act}, one computes the following finer description of the joint eigenvalues of $\heis^v$ in $\cF_\xi^\dag$:
\begin{align}
%
\label{E:h-modes-act}
\bra{\mu} h_{i,\pm l} \ket{\mu} &= A_\mu^{(i)}(q^{\pm l},t^{\pm l})\frac{p^l-p^{-l}}{p-p^{-1}}\frac{(p^{-1}\xi^\dag)^{\pm l}}{l}
\end{align}
for all $\mu\in\Y$, $i\in I$, and $l>0$, where $A_\mu^{(i)}$ denotes the coefficient of $\chi^i$ in $A_\mu(q\chi^{-1},t\chi)$. At this point, it is instructive to recall that the quantities $A_\mu(q^{\pm 1},t^{\pm 1})$ from \eqref{E:Amudef} were precisely the eigenvalues of Macdonald eigenoperators in Corollary~\ref{C:tH-eigen}. As we will see below, \eqref{E:h-modes-act} determines the eigenvalues of wreath Macdonald eigenoperators.

\begin{rem}
In \cite{T}, it is asserted that the dual $\cF_\xi^\vee$ can be made into a $\Utor_r$-module via the antipode of a topological Hopf algebra structure on $\Utor_r$. 
Here we achieve the same desired effect---namely, converting a lowest weight module to a highest weight module---by using the automorphism $\dag$.
\end{rem}

\begin{rem}\label{R:Fock geom}


By \cite{VV,Nak} (see also \cite{Nag}), the space
\begin{align}
\cK=\bigoplus_{\substack{\beta\in Q\\n\ge 0}}K^\T(\fM_{\beta,n})\otimes_{R(\T)}\F=\bigoplus_{N\ge 0}K^\T(\Hilb_N^\Gamma)\otimes_{R(\T)}\F
\end{align}
has the structure of an irreducible highest weight $\Utor_{r,0}$-module with highest weight $(\phi(z^{-1})^{\delta_{i0}})_{i\in I}$. This action is defined geometrically, in fashion similar to the action appearing in Theorem~\ref{T:FT-SV}, with the currents $e_i(z)$ and $f_i(z)$ acting via correspondences, and the Cartan currents $\psi^{\pm}_i(z)$ acting by $K$-theory multiplication with tautological classes.

Recall that we identify the Nakajima quiver varieites $\fM_{\beta,n}$ for $\beta\in Q$ and $n\in\Z_{\ge 0}$ with the irreducible components of $\Hilb_N^\Gamma$ for $N\in\Z_{\ge 0}$. The space $\cK$ has a basis given by the classes $[I_\mu]^\Gamma$ of $\T$-fixed points $I_\mu\in\Hilb_{|\mu|}^\Gamma$ for $\mu\in\Y$. These classes form a weight basis of $\cK$ under the $\Utor_r$ action:
\begin{align}\label{E:Fock geom weights}
\psi_i^\pm(z)[I_\mu]^\Gamma &= \prod_{(a,b)\in A_{i}(\mu)}\phi(q^{-a} t^{-b}qtz)^{-1}\prod_{(a,b)\in R_{i}(\mu)}\phi(q^{-a}t^{-b} z)[I_\mu]^\Gamma.
\end{align}
In particular, $[I_\varnothing]^\Gamma$ is a highest weight vector in $\cK$.

By the uniqueness of irreducible highest weight $\Utor_r$-modules (cf. \cite{FJMM1}), the modules $\cK$ and $\cF_1^\dag$ must be isomorphic. Moreover, the joint spectrum $(\psi_i^\pm(z))_{i\in I}$ is simple in these modules, as the explicit formula \eqref{E:psi-current-act} (or \eqref{E:Fock geom weights}) shows. Thus, for any isomorphism $\Phi : \cK \to \cF_1^\dag$ of $\Utor_r$-modules, it must be true that
\begin{align}\label{E:eigenspaces}
\Phi([I_\mu]^\Gamma) \in \F^\times|\mu\rangle\subset\cF_1^\dag,\quad \text{for $\mu\in\Y$.}
\end{align}

\comment{
\begin{tabular}{c|c}
\cite{Nag} & us\\
\hline
$q=s^{1/2}t^{1/2}$ & $p=s^2$\\
$r=s^{-1/2}t^{1/2}$ & $d^{-1}=u^{-2}$\\
$s=qr^{-1}$ & $pd = q_3^{-1}=q$\\
$t=qr$ & $pd^{-1}= q_1^{-1}=t$
\end{tabular}
}

\end{rem}

\subsection{Vertex representation}

Wreath Macdonald eigenoperators act in the vertex reresentation of $\Utor_r$ \cite{Saito}, as we now describe.

Recall that $Q$ denotes the root lattice of $\mathfrak{sl}_r$ and that $\alb_i=\cl(\al_i)$ for all $i\in I$. Let $\pair{\cdot}{\cdot}: Q^\vee \times Q\to\Z$ denote the canonical pairing.

\subsubsection{Skew group algebra} 
Let $\Z\{Q\}=\bigoplus_{\alpha\in Q}\Z e^\alpha$. We equip $\Z\{Q\}$ with a noncommutative product as follows. For $\alpha=\sum_{i=1}^{r-1} m_i \al_i$ and $\beta=\sum_{j=1}^{r-1} n_j \al_j$, define
\begin{align}\label{E:skew-mult}
	e^{\al}e^{\beta} &= 
	\sg(\al,\beta) e^{\al+\beta}
\end{align}
where
\begin{align}\label{E:cocycle}
	\sg(\al,\beta) &= (-1)^{\sum_{j=1}^{r-2} m_{j+1}n_j}.
\end{align}
One verifies directly that \eqref{E:skew-mult}, extended by linearity, gives $\Z\{Q\}$ the structure of a ring with $1=e^0$. 


The following are easily verified using the above definition:


\begin{lem}
For all $i,j\in I$, we have $e^{\alpha_i}e^{\alpha_j}=(-1)^{a_{ij}}e^{\alpha_j}e^{\alpha_i}$.
\end{lem}

\begin{lem} 
For any $\alpha\in Q$, we have $(e^{\alpha})^{-1} = \sg(\alpha) e^{-\alpha}$ where $\sg(\alpha):=\sg(\alpha,\alpha)$. We have $\sg(\alb_i) = 1$ for $i\neq 0$, and $\sg(\alb_0)=(-1)^r$.
\end{lem}

Let $\F\{Q\}=\Z\{Q\}\otimes_\Z \F$.

\subsubsection{Vertex representation}
The vertex representation of $\Utor_r$ has underlying vector space
\begin{align*}
	\cV &= \Lambda_\F^{\otimes I}\otimes_{\F} \F\{Q\}
\end{align*}
where $\Lambda_\F^{\otimes I}$ is the $I$-fold tensor product of $\Lambda_\F$ with itself over $\F$. In preparation for the definition of the $\Utor_r$-action on $\cV$, we define linear maps $L_\alpha, p^{\partial_{\alpha_i}}: \F\{Q\}\to\F\{Q\}$ and $z^{H_{i,0}}: \F\{Q\}\to\F\{Q\}[z^{\pm 1}]$ as follows:
\begin{align*}
L_\alpha e^\beta &= e^\alpha e^\beta\\
p^{\partial_{\al_i}}e^\beta &= s^{2\pair{\alb_i^\vee}{\beta}}e^\beta\\
z^{H_{i,0}} e^\beta &= z^{\pair{\alb_i^\vee}{\beta}}u^{\pair{\alb_i^\vee}{M_i\beta}}e^{\beta}
\end{align*}
where $\alpha,\beta\in Q$, $i\in I$, and for $\beta=\sum_{j=1}^{r-1} n_j \al_j$:
\begin{align*}
M_i\beta &= n_{i-1}\alb_{i-1} - n_{i+1}\alb_{i+1}
\end{align*}
with $n_0=0$ and, as usual, subscripts taken mod $r$.

We also introduce the vertex operators $E^{(i)}(z), F^{(i)}(z) \in\End(\Lambda^{\otimes I})[[z^{\pm 1}]]$ given by:
\begin{align*}
E^{(i)}(z) &= \sum_{k\in\Z} E^{(i)}_k z^{-k} =
\Omega[s^{-1} zX^{(i)}]\Omega[-s z^{-1} Y^{(i)}(X,s,u)]^\perp \\
F^{(i)}(z) 
&= \sum_{k\in\Z} F^{(i)}_k z^{-k} = \Omega[-s zX^{(i)}]\Omega[s^{-1}z^{-1}Y^{(i)}(X,s^{-1},u)]^\perp
\end{align*}
where
\begin{align*}
Y^{(i)}(X,s,u) &= (1+s^{-4})X^{(i)}-s^{-2}u^2 X^{(i+1)}-s^{-2}u^{-2} X^{(i-1)}\\
&=(1+q^{-1}t^{-1})X^{(i)}-t^{-1} X^{(i+1)}-q^{-1} X^{(i-1)}\\
Y^{(i)}(X,s^{-1},u) &= (1+s^4)X^{(i)}-s^{2}u^2 X^{(i+1)}-s^{2}u^{-2} X^{(i-1)}\\
&=(1+qt)X^{(i)}-q X^{(i+1)}-t X^{(i-1)}.
\end{align*}

\begin{rem}
The vertex operators $E^{(i)}(z)$ and $F^{(i)}(z)$ are natural wreath analogs of $\tD^*(z)$ and $\tD(z)$ from \eqref{E:tD-star} and \eqref{E:tD}. However, to obtain wreath Macdonald eigenoperators, it is necessary to combine various coefficients of $E^{(i)}(z)$ and $F^{(i)}(z)$ for \textit{all} $i\in I$ (see Theorem~\ref{T:eigen} and Example~\ref{X:D} below).
\end{rem}

Now we are ready to describe Saito's vertex representation of $\Utor_r$:

\begin{thm}[\cite{Saito}]
The following formulas define an action of $\Utor_r$ on $\cV$:
\begin{align*}
C_v^{1/2} &\mapsto s\\
C_h^{1/2} &\mapsto 1\\
e_i(z) &\mapsto E^{(i)}(z) \otimes L_{\alb_i}z^{1+H_{i,0}}\\
f_i(z) &\mapsto F^{(i)}(z) \otimes \sg(\alb_i) L_{-\alb_i} z^{1-H_{i,0}}\\
\psi^-_i(z) &\mapsto \Omega[-z(p-p^{-1})X^{(i)}]\otimes p^{-\partial_{\al_i}}\\
\psi^+_i(z) &\mapsto \Omega[z^{-1}(p^2-1)Y^{(i)}(X,s,u)]^\perp\otimes p^{\partial_{\al_i}}
\end{align*}
for all $i\in I$. Moreover, $\cV$ is an irreducible $\Utor_r$-module.
\end{thm}

On the subspaces $\cV_\alpha = \La^{\otimes I}_\F \otimes \F e^{\alpha}$, the action of $\Utor_r$ is given more explicitly as follows:
\begin{align*}
	e_i(z)\Big|_{\cV_\alpha} &=
	\left(z^{1+\pair{\alb_i^\vee}{\alpha}} u^{\pair{\alb_i^\vee}{M_i\alpha}} E^{(i)}(z)  \otimes L_{\alb_i}\right)\res{\alpha} \\
	e_{i,k}\Big|_{\cV_{\alpha}} 
	&= \left(u^{\pair{\alb_i^\vee}{M_i\alpha}}E^{(i)}_{k+1+\pair{\alb_i^\vee}{\alpha}} \otimes L_{\alb_i}\right)\res{\alpha} \\
	f_i(z)\Big|_{\cV_\alpha} &=
	\left(z^{1-\pair{\alb_i^\vee}{\alpha}} u^{-\pair{\alb_i^\vee}{M_i\alpha}} F^{(i)}(z)  \otimes \sg(\alb_i) L_{-\alb_i}\right)\res{\alpha} \\
	f_{i,k}\Big|_{\cV_{\alpha}} 
	&= \left(u^{-\pair{\alb_i^\vee}{M_i\alpha}}F^{(i)}_{k+1-\pair{\alb_i^\vee}{\alpha}} \otimes \sg(\alb_i) L_{-\alb_i}\right)\res{\alpha}\\
	h_{i,-k}\Big|_{\cV_{\alpha}} &= s^{-2\pair{\alb_i^\vee}{\alpha}}\frac{p_k[X^{(i)}(s^2-s^{-2})]}{k(s^2-s^{-2})} \qquad (k>0)\\
	h_{i,k}\Big|_{\cV_{\alpha}} &= s^{2\pair{\alb_i^\vee}{\alpha}}\frac{p_k[Y^{(i)}(X,s,u)(s^4-1)]^\perp}{k(s^2-s^{-2})} \qquad (k>0).
\end{align*}

\subsubsection{Miki twist}
The vertex representation is not a highest (or lowest) weight representation in the ordinary sense. It is a miraculous fact that its twist by the Miki automorphism is a highest weight representation.

More precisely, let $\rho_\cV:\Utor_r \to \End(\cV)$ be Saito's vertex representation. Its Miki twist $\cV^\miki$ is given on the same representation space $\cV^\miki=\cV$ but with the action of $\Utor$ determined by the homomorphism $\rho_{\cV^\miki}=\rho_{\cV}\circ \miki$.

\begin{thm}[\cite{T}]\label{T:tsymbaliuk}
The Miki-twisted vertex representation $\cV^\miki$ is an irreducible highest weight $\Utor_r$-module with highest weight vector $1\otimes e^0$ and highest weight $(\phi(\xi_\circ^\dag/z)^{\delta_{i0}})_{i\in I}$ where $\xi_\circ^\dag=(-1)^{r}s^{2}u^{-r}$.
\end{thm}

\begin{rem}
The power of $u$ in our constant $\xi^\dag$ comes from a correction to \cite{T} made by \cite{Wen}. The sign in $\xi_\circ^\dag$ is slightly different from that of \cite{T,Wen} and it results from our choice of $2$-cocycle \eqref{E:cocycle} used to define the product in $\F\{Q\}$.
\end{rem}

\begin{rem}
A precursor to Theorem~\ref{T:tsymbaliuk} was given in \cite{STU}, where it was shown that the $U^v$-action in $\cV$ agrees with the $U^h$-action in $\cF^\dag_\xi$ (which is independent of $\xi$).
\end{rem}

Recalling now the highest weight Fock representation $\cF_\xi^\dag$ and the uniqueness of irreducible highest weight modules, we deduce that $\cV^\miki\cong\cF_{\xi_\circ}^\dag$ as $\Utor_{r,0}$-modules. In particular, taking into account the action of $\heis^v$ in $\cF_{\xi_\circ}^\dag$ given by Corollary~\ref{C:hwF}, Theorem~\ref{T:tsymbaliuk} implies the following remarkable fact: the horizontal Heisenberg subalgebra $\heis^h=\miki(\heis^v)$ acts diagonally in Saito's vertex representation $\cV$, with one-dimensional joint eigenspaces and eigenvalues given by \eqref{E:h-modes-act}.

\subsubsection{Eigenoperators}

We now consider operators on $\cV$ arising from the action of particular elements of $\heis^h$, namely the following scalar multiples of $\rho_{\cV^\miki}(h_{i,\pm 1})=\rho_{\cV}(\miki(h_{i,\pm 1}))$ (cf. Proposition~\ref{P:miki-formulas}):
\begin{align*}
\tD^{(i)}_1 &= c_i(u) \hD^{(i)}_1\\
\tD^{(i)*}_1 &= c_i^*(u) \hD^{(i)*}_1,
\end{align*}
where
\begin{align*}
\hD^{(i)}_1 &=
\begin{cases}
\rho_{\cV}\left(f_{0,0} \cdot
\begin{bmatrix}
f_{r-1,0} & \dotsm & f_{i+1,0} & f_{1,0} & \dotsm & f_{i-1,0} & f_{i,0}\\
p & \dotsm & p & p & \dotsm & p & p^2
\end{bmatrix}\right) & \text{if $i\neq 0$}\\
& \\
\rho_{\cV}\left(f_{1,1}\cdot
\begin{bmatrix} 
f_{2,0} & \dotsm & f_{r-1,0} & f_{0,-1} \\
p & \dotsm & p & p^2 
\end{bmatrix}\right) & \text{if $i=0$}
\end{cases}\\
\hD^{(i)*}_1 &=
\begin{cases}
\rho_{\cV}\left(\begin{bmatrix}
e_{i,0} & e_{i-1,0} & \dotsm & e_{1,0} & e_{i+1,0} & \dotsm & e_{r-1,0} \\
p^{-2} & p^{-1} & \dotsm & p^{-1} & p^{-1} & \dotsm & p^{-1}
\end{bmatrix}\cdot e_{0,0}\right) & \text{if $i\neq 0$}\\
& \\
\rho_{\cV}\left(\begin{bmatrix}
e_{0,1} & e_{r-1,0} & \dotsm & e_{2,0} \\
p^{-2} & p^{-1} & \dotsm & p^{-1}
\end{bmatrix}
\cdot e_{1,-1}\right) & \text{if $i=0$}
\end{cases}
\end{align*}
and
\begin{align*}
	c_i(u) &= \begin{cases}
		(-1)^{r-i-1} u^{r-2i}&\text{if $i\neq 0$} \\
		u^{2-r} & \text{if $i=0$}
	\end{cases}\\
	c_{i}^*(u) &=
	\begin{cases}
		(-1)^{r-i-1} u^{2i-r}&\text{if $i\neq 0$} \\
		u^{r-2} & \text{if $i=0$,}
	\end{cases}
\end{align*}
and where we take $0\le i \le r-1$ throughout. Inspecting the eigenvalues given by \eqref{E:h-modes-act}, one sees that joint eigenspaces of $\{D^{(i)}_1\}_{i\in I}$ (or $\{D^{(i)*}_1\}_{i\in I}$) alone are one-dimensional, i.e., these sets of operators are sufficient to uniquely characterize the joint eigenfunctions of $\heis^h$ in $\cV$.

\begin{ex}\label{X:D}
For $r=3$, we have 
\begin{align*}
D^{(1)*}_1 &= -u^{-1}\rho_{\cV}([e_{1,0},[e_{2,0},e_{0,0}]_{p^{-1}}]_{p^{-2}}).
\end{align*}
The iterated commutator expands into
\begin{align*}
e_{1,0}e_{2,0}e_{0,0}-p^{-1}e_{1,0}e_{0,0}e_{2,0}-p^{-2}e_{2,0}e_{0,0}e_{1,0}+p^{-3}e_{0,0}e_{2,0}e_{1,0}
\end{align*}
and its summands act, for instance, as follows in $\cV$:
\begin{align*}
&e_{1,0}e_{2,0}e_{0,0}\cdot (f\otimes e^{\al})\\
&= u^{m_{1,2,0}(\al)}(E^{(1)}_{1+\pair{\alb_1^\vee}{\alb_2+\alb_0+\al}}E^{(2)}_{1+\pair{\alb_2^\vee}{\alb_0+\al}}E^{(0)}_{1+\pair{\alb_0^\vee}{\al}}\cdot f)\otimes e^{\alb_1}e^{\alb_2}e^{\alb_0}e^{\al},
\end{align*}
where $m_{1,2,0}(\al) = \pair{\alb_{1}^\vee}{M_{1}(\alb_{2}+\alb_{0}+\al)}+\pair{\alb_{2}^\vee}{M_{2}(\alb_{0}+\al)}+\pair{\alb_{0}^\vee}{M_{0}\al}$ and the product in $\F\{Q\}$ simplifies as $e^{\alb_1}e^{\alb_2}e^{\alb_0}e^{\al}=-e^{\al}$ since $\sg(\alpha_2,\alpha_0)=-1$.

\end{ex}

\comment{
Define $\ad^L_a(f)\cdot g = [f,g]_a = fg-agf$
and $\ad^R_a(f)\cdot g = [g,f]_a$.

\begin{align}
\hD^{(i)}_1 &= \begin{cases} 
\ad^L_{s^{-4}}(e_{i,0})\ad^L_{s^{-2}}(e_{i-1,0})\dotsm \ad^L_{s^{-2}}(e_{1,0})\ad^L_{s^{-2}}(e_{i+1,0})\dotsm \ad^L_{s^{-2}}(e_{r-1,0})\cdot e_{0,0}&\text{for $i\ne0$} \\
\ad^L_{s^{-4}}(e_{0,1})\ad^L_{s^{-2}}(e_{r-1,0})\dotsm\ad^L_{s^{-2}}(e_{2,0})\cdot e_{1,-1} & \text{if $i=0$ and $r\ge2$} \\
	e_{0,-1} & \text{if $i=0$ and $r=1$}
\end{cases}
\end{align}
and
\begin{align}\label{E:cdef}
	c_{i}(u)&=
	\begin{cases}
		(-1)^{r-i-1} u^{2i-r}&\text{for $i\ne0$} \\
		u^{r-2} & \text{for $i=0$ and $r\ge2$} \\
		1 & \text{for $i=0$ and $r=1$}
	\end{cases}
\end{align}
}

\subsubsection{Eigenfunctions are wreath Macdonald polynomials}
By analogy with Theorem~\ref{T:FT-SV}, it is natural to hope that the joint eigenfunctions of $\heis^h$ in $\cV$ are related to wreath Macdonald polynomials. Wen \cite{Wen} proved that this is indeed the case, giving a spectacular extension of the Macdonald eigenoperators from Corollary~\ref{C:tH-eigen} to the wreath setting.



For any $\mu \in \Y$, define a vector $\tH_\mu\in\cV$ by
\begin{align*}
 \tH_\mu = \tH^{t_{-\beta^\vee}}_{\mud} \otimes e^{\beta}
\end{align*}
where $\mu = \tau(\mud,\beta)$.

\begin{ex}\label{X:eigen}
Let $r=3$ and $\mu=(3,3,2,2)$. We have $\kb(\mu)=\beta=-\alb_2$, 
$\quot_3(\mu)=(\yaa,\cd,\cd)$. So $\tH_\mu = \tH^{t_{\alb_2}}_{(\yaa,\cd,\cd)} \otimes e^{-\alb_2}$.
\end{ex}

\begin{thm}[{\cite{Wen}}]\label{T:eigen} 
The vectors $\{\tH_\mu\}_{\mu\in\Y}$ form a joint eigenbasis for the horizontal Heisenberg algebra acting in $\cV$. In particular, we have $\cV = \bigoplus_{\mu\in\Y} \F \tH_\mu$ and
\begin{align}
  \label{E:Di-eigen}
 \tD_1^{(i)}\cdot \tH_\mu &= A_\mu^{(i)}(q^{-1},t^{-1}) \tH_\mu\\
  \label{E:Di*-eigen}
  \tD_1^{(i)*}\cdot \tH_\mu &= A_\mu^{(i)}(q,t) \tH_\mu
\end{align}
for all $i\in I$ and $\mu\in\Y$.
\end{thm}

\begin{rem}\label{R:scalars}
Just as in \eqref{E:eigenspaces}, Theorem~\ref{T:eigen} implies that the isomorphism $\cF_{\xi_\circ}^\dag\cong \cV^\miki$ maps the joint eigenspace $\K\ket{\mu}$ isomorphically to $\K\tH_\mu$. However, it does not send $\ket{\mu}$ to $\tH_\mu$ in general. The isomorphism $\cF_{\xi_\circ}^\dag\cong \cV^\miki$ is unique up to a single nonzero scalar, while the images of $\{\ket{\mu}\}_{\mu\in\Y}$ differ from $\{\tH_\mu\}_{\mu\in \Y}$ by a nonconstant function of $\mu$ which appears to be quite subtle.
\end{rem}

\begin{rem}\label{R:2}
We have made the assumption $r\ge 3$ throughout this section for simplicity of exposition and in order to match the existing literature. However, we have verified by computation that \eqref{E:Di*-eigen} and \eqref{E:Di-eigen} from Wen's Theorem continue to hold as stated in the case $r=2$. In the $r=1$ case, one can make sense of the above formulas by taking $Q=0$ and $\F\{Q\}=\F e^0$, and then the images of $e_{0,-1}$ and $f_{0,-1}$ under $\rho_{\cV}$ are exactly the operators $\tD_0^*$ and $\tD_0^*$ from Corollary~\ref{C:tH-eigen}.
\end{rem}

\subsection{Further implications and related results}

The toolkit for working with wreath Macdonald polynomials provided by Theorem~\ref{T:eigen} goes beyond the eigenoperators of \eqref{E:Di-eigen} and \eqref{E:Di*-eigen}.

\subsubsection{Pieri formulas}
The vertical Heisenberg algebra $\heis^v$ acts in $\cV$ by symmetric function multiplication and skewing operators. Thus, in light of the isomorphism  $\cF_{\xi_\circ}^\dag\cong \cV^\miki$, one may use the action of $\miki^{-1}(\heis^v)$ in $\cF_{\xi^\circ}^\dag$ to obtain Pieri formulas for the $\tH_\mu$. However, a major technical challenge is that, while the isomorphism $\cF_{\xi_\circ}^\dag\cong \cV^\miki$ identifies the eigenspaces $\K\ket{\mu}\cong\K\tH_\mu$, the explicit scalars relating these two bases are not known (see Remark~\ref{R:scalars}). For more detail and recent progress in this direction, see \cite{Wen2}.

\subsubsection{Wreath Macdonald difference operators}
One may consider the wreath Macdonald polynomials specialized to finitely many variables. In \cite{OS,OSW}, we use the shuffle algebra realization of $\Utor_r$~\cite{Neg} to construct $q$-difference operators acting diagonally on such a specialization of the wreath Macdonald polynomials $P^{t_{\beta^\vee}}_{\mud}$. These operators directly generalize the Macdonald difference operator $M_N$ from \eqref{E:Mac op}. A key point is that the shuffle algebra provides access to the higher degree Miki images $\miki(h_{i,\pm l})$ and their action in $\cV$ \cite{T,Wen}.

\subsubsection{Dual eigenoperators}
As explained to us by J. Wen, one can show that the twist of $\cV$ by $\miki^{-1}$ is a \textit{lowest} weight representation of $\Utor_r$ isomorphic to $\cF_\xi$ for a particular $\xi$. From this point of view, one obtains eigenoperators for the functions $\swap\,\inv\,\omega \tH^{t_{-\beta^\vee}}_\mud$ by considering the action of the Heisenberg algebra $\miki^{-1}(\heis^v)$ (which is different from $\heis^h$) in $\cV$.

\end{document}